\documentclass[final,leqno]{siamltex1213}
\usepackage{fancyhdr}
\usepackage{paralist}
\usepackage{xspace}
\usepackage{hyperref}
\usepackage[bbgreekl]{mathbbol}
\usepackage{verbatim}

\usepackage{listings}
\usepackage{amssymb,amsmath,amsfonts}
\newcommand*{\qed}{\hfill\ensuremath{}}%

\newtheorem{prop}[theorem]{Proposition}
\newtheorem{dftn}[theorem]{Definition}
\newtheorem{aspt}[theorem]{Assumption}

\newtheorem{thm}{Theorem}
\newtheorem{lem}[thm]{Lemma}
\newcommand{\R}{\mathbb{R}}

\usepackage{graphicx}
\usepackage{epstopdf}
\usepackage{subcaption}
\usepackage{algorithm}
\usepackage{algcompatible}
\usepackage{float}

\usepackage{xcolor}
\newcommand{\Blue}[1]{{\color{blue}#1}}
\newcommand{\Red}[1]{{\color{red}#1}}
\newcommand{\Magenta}[1]{{\color{magenta}#1}}   

\title{TRPL+K: Thick-Restart Preconditioned Lanczos+K Method for Large Symmetric Eigenvalue Problems}
\author{Lingfei Wu\footnotemark[1],
   Fei Xue\footnotemark[2],
   Andreas Stathopoulos\footnotemark[1]}

\begin{document}
\maketitle
\renewcommand{\thefootnote}{\fnsymbol{footnote}}
\footnotetext[1]{Department of Computer Science, College of William and Mary, Williamsburg, Virginia 23187-8795, 
U.S.A. (lwu@email.wm.edu, andreas@cs.wm.edu)} 
\footnotetext[2]{Department of Mathematical Sciences, Clemson University, Clemson, SC 29634, 
U.S.A. (fxue@clemson.edu)} 
\renewcommand{\thefootnote}{\arabic{footnote}}

\begin{abstract}
The Lanczos method is one of the standard approaches for computing a few eigenpairs of a large, sparse, symmetric matrix. It is typically used with restarting to avoid unbounded growth of memory and computational requirements. Thick-restart Lanczos is a popular restarted variant because of its simplicity and numerically robustness. However, convergence can be slow for highly clustered eigenvalues so more effective restarting techniques and the use of preconditioning is needed. 
In this paper, we present a thick-restart preconditioned Lanczos method, TRPL+K, that combines the power of locally optimal restarting (+K) and preconditioning techniques with the efficiency of the thick-restart Lanczos method. TRPL+K employs an inner-outer scheme where the inner loop applies Lanczos on a preconditioned operator while the outer loop augments the resulting Lanczos subspace with certain vectors from the previous restart cycle before it restarts. We first identify the differences from various relevant methods in the literature. Then, based on an optimization perspective, we show an asymptotic global quasi-optimality of a simplified TRPL+K method compared to an unrestarted global optimal method. The theory applies also to LOBPCG as a special case. Finally, we present extensive experiments showing that TRPL+K either outperforms or matches other state-of-the-art eigenmethods in both matrix-vector multiplications and computational time.
\end{abstract}

%%%%%%%%%%%%%%%%%%%%%%%%%%%%%%%%%%%%%%%%%%%%%%%%%%%%%%%%%%%%%%%%%%%
\section{Introduction}
The numerical solution of large sparse symmetric eigenvalue problems is one of the most computationally intensive tasks in applications ranging from structural engineering, quantum chromodynamics, material science, dynamical systems, machine learning and data mining, to numerical linear algebra \cite{Golub1996MC, saad1992numerical, van2002computational, wu2016estimating, wu2016towards, wu2015preconditioned,wu2016primme_svds, hsieh2014nuclear, Han2011Data}. Depending on the application, one may be interested in computing a few of the smallest or largest eigenpairs, or some eigenpairs in the interior of the spectrum. The challenge is that the size of the eigenproblems in these applications is routinely $O(10^7-10^9)$ \cite{olsen1990passing}. Due to memory and computational constraints, iterative methods that rely on sparse matrix-vector products are the only practical solutions.

We are interested in finding the smallest eigenvalues and associated eigenvectors of the pencil $(A,B)$ when $A, B \in \R^{n \times n}$ are large, sparse symmetric matrices,
\begin{equation}
A v_i = \lambda_i B v_i, \quad \quad i = 1, \ldots, p, \ \ p \ll n ,
\end{equation} 
where $B$ is positive definite, $[v_{1},\ldots , v_{p}] \in \R^{n \times p}$ is a $B$-orthonormal set of eigenvectors and  $\lambda_{1}, \ldots, \lambda_{p}$ are the corresponding eigenvalues of $(A,B)$.
For simplicity, we first discuss our method in the context of the standard eigenvalue problem (where $B$ is the identity matrix) but we extend it to the generalized problem later.

Variants of Krylov subspace methods have long been used to address large-scale eigenvalue problems \cite{lanczos1950iteration,paige1972computational,paige1976error,paige1980accuracy,kalantzis2016spectral}. The unrestarted versions of Lanczos and Arnoldi are considered optimal methods because they obtain their solutions over the entire set of matrix polynomials with degree up to the number of iterations. Yet, for difficult problems they require too many iterations to converge, resulting in impractical memory and computational demands.
Researchers have sought to alleviate these problems with restarting and preconditioning.

Restarting interrupts the iteration, computes the current approximations and uses them to start a new iteration. The pioneering implicitly restarted Arnoldi and Lanczos (IRL) methods perform this in a way so that the restarted vectors continue to form a Krylov subspace \cite{sorensen1992implicit,calvetti1994implicitly}.
The thick-restart Lanczos (TRLan) method \cite{morgan1996restarting,wu2000thick} is equivalent to IRL but simpler to use and numerically robust. 
However, when the desired eigenvalues are poorly separated from the rest of the spectrum, restarting causes further deterioration of convergence that thick-restarting cannot fully recover. Locally optimal restarting is a technique that can result in near optimal convergence when combined with thick-restarting or block methods \cite{knyazev2001toward,stathopoulos1998restarting,stathopoulos2007nearlyI}. The technique is also known as +K restarting which is more descriptive of the number of locally optimal restarting directions we keep. 
The resulting methods, however, do not form Krylov spaces and cannot use efficient constructing strategies such as the Lanczos three-term recurrence.

Near optimal restarting techniques alone cannot address the slow convergence caused by poorly separated eigenvalues, a fact that often appears in real applications \cite{yang2005solving,vomel2008state}. Shift-invert transformations require matrix inversions and are typically expensive, which we do not consider in this paper. Instead we focus on methods that use \emph{inexact} inverses, or \emph{preconditioning}, to accelerate convergence. 
The Davidson and its extension the Generalized Davidson (GD) method are prototypical preconditioned methods \cite{davidson1975iterative,morgan1986generalizations,crouzeix1994davidson,stathopoulos1994davidson}. 
At every step they apply the preconditioner to the current eigenvalue residual to extend the search space in a similar fashion to Arnoldi but producing a non-Krylov space. The Jacobi-Davidson method (JD) \cite{sleijpen2000jacobi} is a special case of GD where the preconditioner is performed with an appropriately projected preconditioned linear system solver. GD can also use thick and locally optimal restarting, a method called GD+K \cite{stathopoulos1998restarting}.
Note that this type of eigenvalue preconditioning can be considered as a step of an optimization method \cite{EdelmanAriasSmith}. Such a view is followed by the Locally Optimal Block Preconditioned Conjugate Gradient method (LOBPCG) \cite{knyazev2001toward} which forgoes the subspace acceleration of GD+K for a block three term recurrence. 
The search spaces of the above methods are not Krylov, which results in two disadvantages: expensive iteration costs (Rayleigh-Ritz projection at each inner step) and selective convergence to a particular eigenpair. 
The Preconditioned Lanczos (PL) \cite{morgan1993preconditioning} and the inverse free preconditioned Krylov subspace method (EigIFP) \cite{golub2002inverse} build a Krylov space of the preconditioned matrix and thus avoid the expensive iteration costs. Although the preconditioned matrix has different eigenvectors than $(A,B)$, the methods invoke the Rayleigh-Ritz projection of the original matrices onto a preconditioned search subspace, and can converge to one eigenpair at a time.

In this paper, we propose a Thick-Restart Preconditioned Lanczos+K method (TRPL+K) to address the aforementioned problems. TRPL+K includes all three major building blocks: thick-restarting, locally optimal restarting, and preconditioning. Unlike GD+K, however, it employs a Krylov inner iteration based on TRLan to build the search space, thus avoiding the expensive Rayleigh-Ritz procedure at every step and requiring about half the memory.
It also differs from JD, since it stores the entire Krylov space. Alternatively, TRPL+K can be considered as an extension to EigIFP (or PL) with thick and locally optimal restarting, thus offering significant speedups.
We provide a convergence analysis of a simplified version of TRPL+K from the perspective of optimization, showing an asymptotic global quasi-optimality of the method compared to an ideal unrestarted global optimal method. This complements some limited earlier theoretical results on the convergence of the +K technique in LOBPCG and GD+K \cite{knyazev2001toward,stathopoulos1998restarting}.
Extensive experiments demonstrate that TRPL+K either outperforms or matches other state-of-the-art eigenmethods in both matrix-vector multiplication counts and computational time. 

In Section \ref{sec:related works} we develop a background framework on which to compare existing and proposed methods. In Section \ref{sec:trpl+k for standard eigenvalue problem}, we present our method for the standard eigenproblem. Section \ref{sec:convergence analysis} considers a simplified version of the algorithm for the generalized eigenproblem and develops a convergence analysis. Section \ref{sec:numerical experiments} compares the efficiency and effectiveness of TRPL+K with other methods through experiments.

%%%%%%%%%%%%%%%%%%%%%%%%%%%%%%%%%%%%%%%%%%%%%%%%%%%%%%%%%%%%%%%%%%%
\section{Background and Related Work}
\label{sec:related works}
Since restarting techniques are the primary concern of this paper, we focus our background section on a common framework in which we can describe most current eigenmethods as well as the proposed one.

\subsection{Thick Restarting, Locally Optimal Restarting, and Preconditioning}
 The unrestarted Lanzcos method converges optimally in terms of the number of matrix-vector multiplications because it dynamically builds the optimal polynomial through an efficient three-term recurrence. In practice, rounding errors cause loss of orthogonality to previous Lanczos vectors, so we typically store all the Lanczos vectors and perform selective or partial re-orthogonalization \cite{parlett1980symmetric,Golub1996MC}. Restarting is intended to reduce the storage requirements and computational cost of orthogonalization. After a maximum number of iteration vectors are stored, we compute the best desired approximations and restart. While limiting storage and computational costs per iteration, restarting inevitably impairs the optimality of unrestarted Lanczos since it discards part of the information. Various techniques \cite{murray1992improved,stathopoulos1998restarting,stathopoulos2007nearlyI,knyazev2001toward,liu2013limited} attempt to partially recover the lost information due to restarting. 
 
 Implicit restarting \cite{sorensen1992implicit} performs this by dumping unwanted components (typically unwanted Ritz vectors) by applying the implicitly shifted QR. However, this technique is complicated to implement in a stable way \cite{hernandez2006lanczos}.
Thick restarting is mathematically equivalent to implicit restarting \cite{stathopoulos1998dynamic,morgan1996restarting} yet it is easier to implement in a stable way in the Lanczos \cite{wu2000thick}, Arnoldi \cite{morgan1996restarting}, and GD methods \cite{stathopoulos1998dynamic}. Thick restarting directly keeps the wanted Ritz vectors instead of dumping the unwanted ones from the basis. Moreover, with thick restarting it is straight-forward to add arbitrary (non-Krylov) vectors to the restarted space. This is a key feature for our proposed method where we augment the restarted space with Ritz vectors from a previous cycle. Due to simplicity, numerically stability, and flexibility, thick restarting has been applied to various Krylov and GD (or JD) type methods for both eigenvalue and SVD problems \cite{wu2000thick,morgan1996restarting,EV_CY_FX_16,li2016thick,sleijpen2000jacobi,stathopoulos1998dynamic,stathopoulos2007nearlyI,baglama2005augmented,wu2015preconditioned}. 

The locally optimal restarting technique has been studied under different names in the literature such as locally optimal recurrence in LOBPCG \cite{knyazev2001toward}, +K restarting in GD+K \cite{stathopoulos1998restarting,stathopoulos2007nearlyI}, and Krylov subspace optimization \cite{liu2013limited}. There are different ways to justify the use of this technique. One is from an optimization viewpoint that extends the non-linear Conjugate Gradient (CG) method for optimizing the Rayleigh quotient by performing a Rayleigh Ritz on the three successive iterates. Another viewpoint is that three successive Lanczos iterates are sufficient to guarantee full orthogonality of the space. Yet another viewpoint relates the Lanczos iterates to the ones from a three term recurrence of the linear CG on the Jacobi-Davidson correction equation. Regardless of the viewpoint, the idea of using Ritz vectors from both the current and the previous iteration has given rise to methods that converge near-optimally for seeking one eigenpair under limited memory, especially when combined with block methods or thick restarting \cite{knyazev2001toward,stathopoulos2007nearlyI}. Rigorous analysis, however, has been difficult. In \cite{stathopoulos1998restarting} the last viewpoint was analyzed providing some bounds on how well the locally optimal restarting matches the effects of a global optimization over the unrestarted space. 
In this paper, we provide an optimization viewpoint analysis that establishes the asymptotic global quasi-optimality of our new method by quantifying the relative difference between the locally and a globally optimal Rayleigh quotients.

%Restarting techniques try to retrieve some of the optimality of the unrestarted method. 
Preconditioning (inexact shift-invert) is crucial to improve the convergence of all these methods. For large problems, as exact matrix factorizations are prohibitive or infeasible, we focus on preconditioners such as incomplete ILU or LDL factorization \cite{golub2002inverse}. Although GD (or JD) type methods use a preconditioner to build a general, non-Krylov subspace, a few methods have been proposed to exploit a preconditioned Krylov subspace  \cite{golub2002inverse,morgan1993preconditioning,EV_CY_FX_16}. This paper further explores this line of research.

\subsection{Comparison of Subspaces of Various Methods}
We use the following notation. 
A cycle refers to all the work a method performs between restarts and is denoted by a subscript. 
When present, a second subscript refers to the eigenvector index, e.g., $x_{i,2}$ is the approximate eigenvector for the second smallest eigenvalue at the end of cycle $i$. A matrix or block of vectors followed by parenthesis uses MATLAB index notation.
At restart, thick restarting keeps at least the $p$ wanted Ritz vectors. A method then continues building a basis $U$ with $m$ new vectors, which differentiates most methods. In +K restarting, $U$ is then augmented by $l$ Ritz vectors from the previous step (which previous step depends on the method). Thus, at the end of the $i$-th outer cycle, the basis is $U \in \R^{n \times q}$, where $q = p+m+l$ is the maximum basis size. For a given shift $\rho$, let $A_\rho = (A - \rho I)$, $\widehat{A}_\rho = L^{-1} (A - \rho I) L^{-T}$, and a preconditioner $M = L^{-T} L^{-1} \approx (A - \rho I)^{-1}$. 
The shift is usually the Rayleigh quotient for some approximate eigenvector $x$,
$\rho(x) = \frac{x^T A x}{x^T x}$.
We use the Rayleigh-Ritz (RR) procedure to extract approximate eigenpairs from $\mathrm{span}(U)$. $\mathcal{K}_{m}(A,u_1)$ refers to the Krylov subspace of dimension $m$ of $A$ with initial vector $u_1$.

Thick-restart Lanczos \cite{wu2000thick} is mathematically equivalent to implicit restarted Lanczos. At the end of the $i+1$ cycle, it forms a space $K_{TRL}$ which includes the $p$ wanted Ritz vectors obtained by RR at the end of the $i$-th cycle, and $m$ additional Lanczos vectors starting from the last Lanczos vector $r = u_{i,p+m+1}$ of the $i$-th cycle. 
\begin{equation}\label{eq:subspace_trlan}
    K_{TRL} = \text{span} \{\ \underbrace{x_{i,1}, x_{i,2}, \ldots, x_{i,p},}_{\text{Wanted Ritz vectors}} \ \underbrace{r, Ar, A^2r, \ldots, A^{m-1} r}_{\text{Lanczos iterations}} \ \}.
\end{equation}
$K_{TRL}$ always remains a Krylov space, but it allows for a more efficient implementation than the implicitly restarted Lanczos. Thick-restart Lanczos is the only method we consider that cannot use preconditioning directly, but only through shift-invert.

LOBPCG \cite{knyazev2001toward} forms a subspace $K_L$ at the end of cycle $i+1$ from which it will use RR to compute $p$ approximate eigenpairs. The subspace is built by the following locally optimal recurrence with $m=p=l$,
\begin{equation}\label{eq:subspace_lobpcg}
    K_L = \text{span} \{\ \underbrace{x_{i,1}, x_{i,2}, \ldots, x_{i,p},}_{\text{Wanted Ritz vectors}} \ \underbrace{Mr_1, Mr_2, \ldots, Mr_p,}_{\text{Preconditioned residual vectors}} \ \underbrace{x_{i-1,1}, \ldots, x_{i-1,p}}_{\text{Previous Ritz vectors}}\ \}
\end{equation}
where $x_{i,j}$ is the $j$-th Ritz vector from cycle $i$, $\{r_j\}$ is its residual, and $x_{i-1,j}$ is the corresponding Ritz vector from cycle $i-1$ 
\footnote{One could use search directions instead of previous 
Ritz vectors for better numerical stability as suggested in \cite{knyazev2001toward,knyazev2007block}, but both variants essentially construct the same subspace.}. 
Note that $\{M r_j\}$ are typically computed in a block form, not one at a time through an inner iteration.

GD+K \cite{stathopoulos1998restarting,stathopoulos2007nearlyI} is an extension of GD where, at restart, $l$ number of previous Ritz vectors are used to augment the thick restarted basis. Thus at the end of cycle $i+1$ the subspace of size $q$ is
\begin{equation}\label{eq:subspace_gd+k}
    K_G = \text{span} \{\ \underbrace{x_{i,1}, \ldots, x_{i,p},}_{\text{Wanted Ritz vectors}} \ \underbrace{Mr_1^{(p)}, \ldots, Mr_1^{(p+m-1)},}_{\text{Preconditioned residual vectors}} \ \underbrace{x_{i,1}^{(q-1)}, \ldots, x_{i,l}^{(q-1)}}_{\text{Previous Ritz vectors}}\ \}. 
\end{equation}
Here, $x_{i,j}$ and $x_{i,j}^{(q-1)}$ are the Ritz vectors computed at cycle $i$ from the spaces $K_G$ with $size(K_G)=q$ and $size(K_G)=q-1$, respectively. 
$r^{(j)}$ denotes the residual vector of the targeted Ritz vector $x_{i+1,1}^{(j)}$
at inner iteration $j-p+1$ of the current cycle $i+1$, i.e., when $size(K_G)=j$.
Without preconditioning ($M=I$) and with $l=0$, GD+K is equivalent to thick-restart Lanczos. 
%When $l>0$, $K_G$ is not a Krylov space. 
When $m=p=l=1$, GD+K is equivalent to the locally optimal preconditioned conjugate gradient (LOPCG, or LOBPCG with $p=1$).
A block version of GD+K is also possible. 

Preconditioned Lanczos (PL) \cite{morgan1993preconditioning} employs a preconditioned Krylov inner iteration on $\widehat{A}_\rho=L^{-1}(A-\rho I)L^{-T}$ to build a basis $G$ of $K_P$ in (\ref{eq:subspace_pl}), applies the RR of $\widehat{A}_\rho$ onto $G$ to find a primitive Ritz pair $(\theta,y)$, and converts the Ritz pair back $(\rho+\theta,L^{-T}Gy)$ for the original eigenvalue problem. Additional eigenpairs are found one at a time. The size $m$ of the Krylov space varies dynamically.
\begin{equation}\label{eq:subspace_pl}
     K_P = \text{span} \{ \underbrace{L^{T}x_{i,1},}_{\text{Wanted Ritz vector}} \underbrace{\widehat{A}_\rho(L^{T}x_{i,1}), \widehat{A}_\rho^2(L^{T}x_{i,1}), \ldots, \widehat{A}_\rho^{m}(L^{T}x_{i,1})}_{\text{Lanczos iterations}} \ \}.
\end{equation}
Without preconditioning PL is an explicitly restarted Lanczos with one vector (i.e., no thick restarting). 

The inverse free preconditioned Krylov subspace method (EigIFP) \cite{golub2002inverse} produces the same approximations as PL in exact arithmetic, but the application of the preconditioner does not have to be in factorized form. The basis $V$ of the search space $K_F$ built by EigIFP is related to the PL's basis as $U = L^{-T} G$. 
RR is performed by projecting $A_\rho = A-\rho I$ onto $U$, yielding the approximate eigenpair $(\rho+\theta,U y)$. Since all vectors are stored, $m = q-1.$ Otherwise, EigIFP has the same limitations as PL. 
\begin{equation}\label{eq:subspace_eigifp}
    K_F = \text{span} \{ \underbrace{x_{i,1},}_{\text{Wanted Ritz vector}} \underbrace{MA_\rho x_{i,1}, (MA_\rho )^2x_{i,1}, \ldots, (MA_\rho )^m x_{i,1}}_{\text{Lanczos iterations}} \ \}.
\end{equation}

The GD+K method typically demonstrates faster convergence than the rest of the methods in terms of number of matrix-vector products because it combines both thick and locally optimal restarting and uses subspace acceleration to obtain the ``best'' Ritz vector at every inner iteration to improve by preconditioning. However, the faster convergence of GD+K is at the cost of applying more frequent Rayleigh-Ritz (RR) procedure, which could be a quite expensive operation when the subspace is large \cite{vecharynski2015projected}. Therefore, it is unclear if the GD+K method is still the method of the choice in terms of the runtime when the matrix-vector operation is inexpensive. 
For large numbers of well separated eigenvalues, however, a block method such as LOBPCG, could also be competitive.
On the other hand, PL and EigIFP can generate the inner Krylov space without the overhead of multiple RR projections required by GD+K.

%%%%%%%%%%%%%%%%%%%%%%%%%%%%%%%%%%%%%%%%%%%%%%%%%%%%%%%%%%%%%%%%%%%
\section{Thick-Restart Preconditioned Lanczos +K Method}
\label{sec:trpl+k for standard eigenvalue problem}
The motivation of our proposed TRPL+K method is to extend the computationally efficient EigIFP method with the thick and locally optimal restarting techniques of GD+K. 
It can also be viewed as an extension of thick-restart Lanczos to allow for locally optimal restarting and preconditioning.
We note that the JD method with CG as inner iteration and GD+K as the outer method is even more computationally efficient per step because the inner Krylov space does not need to be stored. However, the result of CG is a correction to a single Ritz vector that does not benefit convergence to nearby eigenpairs, which is left to the subspace acceleration of the outer iteration. 
With TRPL+K we hope that the inner iteration generates useful correction information to all $p$ required eigenvectors. For simplicity, we describe the method in detail for the standard eigenvalue problem only.

Extending EigIFP to include thick restarting is rather straightforward. Without preconditioning, the $K_F$ space in (\ref{eq:subspace_eigifp}) is a Krylov space so it can be restarted as in thick-restart Lanczos (\ref{eq:subspace_trlan}). With preconditioning, $K_F$ is still a Krylov space but of the preconditioned matrix $M(A-\rho I)$. Note that if $\rho$ differs between cycles so do the matrices of the corresponding Krylov subspaces. At the end of cycle $i+1$, the RR must be performed on matrix $A$ or $A_\rho= A-\rho I$ to compute the new Ritz vectors 
$X_{i+1} = [x_{i+1,1}, \ldots, x_{i+1,p}]$. 
This implies that the thick restart vectors $X_i$ used in the basis of cycle $i+1$ do not form vectors of a Krylov sequence.
After restart, we have $T = X_i^T A X_i = \text{diag}([\theta_{i,1},\ldots ,\theta_{i,p}])$, but we cannot use the TRLan relations for the subsequent Krylov vectors.

To efficiently build such an augmented Krylov space we can use a technique based on an FGMRES like method \cite{chapman1997deflated} or on the GCRO method \cite{Sturler_GCR,Sturler_trunc}.
We have followed a GCRO like method to build a Krylov basis $G_{i+1}$ orthogonal to $X_{i} = [x_{i,1}, \ldots, x_{i,p}]$, i.e., we build 
$\text{span}(G_{i+1}) = \mathcal{K}_m\left((I-X_iX_i^T)M(A-\rho I), (I-X_iX_i^T)Mr_1\right)$, with $r_1 = (A-\rho I)x_{i,1}$. 
This allows us to build the projection matrix $T = [X_i,G_{i+1}]^T A [X_i, G_{i+1}]$ without additional matrix-vector products.
At step $j$ of the inner Krylov method, before we compute $G_{i+1}(:,j+1)$ we compute
\begin{equation}
\label{eq:UpdateT}
\begin{aligned}
    z = & \ A G_{i+1}(:,j) \\
    T(1\!:\!p, p+j) = & \ X_i^T z \\
    T(p+1\!:\!p+j, p+j) = & \ G_{i+1}(:,1\!:\!j)^Tz.
\end{aligned}
\end{equation}
Then, we continue with the inner method
$G_{i+1}(:,j+1) = (I-X_iX_i^T)M(z-\rho x_{i,1})$.

Employing locally optimal restarting to the above thick restarted preconditioned Lanczos (TRPL) is more involved. 
At cycle $i+1$, after the inner method concludes its $m$ steps, we perform RR to obtain $X_{i+1}$ using the space 
$U_{i+1}(:,1\!\!:\!m+p) = [X_i, G_{i+1}]$.
We want to augment this space with some `previous' directions, $X^{prev}$.
A choice similar to GD+K does not work. GD+K uses as $X^{prev}$ the Ritz vectors from the penultimate step of the Krylov method before restart,
i.e., the Ritz vectors from the subspace
$U_i(:,1\!:\!p+m-1) = [X_{i-1}, G_i(:,1\!:\!m-1)]$.
The idea is that the optimal projection over $U_i(:,1\!:\!p+m+1)$ (i.e., the next step of the unrestarted method) can be approximated through the Ritz vectors of the last two iterations and the residual. But our method does not optimize over the $X^{prev}$ directions to expand the basis at cycle $i+1$; it builds a Krylov space.

 If we assume that our inner method was a polynomial returning only one vector $s(A)r_1$ (not the entire $G_{i+1}$ space of size $m$) to be used in the outer RR and that we were seeking $p=1$ eigenpair, the outer method would be similar to LOBPCG on the operator $s(A)$. For this operator, the choice for locally optimal restarting would be $X^{prev}=X_{i-1}$, i.e., the Ritz vector at the beginning of the previous cycle. We take the liberty to use LOBPCG's choice $X^{prev}$ for our method, which we now call TRPL+K. The space at the end of cycle $i+1$ should include $X_i, G_{i+1}$, and $X_{i-1}$. 

The order of the three blocks in the search subspace deserves a careful discussion. Since $X_i$ is produced from a basis that includes $X_{i-1}$ at cycle $i$, it is possible to orthogonalize the two sets implicitly and compute their interaction on the projection matrix $T_{i+1}$ without matrix-vector products. 
Then, by augmenting our space not only with $X_i$ but with span($[X_i, X_{i-1}]$), we could build $G_{i+1}$ as described above. However, we noticed experimentally that this choice does not perform well. 

We observed much better convergence if the previous vectors $X_{i-1}$ were included in the basis after $G_{i+1}$ was computed. The disadvantage is that $ X_{i-1}$ has to be explicitly orthogonalized against the rest of the basis vectors and $l$ matrix-vector products have to be performed to compute the resulting $T$. However, the relative expense is small for large $m$ (inner iterations), and typically $l=1$ gives very close to optimal convergence while more previous vectors give little additional improvement. 

We can now describe the space that TRPL+K builds. 
Let $X_i = [x_{i,1},\ldots,x_{i,p}]$, 
    $C = (I-X_iX_i^T)M(A-\rho_i I)$,
and $u_1 = C x_{i,1}$. 
Then, at the end of cycle $i+1$ TRPL+K computes the Ritz vectors from the subspace 
\begin{equation}\label{eq:subspace_trpl+k}
    K_{TRPL+K} = \text{span} \{ \underbrace{x_{i,1},\ldots,x_{i,p}}_{\text{Wanted Ritz vector}} \underbrace{u_1, C u_1, \ldots, C^{m-1} u_1}_{\text{Lanczos iterations}} \
    \underbrace{x_{i-1,1},\ldots,x_{i-1,l}}_{\text{Previous Ritz vector}}\ \}.
\end{equation}
Algorithm \ref{alg:trpl+k} summarizes TRPL+K for finding $p$ smallest eigenpairs of $A$. For simplicity we also let the minimum thick restart size be $p$. However, the algorithm can thick restart with any size $\hat{p}>p$. Note also that the preconditioner $M$ may change at every cycle with the goal to approximate $(A-\rho_i I)^{-1}$. 

Finally, we note that the algorithm can be easily extended to the generalized eigenvalue problem, $Av = \lambda Bv$ with $B$ symmetric positive definite. We forego this description for simplicity. 
The analysis in the next section is based on a simplified version of our algorithm for the generalized eigenvalue problem, aiming to compute only the lowest eigenpair $(\lambda_1,v_1)$ ($p=1$), replacing the previous Ritz vector $x_{i-1,1}$ in \eqref{eq:subspace_trpl+k} with a search direction connecting $x_{i-1,1}$ and $x_{i,1}$ in two consecutive cycles. We show the algorithm with soft locking but hard locking can also be applied. 

\begin{algorithm}[H]
\caption{Thick-Restart Preconditioned Lanczos +K method}\label{alg:trpl+k}
\begin{algorithmic}[1]
    \STATEx {\bf Input:} matrix $A \in \R^{n \times n}$, preconditioner $M \in \R^{n \times n}$,  $p$ the number of desired eigenpairs, $X = [x_1\ldots x_p]$ any available initial approximations to the $p$ eigenpairs, $x^T_i x_j = \delta_{ij}$, maximum basis size $q$, maximum number of retained previous vectors $l$, maximum number of cycles $maxIter$
    \STATEx {\bf Output:} approximate eigenvalues $\theta_1,\ldots,\theta_p$, and  eigenvectors $x_1, \ldots, x_p$.
	\STATE Perform RR on the orthonormal set $U = X =[x_1,\ldots ,x_p]$,  $T(1\!:\!p,1\!:\!p) \leftarrow X^TAX$
	\STATE Let $Y\Theta Y^T= X^TAX$ be the eigen-decomposition of $X^TAX$. Reset $U,X \leftarrow UY$ (the $p$ Ritz vectors) and $T \leftarrow \Theta$. Set target $t=1$, $\rho \leftarrow x_t^T A x_t$, $X^{prev} \leftarrow [\ ]$
	\FOR{$k=1:maxIter$}
		\STATE Use Lanczos (\ref{eq:UpdateT}) to generate $G$ containing the orthonormal basis vectors for \STATEx{\ \ }  $\mathcal{K}_{m}((I-XX^T)M(A - \rho I),u_1)$ and the projection $T(1\!:\!p+m,\ p+1\!:\!p+m)$ 
		\STATE Set $U(:,1\!:\!p+m) \leftarrow [X, G]$	
		\IF{$k > 1$}
			\STATE Orthogonalize $X^{prev}$ against $U(:,1\!:\!p+m)$ to build $U(:,p+m+1\!:\!q)$ 
			\STATE Compute the remaining $T$ such that $T = U^TAU$  
		\ENDIF	
		\STATE $X^{prev} = [x_1^{prev}, \ldots, x_l^{prev}] \leftarrow [x_t, \ldots, x_{t+l-1}]$ \ (up to a $\min(t+l-1,p)$)
		\STATE Compute eigenpairs of $T$, 
		$T = Y \Theta Y^T$
		    and Ritz pairs: ($\theta_i, x_i = U Y(:,i)$)
		\IF{$\|Ax_t - \theta_t x_t\| \leq \epsilon$ }
		\STATE flag $(\theta_t,x_t)$,
		remove $x_1^{prev}$ from $X^{prev}$, and advance target $t=t+1$
        \ENDIF		
		\STATE Set $\rho \leftarrow \theta_t$, $U(:,1\!:\!p+1) \leftarrow [UY(:,1\!:\!p), M(Ax_t - \rho x_t)]$,
		$X=U(:,1\!:\!p)$ 
	\ENDFOR
\end{algorithmic}
\end{algorithm}

%%%%%%%%%%%%%%%%%%%%%%%%%%%%%%%%%%%%%%%%%%%%%%%%%%%%%%%%%%%%%%%%%%%
\section{Asymptotic Convergence Analysis of a PL+1 Method}
%\section{Global quasi-optimality of preconditioned Lanczos+1}
\label{sec:convergence analysis}
In this section, we consider a preconditioned Lanczos+1 method (i.e., TRPL+K with $p=1=l$) for computing the lowest eigenpair $(\lambda_1,v_1)$ of $(A,B)$, establishing an \emph{asymptotic global quasi-optimality} of the method. 

\subsection{Preliminaries} 
Consider the matrix pencil $(A,B)$, where $A, B \in \mathbb{R}^{n\times n}$ are symmetric, typically large and sparse, and $B$ is positive definite. Let $\lambda_1 < \lambda_2 \leq \ldots \leq \lambda_n$ and $\{v_i\}$ be the eigenvalues and eigenvectors of the matrix pencil, such that $Av_i = \lambda_i Bv_i$, $(v_i,v_j)_B = v_i^TBv_j = \delta_{ij}$, $\|v_i\|_B = 1$, $1 \leq i,j \leq n$. 

Let $x$ be an approximation to $v_1$, the eigenvector associated with the lowest eigenvalue $\lambda_1$, with the decomposition 
\begin{eqnarray}\label{xdecomp}
x = v_1 \cos \theta + f \sin\theta, \:\mbox{ where }\: \theta \ne 0,\; f \perp Bv_1, \mbox{ and } \|f\|_B = 1.\vspace{-0.1in}
\end{eqnarray}
This suggests that $\|x\|_B = \left(x^TBx\right)^{\frac{1}{2}} = \left(\|v_1\|_B^2 \cos^2\theta + \|f\|_B^2 \sin^2 \theta \right)^{\frac{1}{2}}=1$. Since $f \perp Bv_1$, it has the form $f = \sum_{j=2}^n s_j v_j$, where the scalars $\{s_j\}_{j=2}^n$ satisfy $\sum_{j=2}^n s_j^2 = \sum_{j=2}^n s_j^2 \|v_j\|_B^2=\|f\|^2_B = 1$. The Rayleigh quotient of $x$ is hence $$\textstyle \rho(x) = \frac{x^TAx}{x^TBx} = x^TAx = v_1^TAv_1 \cos^2 \theta + f^TAf \sin^2\theta = \lambda_1 \cos^2\theta + \rho(f) \sin^2\theta,$$where $\rho(f) = \frac{f^TAf}{f^TBf} = f^TAf=\sum_{j=2}^n s_j^2 \lambda_j \in [\lambda_2,\lambda_n]$. \smallskip

\begin{prop}
Let $x$ be a vector with $\|x\|_B=1$. The gradient and the Hessian of $\frac{1}{2}\rho(x)$ with respect to $x$, respectively, are 
{\footnotesize
\begin{eqnarray}\label{gradHessrho}
\qquad \nabla \frac{1}{2}\rho(x)  &=& \frac{1}{x^TBx}\left(A - \rho(x) B\right)x = Ax -\rho(x) Bx, \qquad\mbox{ and } \\ \nonumber
\qquad \nabla^2 \frac{1}{2}\rho(x)  &=& \frac{1}{x^TBx}\left\{A-\rho(x) B-\frac{2}{x^TBx}(Ax-\rho(x) Bx) (Bx)^T-\frac{2}{x^TBx}Bx\,(Ax-\rho(x) Bx)^T \right\} \\
&=& A-\rho(x) B-2(Ax-\rho(x) Bx) (Bx)^T-2Bx\,(Ax-\rho(x) Bx)^T.
\end{eqnarray}
}
\end{prop}
\begin{proof} Done by letting $T(\rho) = \rho B- A$ in \cite[Proposition 3.1]{Szyld.Xue.2016}.
\end{proof}
\smallskip

For small $\theta$, $\rho(x)$ is a second order approximation to $\lambda_1$, i.e., 
\begin{eqnarray}\label{rqerror}
\rho(x)-\lambda_1 = \sin^2 \theta (\rho(f) - \lambda_1),
\end{eqnarray}
and hence $\nabla \frac{1}{2}\rho(x)$, i.e., the eigenresidual associated with $x$, is
\begin{eqnarray}\nonumber
&&Ax-\rho(x) Bx = \left(A-\rho(x) B\right)(v_1 \cos\theta+f\sin\theta) \\ \nonumber
%&=& \cos\theta \left(A-\rho(x) B\right)v_1+\sin\theta \left(A-\rho(x) B\right)f \\ \nonumber
&=& \left(\lambda_1-\rho(x)\right)\cos\theta Bv_1 + \sin\theta \left(A-\rho(x) B\right)f \\ \nonumber
%&=& \sin\theta\left[\left(Af - \rho(x) Bf\right)- \sin\theta\cos\theta\left(\rho(f)-\lambda_1\right)Bv_1 \right] \\ \nonumber
%&=& - \sin^2\theta \cos\theta(\rho(f)-\lambda_1)Bv_1 + \sin\theta \left(A-\rho(x) B\right)f \\ \nonumber
&=& \textstyle\sin\theta \big[\sum_{j=2}^n s_j(\lambda_j-\rho(x))Bv_j -\sin\theta \cos\theta(\rho(f)-\lambda_1)Bv_1\big].
\end{eqnarray}
Since $\lambda_1 < \lambda_2$ and $\sum_{j=2}^n s_j^2 = 1$, $Af-\rho(x)Bf = \sum_{j=2}^n s_j(\lambda_j-\rho(x))Bv_j$ will not vanish as $\theta \rightarrow 0$ and $\rho(x) \rightarrow \lambda_1$. Therefore, for sufficiently small $\theta$, 
\begin{eqnarray}\label{eigresbd}
\quad (1-\delta) \sin\theta \|Af-\rho(x)Bf\| \leq \|Ax - \rho(x) Bx\| \leq (1+\delta) \sin\theta \|Af-\rho(x)Bf\|
\end{eqnarray}
for some small $\delta > 0$ independent of $\theta$, or simply $
\|Ax-\rho(x)Bx\| = \mathcal{O}(\sin\theta)$.

\subsection{A preconditioned Lanczos+1 method}
The framework of a preconditioned Lanczos+1 method is summarized in Algorithm \ref{alg:pl+1}. This is a special case of Algorithm \ref{alg:trpl+k} for $p=l=1$ but extends it to the generalized eigenvalue problem. 
It can also be considered an extension of the PL method enhanced with the search direction adopted in LOPCG.
It is important to note that when $X$ is a single vector, $[X_k, G_{k+1}]$ in (\ref{eq:subspace_trpl+k}) is a Krylov space and 
$\text{span}\{x_{k,1}, u_1, C u_1, \ldots, C^{m-1} u_1\} = 
\text{span}\{x_{k,1}, Ax_{k,1},\ldots A^m x_{k,1}\}$ (see Lemma 4.1 in \cite{stathopoulos1998restarting}).  
Since we only look for one eigenvalue, in the rest of this section we drop the second subscript notation, and use the subscript $k$ to represent the cycle. Hence the simpler form in Algorithm \ref{alg:pl+1}. 

\begin{algorithm}[H]
\caption{Preconditioned Lanczos+1 method for computing $(\lambda_1,v_1)$ of $(A,B)$}\label{alg:pl+1}
\begin{algorithmic}[1]
\STATE Choose an SPD preconditioner $M \approx (A-\sigma B)^{-1}$ ($\sigma < \lambda_1$), tolerance $\delta > 0$, $m > 0$
\STATE Choose vector $x_0$ with $\|x_0\|_B = 1$, set $\rho_0 = \frac{x_0^TAx_0}{x_0^TBx_0}$ and $r_0 = (A-\rho_0 B)x_0$.
\FOR{$k = 0,\, 1,\, \ldots, $ until the convergence, i.e., $\|r_k\| \leq \delta$}
%5. \quad\quad Orthogonalize $x_k$ against $\{v_i\}_{i=1}^j$ with respect to $[\cdot,\cdot]$ and normalize $x_k$\\
    \STATE Form $G_k$ containing $B$-normalized basis vectors of $M\mathcal{K}_m\left((A-\rho_k B)M,r_k\right)$
    \STATEx{\ \ } $=\mathrm{span}\left\{M(A-\rho_k B)x_k, \left[M(A-\rho_k B)\right]^2 x_k,\ldots, \left[M(A-\rho_k B)\right]^m x_k\right\}$ 
    \STATE Form $Q_k= \left\{\begin{array}{ll}\!\! [x_k,G_k] & \!\!(k=0) \\\!\! \left[x_k,G_k,p_{k-1}\right] & \!\!(k>0)\end{array}\!\!\right.$, perform the RR projection and
    \STATEx{\ \ } solve $Q_k^TA Q_k w = \rho Q_k^TBQ_k w$ for the lowest primitive Ritz pair $\big(\rho_k,w_k\big)$.
    \STATE $g_k \!= G_kw_k(2\!:\!m\!+\!1)$
    \STATEx{\ \ } $\widetilde{x}_{k+1}\!=Q_kw_k \!= \!\left\{\!\begin{array}{ll}\!\! x_kw_k(1)\!+g_k\! & \!\!\!, (\mbox{if } k=0) \\ \!\! x_kw_k(1)\!+g_k\!+p_{k-1}w_k(m\!+\!2) \!\!& \!\!\!, (\mbox{if } k>0)\end{array}\!\!\!\right.$,
    \STATEx{\ \ } $x_{k+1} = \frac{\widetilde{x}_{k+1}}{\|\widetilde{x}_{k+1}\|_B}$, $\rho_{k+1} = \frac{x_{k+1}^TAx_{k+1}}{x_{k+1}^TBx_{k+1}}$, and $r_{k+1} = (A-\rho_{k+1} B)x_{k+1}$.
%$\widetilde{x}_{k+1} = Q_k {w_k}$, $x_{k+1} = \frac{\widetilde{x}_{k+1}}{\|\widetilde{x}_{k+1}\|_B}$, $\rho_{k+1} = \frac{x_{k+1}^TAx_{k+1}}{x_{k+1}^TBx_{k+1}}$, $r_{k+1} = (A-\rho_{k+1} B)x_{k+1}$. \\
\STATE $\widetilde{p}_k = \left\{\begin{array}{ll} \!\!g_k & \!\!\!(k=0) \\ \!\!g_k - \frac{p_{k-1}^T(A-\rho_{k-1}B)g_k}{p_{k-1}^T(A-\rho_{k-1} B)p_{k-1}}p_{k-1} & \!\!\!(k>0) \end{array}\right.$, $p_k = \frac{\widetilde{p}_k}{\|\widetilde{p}_k\|_B}$.
\ENDFOR
\end{algorithmic}
\end{algorithm}

%To understand the superior performance of our new algorithm, we present its simplified variant, namely, a preconditioned Lanczos +1 algorithm, for computing the lowest eigenpair of $(A,B)$, and we investigate its global quasi-optimality. 
In each outer iteration $k$ ($k>0$), at steps 4 and 5, an augmented Krylov subspace
$\mathrm{range}(Q_k) = \mathrm{span}\{x_k\}+M\mathcal{K}_m\left((A-\rho_k B)M,r_k\right)+\mathrm{span}\{p_{k-1}\}$, i.e., 
\begin{eqnarray}\label{spacealg1}
\mathcal{K}_{m+1}\left(M(A-\rho_k B),x_k\right)+\mathrm{span}\{p_{k-1}\}
\end{eqnarray}
is formed as the subspace for the RR projection. Algorithm~\ref{alg:pl+1} is slightly different from the variant used in practice at step 7, where a commonly adopted approach sets $p_{k}=g_k+p_{k-1}w_k(m+2)$, using the $(m+2)$nd element of primitive Ritz vector $w_k$ as the coefficient for $p_{k-1}$.
Our choice of such a particular linear combination makes it easy to show the near conjugacy between $p_{k-1}$ and $p_k$. 

We present a preliminary convergence result for Algorithm~\ref{alg:pl+1}, which is essentially a restatement of \cite[Theorem 3.4]{golub2002inverse}. Here we incorporate preconditioning and note the fact that the space for projection used in \cite{golub2002inverse}, $\mathcal{K}_{m+1}\left(M(A-\rho_k B),x_k\right)$, is a \emph{subspace} of the one constructed by Algorithm~\ref{alg:pl+1} in \eqref{spacealg1}.

\begin{thm}\label{thmcvgrate}
Let $\lambda_1 < \lambda_2 \leq \ldots \leq \lambda_n$ be the eigenvalues of $(A,B)$, $(\rho_{k},x_{k})$ be the $k$-th iterate of Algorithm~\ref{alg:pl+1}, and $LDL^T \approx (A-\rho_k B)$ a preconditioner where $D$ is a diagonal matrix of $\pm1$ elements. Assume that $C_k = L^{-1}(A-\rho_k B)L^{-T}$ has eigenvalues $\sigma_1 < \sigma_2 \leq \ldots \leq \sigma_n$, and satisfies $C_k u_1 = \sigma_1 u_1$ with $\|u_1\|_2=1$. If $\lambda_1 < \rho_k < \lambda_2$, then $
\rho_{k+1}-\lambda_1 \leq (\rho_k-\lambda_1)\epsilon_m^2 + (\rho_k-\lambda_1)^{\frac{3}{2}}\epsilon_m \left(\frac{\|L^{-1}BL^{-T}\|}{\sigma_2}\right)^{\frac{1}{2}}+\delta_k$,
where $0 < \delta_k \equiv \rho_k-\lambda_1+\frac{\sigma_1}{u_1^TL^{-1}BL^{-T}u_1}=\mathcal{O}\left((\rho_k-\lambda_1)^2\right)$, and $\epsilon_m = \min_{p \in \mathcal{P}_m,p(\sigma_1) = 1}\max_{2\leq i \leq n} |p(\sigma_i)|$, with $\mathcal{P}_m$ denoting the set of all polynomials of degree not greater than $m$. 
\end{thm}

{\bf Remark.} The value of $\epsilon_m$ depends on the quality of the preconditioner $M$ approximating $A-\rho_k B$. If $M = A-\sigma B$, then $(A-\rho_k B)V = B V (\Lambda - \rho_k I)$ and $MV = B V (\Lambda-\sigma I)$, where $\Lambda = \mathrm{diag}(\lambda_1,\ldots,\lambda_n)$ is the diagonal matrix of the eigenvalues of $(A,B)$ and $V=[v_1,\ldots,v_n]$ contains the corresponding eigenvectors. It follows that $M(A-\rho_k B)V = V(\Lambda-\sigma I)^{-1}(\Lambda-\rho_k I)$. That is, $v_i$ is an eigenvector of $M(A-\rho_k B)$ associated with eigenvalue $\sigma_i = \frac{\lambda_i -\rho_k}{\lambda_i - \sigma}$. Since $C_k = L^{-1}(A-\rho_k B)L^{-T}$ and $L^{-T}L^{-1}(A-\rho_k B)$ have identical spectrum, $\sigma_i$ is also an eigenvalue of $C_k$. Therefore, if $\sigma$ and $\rho_k$ are close to $\lambda_1$, then $\sigma_2,\ldots,\sigma_n$ are all close to $1$, and hence $\epsilon_m = \min_{p \in \mathcal{P}_m,p(\sigma_1) = 1}\max_{2\leq i \leq n} |p(\sigma_i)|$ would be fairly small with a small value of $m$, indicating a fast rate of convergence in outer iteration.

\subsection{Global quasi-optimality} Theorem \ref{thmcvgrate} shows that Algorithm~\ref{alg:pl+1} converges \emph{at least} linearly with an asymptotic factor no greater than $\epsilon_m^2$. Our goal is to explore the role of the search directions $p_k$ (the `+1' strategy), which helps the algorithm achieve a greatly improved  convergence rate associated with global quasi-optimality. The global quasi-optimality is defined as follows.
\begin{dftn}\label{defnglbopt}Consider an iterative method for computing the lowest eigenpair $(\lambda_1,v_1)$ of a real symmetric matrix pencil $(A,B)$ with positive definite $B$. Let $x_0$ be the starting vector, $x_k$ be the approximation obtained  at step $k$, and $\theta_k = \angle(x_k,v_1)_B = \cos^{-1}\frac{(v_1,x_k)_B}{\|v_1\|_B\|x_k\|_B}$ be the error angle of $x_k$. Let $U_1,U_2,\ldots$ be a sequence of subspaces of increasing dimension, such that for each $k\geq 1$, $x_k \in U_k$, and $U_i \subset U_j$ for all $1\leq i < j$. Let $y_k^*\in U_k$ be the global minimizer of the Rayleigh quotient $\rho(x)=\frac{x^TAx}{x^TBx}$ in $U_k$. Then the iterate $x_k$ achieves \emph{global quasi-optimality} if 
\begin{eqnarray}\label{glbqopt}
\lim_{\theta_0 \rightarrow 0}\frac{\rho(x_k)-\rho(y_k^*)}{\rho(x_k)-\lambda_1} = 0.
\end{eqnarray}
\end{dftn}

\subsubsection{Linear convergence assumption}To show the global quasi-optimality of Algorithm~\ref{alg:pl+1}, we first make an assumption of its \emph{precisely} linear convergence.

\begin{aspt}\label{assumptionlincvg}
Assume that Algorithm~\ref{alg:pl+1} starting with initial $\rho(x_0) \in (\lambda_1,\lambda_2)$ converges {precisely} linearly \textup{(}not superlinearly or faster\textup{)} to $\lambda_1$; in other words, there exist constants $0 < \underline{\xi} < \bar{\xi} < 1$, independent of the progress of Algorithm~\ref{alg:pl+1}, such that
\begin{eqnarray}\label{lincvgfactor}
\underline{\xi}^{k-j}(\rho_j-\lambda_1) \leq \rho_k - \lambda_1 \leq \bar{\xi}^{k-j}(\rho_j - \lambda_1), \:\mbox{ for all }\: 0 \leq j < k.
\end{eqnarray}
\end{aspt}

The assumption on the existence of a lower bound $\underline{\xi}$ on the convergence factor is realistic. Given $M \approx (A-\sigma B)^{-1}$ with a \emph{fixed} $\sigma < \lambda_1$, and a \emph{fixed} dimension $m$ of the preconditioned Krylov subspace, extensive experiments suggest that Algorithm~\ref{alg:pl+1} exhibits simply linear convergence as the outer iteration proceeds. %We shall use the assumption of linear convergence \eqref{assumptionlincvg} to study the global quasi-optimality of Algorithm~\ref{alg:pl+1} with a sufficiently small initial error angle $\theta_0 = \angle(x_0,v_1)_B$. 

Assumption \ref{assumptionlincvg} has an equivalent form in terms of angles. Consider two iterates $x_j = v_1\cos\theta_j+f_j\sin\theta_j$ and $x_k = v_1\cos\theta_k + f_k \sin\theta_k$, $0 \leq j < k$, where $f_j, f_k \perp Bv_1$, $\|f_j\|_B=\|f_k\|_B = 1$, and $\|x_j\|_B = \|x_k\|_B = 1$. From \eqref{rqerror}, we have \eqref{lincvgfactor} equivalent to  $\underline{\xi}^{k-j}\sin^2\theta_j(\rho(f_j)-\lambda_1) \leq \sin^2\theta_k(\rho(f_k)-\lambda_1) \leq \bar{\xi}^{k-j}\sin^2\theta_j (\rho(f_j)-\lambda_1)$, i.e.,
\begin{eqnarray}\label{assumptionlincvg_eqv}
\textstyle\sqrt{\underline{\xi}}^{k-j} \sqrt{\frac{\rho(f_j)-\lambda_1}{\rho(f_k)-\lambda_1}}\sin\theta_j \leq \sin\theta_k \leq \sqrt{\bar{\xi}}^{k-j} \sqrt{\frac{\rho(f_j)-\lambda_1}{\rho(f_k)-\lambda_1}}\sin\theta_j.
\end{eqnarray}

\subsubsection{Relevant spaces}
To study the global quasi-optimality, we define
\begin{eqnarray}
W_k &=& \mathrm{span}\{g_0,\ldots,g_{k-1}\}=  \mathrm{span}\{Md_0,\ldots,Md_{k-1}\}, \mbox{ where}\\ \nonumber
&& d_\ell \in \mathcal{K}_m\left((A-\rho_\ell B)M,r_\ell\right), \mbox{ and} \\ 
U_k &=& \mathrm{span}\{x_0\}+W_k.
\end{eqnarray}

\begin{lem}\label{lemmaWkxk} For all $k \geq 1$, $W_k = \mathrm{span}\{p_0,\ldots,p_{k-1}\}$ and $x_k \in U_k$. 
\end{lem}
The proof is straightforward by induction and omitted. Next, we make an important assumption about $W_k$ for the subsequent analysis.
\begin{aspt}\label{anglev1Wk}Assume that there is a constant $\delta > 0$, independent of $\theta_0=\angle(x_0,v_1)_B$, such that  $\angle(v_1,W_k) \geq \delta$ for all $k \geq 1$.
\end{aspt}

{\bf Remark.} The above assumption is guaranteed to hold if the preconditioner $M$ is equipped with the projector $P=I-\frac{x_0 x_0^TB}{x_0^TBx_0}$. That is, let $\mathbb{M}=P^TMP$ be the preconditioner for Algorithm~\ref{alg:pl+1}, such that $\mathrm{range}(\mathbb{M}^{\dagger})=\left(\mathrm{span}\{Bx_0\}\right)^\perp$. If $x_0$ is sufficiently close to $v_1$, then $W_k \subset \mathrm{range}(\mathbb{M}^{\dagger}) \approx \mathrm{span}\{v_2,\ldots,v_n\}$, and hence $\angle(v_1,W_k) \geq \delta$. The Jacobi-Davidson method uses this strategy to enhance the robustness convergence. 

%From Lemma \ref{lemmaWkxk}, $x_k \in U_k$ with $\|x_k\|_B=1$ can be written as $x_k = \beta_{0k}x_0 + \gamma_{0k}g_{0k}$, where $g_{0k} \in W_k$ with $\|g_{0k}\|_B = 1$. %Also, by letting $j=0$ in Lemma \ref{xjxk_lemma1}, we see that there is a valid decomposition of $x_k$ described in Lemma \ref{xjxk_lemma1} with the normalized \emph{correction vector} $g_{0k} \in W_k$. This observation helps us justify the critical assumption in Lemma \ref{xjxk_lemma1} that $\rho(g_{0k}) - \rho(x_0) \geq d > 0$. In fact, we shall see in the following Lemma that $g_{0k}$ is almost orthogonal to $Mx_0$. 

%For a given $k$, let $y^*\in U_k$ be the global minimizer such that $\rho(y^*) \leq \rho(x)$ for all $x \in U_k$. We shall show that $x_k$ computed by Algorithm~\ref{alg:pl+1} satisfies $\frac{\rho(x_k)-\rho(y^*)}{\rho(x_k)-\lambda_1} \leq \frac{(\underline{\xi}^{-1})^k\sum_{j=3}^k C_j\sin\theta_j}{\lambda_2-\lambda_1}$, where $\theta_j = \angle(x_j,v_1)_B$, and $C_j > 0$ is some constant independent of $\theta_j$. In other words, for a sufficiently small $\theta_0 = \angle(x_0,v_1)_B$, in a modest number of outer steps $k$, Algorithm~\ref{alg:pl+1} yields `locally optimal' iterate $x_j$ almost as good as the corresponding global optimizer $y^* \in U_j$ for all $3 \leq j \leq k$ ($x_1$ and $x_2$ are global minimizers in $U_1$ and $U_2$, respectively). Our analysis will be carried out primarily based on an optimization point of view. 

\subsubsection{Preliminary results}
We present a few preliminary results useful for the proof of the main theorems in the next section. For the sake of readability, we move most of the technical proof of our results to the Appendix. 
\begin{lem}\label{lemmaoptalpha}
Let $x \approx v_1$ in \eqref{xdecomp}, and $p$ be a descent direction for $\rho(x)$ such that $\left( p,\nabla \rho(x) \right) < 0$. Assume that there exists a $\delta > 0$ independent of $\theta=\angle(x,v_1)_B$, such that $\rho(p)-\rho(x) \geq \delta$. Then the optimal step size $\alpha^*$ minimizing $\rho(x+\alpha p)$ is the unique or the smaller positive root of $a(x,p)\alpha^2+b(x,p)\alpha+c(x,p) = 0$, where
\begin{eqnarray}\label{coefabc}
&&a(x,p) = (p^TAp)(p^TBx)-(p^TBp)(p^TAx) = \|p\|^2_B x^T(A-\rho(p)B)p, \\ \nonumber
&&b(x,p) = (p^TAp)(x^TBx)-(p^TBp)(x^TAx) = \|x\|^2_B p^T(A-\rho(x) B) p \\ \nonumber
&& \qquad\quad = \|x\|^2_B\|p\|^2_B\big(\rho(p)-\rho(x)\big) \geq  \|x\|^2_B\|p\|^2_B\delta>0, \:\:\mbox{and}  \\ \nonumber
&&c(x,p) = (p^TAx)(x^TBx)-(p^TBx)(x^TAx) = \|x\|^2_B p^T(A-\rho(x) B)x < 0. %\mbox{\textup{ (descent direction)}.}
\end{eqnarray}
\end{lem}%\begin{proof} See Appendix. \end{proof}

%With the knowledge of the optimal step size $\alpha^*$ in a given descent direction $p$, we can estimate how much the Rayleigh quotient decreases from $x$ to $x+\alpha^* p$.

\begin{lem}\label{rhodecrease}
Let $x$ be an approximation to $v_1$ with decomposition \eqref{xdecomp}, $p$ be a descent direction for $\rho(x)$, and $\delta > 0$ be a constant independent of $\theta = \angle(x,v_1)_B$, such that $\rho(p) - \rho(x) \geq \delta > 0$, and $\alpha^*$ the optimal step size minimizing $\rho(x+\alpha p)$. Then for sufficiently small $\theta$, with $a,b$ and $c$ defined in \eqref{coefabc},
\begin{eqnarray}\nonumber
\frac{-\frac{c^2}{2b}-\frac{ac^3}{3b^3}+\mathcal{O}(c^4)}{\min\{\|x\|^2_B,\|x+\alpha^* p\|^2_B\}} &\leq& \rho(x+\alpha^* p) - \rho(x) \leq \frac{-\frac{c^2}{2b}-\frac{ac^3}{3b^3}+\mathcal{O}(c^4)}{\max\{\|x\|^2_B,\|x+\alpha^* p\|^2_B\}},
\end{eqnarray}
and hence $\rho(x+\alpha^*p)-\rho(x) = -\mathcal{O}(\sin^2\theta)\mathcal{O}\big(\cos^2\angle(p,\nabla \rho(x))\big)$.
\end{lem}%\begin{proof} See Appendix. \end{proof}
%%%%%%= 

\begin{lem}\label{xjxk_lemma1} Let $j \geq 0$ be fixed, and $k > j$ be a variable integer. Consider two iterates $x_j=v_1\cos\theta_j+f_j\sin\theta_j \in U_j$, $x_k=v_1\cos\theta_k + f_k\sin\theta_k \in U_k$ computed by Algorithm~\ref{alg:pl+1}, where $f_j,f_k \perp Bv_1$, $\|f_j\|_B=\|f_k\|_B=1$, $\|x_j\|_B=\|x_k\|_B=1$, with sufficiently small $\theta_j, \theta_k$ such that $\rho(x_k) \leq \rho(x_j) < \lambda_2$. Consider a decomposition 
\begin{eqnarray}\label{xjxkeqn}
x_k = \beta_{jk}x_j+\gamma_{jk}g_{jk},
\end{eqnarray}
with $\beta_{jk},\gamma_{jk}>0$, $\|g_{jk}\|_B=1$, $g_{jk} \in \mathrm{span}\{p_{j-1},g_j,\ldots,g_{k-1}\}$ $(j > 0)$ or $g_{0k} \in W_k=\mathrm{span}\{g_0,\ldots,g_{k-1}\}$ $(j=0)$. Under Assumption \ref{anglev1Wk}, let $\mu_{jk}=g_{jk}^TBx_j$, and we have%Assume that $g_{jk}$ satisfies $\rho(g_{jk}) -\rho(x_j) = d_{jk} \geq d > 0$ for all $k > j$. Then 
\begin{eqnarray}\label{decompxk}
\qquad \gamma_{jk} = \mathcal{O}(\sin\theta_j),\, \beta_{jk} = 1-\mu_{jk}\gamma_{jk}-\frac{1}{2}(1-\mu^2_{jk})\gamma^2_{jk}+\mathcal{O}(\gamma^4_{jk})=1-\mathcal{O}(\sin\theta_j).
\end{eqnarray}
\end{lem}%\begin{proof} See Appendix. \end{proof}

\subsubsection{Main theorems}
We are ready to prove the global quasi-optimality of Algorithm~\ref{alg:pl+1}. To this end, we will show the following results step by step.
\begin{enumerate}
\item If the search directions $\{p_0,\ldots,p_{k}\}$ of Algorithm~\ref{alg:pl+1} are approximately conjugate, then $x_{k+1} \in U_{k+1}$ is sufficiently close to the global minimizer in $U_{k+1}$ as long as $x_k \in U_k$ is sufficiently close to the global minimizer in $U_k$;
\item Any two consecutive search directions $p_{k-1}$ and $p_k$ are approximately conjugate, and so are $p_0$ and $p_2$ (hence, $x_3$ is globally quasi-optimal);
\item If $x_k$ is globally quasi-optimal, then $r_k$ is nearly orthogonal to $W_k$; in fact, $\lim_{\theta_0\rightarrow 0} \cos\angle(r_k,W_k) = 0$ (we can hence define $\cos\angle(r_k,W_k)\big|_{\theta_0 = 0} = 0$);
\item Assume that $\cos \angle(r_k,W_k)$ is differentiable at $\theta_0=0$. If $x_k \in U_k$ is globally quasi-optimal, then $p_k$ is approximately conjugate to $\{p_0,\ldots,p_{k-1}\}$, and hence $x_{k+1} \in U_{k+1}$ is also  globally quasi-optimal, as a result of induction. 
\end{enumerate}

First, we show that a set of approximately conjugate search directions guarantee that the quality of the iterate of Algorithm~\ref{alg:pl+1} at (outer) step $k$ for approximating the corresponding global minimizer can be extended to step $k+1$ with a possible very small deterioration on the order of $\mathcal{O}(\sin\theta_{k+1})\mathcal{O}(\sin^2\theta_0)$. 
\begin{thm}\label{conjoptext}
Let $\{x_k\}$ be the iterates of Algorithm~\ref{alg:pl+1}, $r_k = (A-\rho(x_k)B)x_k$ the residual, and $\{p_k\}$ the $B$-normalized search directions. For a given $k \geq 0$, consider all vectors of the form $z= y+\alpha p_k \in U_{k+1}$, where $y=\beta_y x_0 +\gamma_y g_y \in U_k$ with $\|y\|_B = \|g_y\|_B=1$, and $g_y \in W_k$, satisfying $\rho(z) \leq \rho(y) \leq \rho(x_0)$. Assume that $\{p_k\}$ are pairwise approximately conjugate, i.e., $p_k^T(A-\rho(x_k)B)p_j = \mathcal{O}(\sin\theta_j)$ $(0 \leq j < k)$. Let $y^*$ and $z^*$ be the global minimizer of $\rho(\cdot)$ in $U_k$ and $U_{k+1}$, respectively. Then $$\rho(x_{k+1})-\rho(z^*) \leq \rho(x_k)-\rho(y^*)+\mathcal{O}(\sin\theta_{k+1})\mathcal{O}(\sin^2\theta_0).$$
\end{thm}

To prove the global quasi-optimality of Algorithm~\ref{alg:pl+1}, it is hence crucial to show that the $B$-normalized search directions are pairwise approximately conjugate, i.e.,  $$p_k^T(A-\rho(x_k)B)p_j = \mathcal{O}(\sin\theta_j)$$ for all integers $0 \leq j < k$. To achieve this, our second step is to show that any two consecutive search directions $p_{k-1}$ and $p_k$ are approximately conjugate, and so are $p_0$ and $p_2$. We will establish the complete near conjugacy in Theorem \ref{thminductive}. 

\begin{lem}\label{apprxconjkm1k}
The $B$-normalized search directions of Algorithm~\ref{alg:pl+1} satisfy $$p_k^T(A-\rho(x_k)B)p_{k-1} = \mathcal{O}(\sin^2\theta_{k-1})\:\:\:\:\mbox{for all }\:k \geq 1.$$
\end{lem}%\begin{proof} See Appendix. \end{proof}

%We note that the above conjugate relation remains valid if the shift $\rho(x_k)$ is replaced with $\rho(x_{k-1})$. In fact 
%\begin{eqnarray}\label{shiftedconjkkm1}
%\hspace{-0.3in}p_k^T(A-\rho(x_{k-1})B)p_{k-1} &=& p_k^T(A-\rho(x_k)B)p_{k-1}+\left(\rho(x_k)-\rho(x_{k-1})\right)p_k^TBp_{k-1} \\ \nonumber
%&=&\mathcal{O}(\sin\theta_{k-1})+\mathcal{O}(\sin^2\theta_{k-1})( p_k, p_{k-1} )_B = \mathcal{O}(\sin\theta_{k-1}).
%\end{eqnarray}

To show the complete near conjugacy $p_k^T(A-\rho(x_k)B)p_j = \mathcal{O}(\sin\theta_j)$ for all $0 \leq j < k$, we make an assumption about $g_k \in M\mathcal{K}_m\left((A-\rho_k B)M,r_k\right)$ as follows.

\begin{aspt}\label{samepoly}
Let $g_k\!=\! Mp^{(k)}_{m-1}\!\left((A-\rho_k B)M\right)r_k$ at step 6 of Algorithm~\ref{alg:pl+1}, where $p^{(k)}_{m-1}$ is a polynomial of degree no greater than $m\!-\!1$ with real coefficients. With a sufficiently small $\theta_0 = \angle(x_0,v_1)_B$, assume that for all $0 \leq j < k$, \vspace{-0.1in}
\begin{eqnarray}\nonumber
Mp^{(k)}_{m-1}\left((A-\rho_k B)M\right) = Mp^{(j)}_{m-1}\left((A-\rho_j B)M\right)+\mathcal{O}(\sin\theta_j).\vspace{-0.1in}
\end{eqnarray}
\end{aspt}

{\bf Remark.} Assumption \ref{samepoly} holds trivially for LOPCG, i.e., Algorithm~\ref{alg:pl+1} with $m=1$, because $p^{(k)}_{m-1}(\cdot) = I$ for all $k$, so that $g_k = Mr_k$ up to a scaling factor. 

\begin{lem}\label{apprxconj02} 
Under Assumptions \ref{assumptionlincvg} and \ref{samepoly}, $p_2^T(A-\rho(x_2)B)p_0 = \mathcal{O}(\sin\theta_0)$. 
\end{lem}
%\begin{proof} See Appendix. \end{proof}
%The approximate conjugate relation also holds if $\rho(x_k)$ is replaced with $\rho(x_{k-1})$. In fact,
%\begin{eqnarray}
%p_{k-1}^T(A-\rho(x_{k-1})B)p_k &=& p_{k-1}^T(A-\rho(x_{k})B)p_k + [\rho(x_k)-\rho(x_{k-1})]p_{k-1}^TBp_k \\ \nonumber
%&=& \mathcal{O}(\sin\theta_k)+\mathcal{O}(\sin^2\theta_{k-1})(p_{k-1},p_k)_B = \mathcal{O}(\sin\theta_{k-1})
%\end{eqnarray}

%The main observation here is that $p_2$ is a particular linear combination of $Mr_2$ and $p_1$, such that it is approximately conjugate to $p_0$. In summary, LOPCG generates the global minimizer $x_2$ in $\beta x_0 + \mathrm{span}\{Mr_0,Mr_1\}$, and then constructs the new search direction $p_2$ satisfying the approximate conjugate conditions \eqref{p2p1} and \eqref{p2p0}. Our aim here is to show that the next iterate $x_3$ found by LOPCG is very close to the global minimizer $x_3$ in  $\beta x_0 + \mathrm{span}\{Mr_0,Mr_1, Mr_2\}$.

%Since LOPCG's iterate $x_2$ is the global minimizer in $W_2$, we have $\rho(x_3)-\rho(z^*) \leq \mathcal{O}(\sin^3\theta_0)$, where $z^*$ is the global minimizer in $W_3$. 

From Lemmas \ref{apprxconjkm1k} and \ref{apprxconj02}, the search directions $\{p_0,p_1,p_2\}$ are pairwise approximately conjugate. In addition, since $x_2 \in \mathrm{span}\{x_1,g_1,p_0\}=\mathrm{span}\{x_0,p_0,p_1\}=U_2$ is obtained from the RR projection, it is the global minimizer in $U_2$. Let $z^*$ be the global minimizer of $\rho(\cdot)$ in $U_3$. It follows from Theorem \ref{conjoptext} that 
\begin{eqnarray}\label{basecase}
\qquad \rho(x_3)-\rho(z^*) \leq \rho(x_2)-\rho(y^*)+\mathcal{O}(\sin\theta_3)\mathcal{O}(\sin^2\theta_0) = \mathcal{O}(\sin\theta_3)\mathcal{O}(\sin^2\theta_0),\vspace{-0.1in}
\end{eqnarray}
where $x_2 = y^*$ is the global minimizer in $U_2$. We thus have the \emph{base case}: $x_3$ is a global quasi-minimizer in $U_3$. 

The rest of our work is focused on the \emph{inductive step}: assuming that the global quasi-optimality is achieved at $x_{k}$, we want to show that the new $p_k$ is approximately conjugate to $\{p_0,\ldots,p_{k-1}\}$, such that the quasi-optimality can be extended to $x_{k+1}$. %To accomplish this, we need the following approximate orthogonality condition. 

\begin{lem}\label{qopteqvorth}
For a given $k \geq 3$, assume that the iterate $x_k$ of Algorithm~\ref{alg:pl+1} achieves the {global quasi-optimality} (Definition \ref{defnglbopt}).
%; that is
%\begin{equation}\label{glbqopt}
%\lim_{\theta_0 \rightarrow 0} \frac{\rho(x_k)-\rho(y^*)}{\rho(x_k)-\lambda_1} = 0,
%\end{equation}
%where $y^*$ is the global minimizer of $\rho(\cdot)$ in $U_k$. 
Then $r_k=(A-\rho(x_k)B)x_k$ satisfies \begin{equation}\label{nearorth}
\lim_{\theta_0 \rightarrow 0}\cos \angle(r_k, W_k) = 0,
\end{equation}
i.e., $r_k$ is asymptotically orthogonal to $W_k = \mathrm{span}\{p_0,\ldots,p_{k-1}\}$. 
\end{lem}
%\begin{proof} See Appendix. \end{proof}

Note that $\cos\angle(r_k,W_k)|_{\theta_0 = 0}$ is not defined, since Algorithm~\ref{alg:pl+1} with $x_0 = v_1$ would not proceed. However, thanks to \eqref{nearorth}, it is reasonable for us to make the following assumption about the behavior $\cos\angle(r_k,W_k)$ near $\theta_0=0$.

\begin{aspt}\label{differentiability}For a given step $k \geq 3$, suppose that the iterate $x_k$ of Algorithm~\ref{alg:pl+1} achieves global quasi-optimality \eqref{glbqopt}. We define $\cos\angle(r_k,W_k)\big|_{\theta_0=0}=0$ and assume that $\cos\angle(r_k,W_k)$ is differentiable at $\theta_0=0$; that is, we assume that $\cos\angle(r_k,W_k) \leq \mathcal{O}(\sin\theta_0) = \mathcal{O}(\sin\theta_\ell)$ \textup{($1 \leq \ell \leq k$)} for a sufficiently small $\theta_0$. 
\end{aspt}

We note that Assumption \ref{differentiability} is not just presented for technical convenience, but is consistent with our numerical experience, at least for relatively small $k$. Under such an assumption, we can establish the inductive step as follows. 

\begin{thm}\label{thminductive}Assume that the $B$-normalized $\{p_0,\ldots, p_{k-1}\}$ of Algorithm~\ref{alg:pl+1} are approximately conjugate, and hence $x_k$ achieves global quasi-optimality \eqref{glbqopt}. Under Assumptions \ref{assumptionlincvg}, \ref{anglev1Wk}, and \ref{samepoly}, $\{p_0,\ldots,p_{k}\}$ are approximately conjugate. 
%Then at any given step $k$, for sufficiently small $\theta_0$, $\frac{\rho(x_k)-\rho(y^*)}{\rho(x_k)-\lambda_1} = \mathcal{O}(\sin\theta_0),$ where $y^*$ is the global minimizer in $U_k$.
\end{thm}
\begin{proof}
%The proof is carried out by mathematical induction. First, we note that LOPCG iterate $x_2$ is globally optimal in $U_2=\mathrm{span}\{x_0\}+W_2 = \mathrm{span}\{x_0\}+\mathrm{span}\{Mr_0,Mr_1\}$, and that $p_1^T(A-\rho(x_1)B)p_0 = \mathcal{O}(\sin\theta_0)$, $p_2^T(A-\rho(x_2)B)p_1 = \mathcal{O}(\sin\theta_1)$ (Lemma \ref{apprxconjkm1k}) and $p_2^T(A-\rho(x_2)B)p_0 = \mathcal{O}(\sin\theta_0)$ (Lemma \ref{apprxconj02}). That is, global quasi-optimality and pairwise approximate conjugate search directions are achieved at step 2. 
%Assume that $\{p_0,\ldots,p_{k-1}\}$ are pairwise approximately conjugate, and $x_k$ is globally quasi-optimal. We know from Lemma \ref{qopteqvorth} that $\lim_{\theta_0 \rightarrow 0}\cos\angle(r_k,W_k) = 0$. By the assumption of differentiability, $\big|\cos\angle(r_k,W_k)\big| \leq \mathcal{O}(\sin\theta_0)$. 
To complete the proof, it is sufficient to show that $p_k$ is approximately conjugate to $p_{\ell}$ for all $0 \leq \ell \leq k-2$. For any $\ell$, $0 \leq \ell \leq k-2$, note that $x_{\ell+1} = \beta_{\ell(\ell+1)}x_\ell + \gamma_{\ell(\ell+1)}p_\ell$, where $\gamma_{\ell(\ell+1)} = \mathcal{O}(\sin\theta_\ell)$ and $\beta_{\ell(\ell+1)} = 1-\mathcal{O}(\sin\theta_\ell)$. It follows that $p_\ell = \frac{1}{\gamma_{\ell(\ell+1)}}\left(x_{\ell+1}-\beta_{\ell(\ell+1)}x_\ell\right)$, and 
{\small
\begin{eqnarray}\nonumber
&& \big|p_\ell^T (A-\rho(x_k)B)g_k\big| = \frac{1}{|\gamma_{\ell(\ell+1)}|}\Big|\big(x_{\ell+1}^T-\beta_{\ell(\ell+1)}x_\ell^T\big)(A-\rho(x_k)B)g_k\Big| \\ \nonumber
&=& \frac{1}{|\gamma_{\ell(\ell+1)}|}\Big|\big[x_{\ell+1}^T(A-\rho(x_{\ell+1})B)-\beta_{\ell(\ell+1)}x_\ell^T(A-\rho(x_\ell)B)\big]g_k + \\ \nonumber
&& \qquad\qquad \qquad \left(\rho(x_{\ell+1})-\rho(x_k)\right)x_{\ell+1}^TBg_k - \beta_{\ell(\ell+1)}\left(\rho(x_\ell)-\rho(x_k)\right)x_\ell^TBg_k\Big| \\ \nonumber
v b&=& \frac{ \big|(r_{\ell+1}^T-\beta_{\ell(\ell+1)}r_\ell^T)g_k + \mathcal{O}(\sin^2\theta_{\ell+1})\|Bx_{\ell+1}\|\|g_k\|+ \mathcal{O}(\sin^2\theta_\ell)\|Bx_\ell\|\|g_k\| \big|}{|\gamma_{\ell(\ell+1)}|} \\ \nonumber
& \leq & \frac{\mathcal{O}(\sin^2\theta_\ell)\|g_k\|}{\mathcal{O}(\sin\theta_\ell)}+\frac{\big|(r_{\ell+1}^T-\beta_{\ell(\ell+1)}r_\ell^T)Mp^{(k)}_{m-1}\left((A-\rho_k B)M\right)r_k\big|}{|\gamma_{\ell(\ell+1)}|} \\ \nonumber
& = & \mathcal{O}(\sin\theta_\ell)\|g_k\|+\frac{1}{|\gamma_{\ell(\ell+1)}|}\Big|r_{\ell+1}^T\left[p^{(\ell+1)}_{m-1}\!\left(M(A-\rho_{\ell+1}B)\right)M+\mathcal{O}(\sin\theta_{\ell+1})\right]r_k \\ \nonumber
&& \qquad \qquad \qquad \qquad \qquad\: -\beta_{\ell(\ell+1)}r_\ell^T \left[p^{(\ell)}_{m-1}\!\left(M(A-\rho_\ell B)\right)M+\mathcal{O}(\sin\theta_\ell)\right]r_k\Big|.
%
%& \leq & {\frac{\|r_k\|\left(\|Mr_{\ell+1}\|\cos \angle(Mr_{\ell+1},r_k)+\beta_{\ell(\ell+1)}\|Mr_\ell\| \cos\angle(Mr_\ell,r_k)\right)+\|Mr_k\|\mathcal{O}(\sin^2\theta_\ell)}{\mathcal{O}(\sin\theta_\ell)}} \\ \nonumber
%&\leq & \frac{\|M\|\|Mr_k\|\left(\|M\|\|r_{\ell+1}\| +\beta_{\ell(\ell+1)}\|M\|\|r_\ell\|\right)\mathcal{O}(\sin\theta_0)}{\mathcal{O}(\sin\theta_\ell)}+ \|Mr_k\|\mathcal{O}(\sin\theta_\ell)\\ \nonumber
%&=& \kappa_2(M)\|Mr_k\|\mathcal{O}(\sin\theta_0)+\|Mr_k\|\mathcal{O}(\sin\theta_\ell) = \|Mr_k\|\mathcal{O}(\sin\theta_0).
\end{eqnarray}
}Since $p^{(\ell)}_{m-1}\left(M(A-\rho_\ell B)\right)M$ is symmetric, $r_{\ell}p^{(\ell)}_{m-1}\left(M(A-\rho_\ell B)\right)M= g^T_\ell$, and this holds similarly if $\ell$ is replaced with $\ell+1$. Also, since $x_{k}$ achieves global quasi-optimality, by Assumption \ref{differentiability}, $|\cos\angle(r_k,W_k)| \leq \mathcal{O}(\sin\theta_\ell)$ ($0 \leq \ell \leq k$). Since $g_\ell,g_{\ell+1} \in W_k$, we have $|\cos\angle(r_k,g_\ell)| \leq \mathcal{O}(\sin\theta_\ell)$ and $|\cos\angle(r_k,g_{\ell+1})| \leq \mathcal{O}(\sin\theta_{\ell+1})$. Note that $\|r_\ell\|=\mathcal{O}(\sin\theta_\ell)$, and $\|r_k\| = \mathcal{O}(\sin\theta_k) = \mathcal{O}(\sin\theta_{\ell+1}) = \mathcal{O}(\sin\theta_{\ell})$. Hence,
{\small
\begin{eqnarray}
\hspace{-0.5in}&&\big|p_\ell(A-\rho(x_k)B)g_k \big|  \\ \nonumber
&\leq& \|g_k\|\mathcal{O}(\sin\theta_\ell)+\frac{\Big|g_{\ell+1}^Tr_k + r_{\ell+1}^T r_k \mathcal{O}(\sin\theta_{\ell+1})-\beta_{\ell(\ell+1)}g_\ell^T r_k + r_{\ell}^T r_k \mathcal{O}(\sin\theta_\ell)\Big|}{\mathcal{O}(\sin\theta_\ell)} \\ \nonumber
&\leq & \|g_k\|\mathcal{O}(\sin\theta_\ell) + \frac{\|g_{\ell+1}\|\|r_k\||\cos\angle(r_k,g_{\ell+1})|+|\beta_{\ell(\ell+1)}|\|g_\ell\|\|r_k\||\cos\angle(r_k,g_\ell)|}{\mathcal{O}(\sin\theta_\ell)} \\ \nonumber
&& \qquad\qquad \qquad + \|r_{\ell+1}\|\|r_k\|\mathcal{O}(1) + \|r_\ell\|\|r_k\|\mathcal{O}(1) \\ \nonumber
&\leq & \|g_k\|\mathcal{O}(\sin\theta_\ell) + \frac{\|r_k\|\left(\|g_{\ell+1}\|\mathcal{O}(\sin\theta_{\ell+1})+|\beta_{\ell(\ell+1)}|\|g_\ell\|\mathcal{O}(\sin\theta_\ell)\right)}{\mathcal{O}(\sin\theta_\ell)} + \mathcal{O}(\sin^2\theta_\ell)\\ \nonumber
& \leq & \|g_k\|\mathcal{O}(\sin\theta_\ell) + \|g_{\ell+1}\|\mathcal{O}(\sin\theta_k) + \|g_\ell\|\mathcal{O}(\sin\theta_k)+\mathcal{O}(\sin^2\theta_\ell) = \mathcal{O}(\sin\theta_\ell).
%&\leq &  \|g_k\|\mathcal{O}(\sin\theta_\ell)+\frac{1}{\mathcal{O}(\sin\theta_\ell)}\Big|\|g_{\ell+1}\|\|r_k\||\cos\angle(r_k,g_{\ell+1})|+\|r_{\ell+1}\|\|r_k\|\mathcal{O}(\sin\theta_{\ell+1})- \\ \nonumber
%&& \qquad\qquad \qquad \qquad\qquad\qquad \beta_{\ell(\ell+1)}\|g_\ell\|\|r_k\||\cos\angle(r_k,g_\ell)| + \|r_\ell\|\|r_k\|\mathcal{O}(\sin\theta_\ell)\Big| \\ \nonumber
%&\leq&  \|g_k\|\mathcal{O}(\sin\theta_\ell) + \frac{\|r_k\|}{\mathcal{O}(\sin\theta_\ell)}\Big|\|g_{\ell+1}\|\mathcal{O}(\sin\theta_k)+\|r_{\ell+1}\|\mathcal{O}(\sin\theta_{\ell+1})- \\ \nonumber
%&& \qquad\qquad \qquad \qquad\qquad\qquad \beta_{\ell(\ell+1)}\|g_\ell\|\mathcal{O}(\sin\theta_k)+\|r_\ell\|\mathcal{O}(\sin\theta_\ell)\Big| \\ \nonumber
%&\leq&  \|g_k\|\mathcal{O}(\sin\theta_\ell)+\|g_{\ell+1}\|\mathcal{O}(\sin\theta_k)+\mathcal{O}(\sin\theta_{\ell})\mathcal{O}(\sin\theta_k) + \\ \nonumber
%&&  \qquad\qquad \qquad \quad |\beta_{\ell(\ell+1)}|\|g_\ell\|\mathcal{O}(\sin\theta_k)+\mathcal{O}(\sin\theta_\ell)\mathcal{O}(\sin\theta_k).
\end{eqnarray}
}Moreover, it is not difficult to see that $\big|p_\ell(A-\rho(x_k)B)g_k \big| = \|g_k\|\mathcal{O}(\sin\theta_\ell).$
Recall that $\widetilde{p}_k = g_k - \frac{p_{k-1}^T(A-\rho(x_{k-1})B)g_k}{p_{k-1}^T(A-\rho(x_{k-1})B)p_{k-1}}p_{k-1}$, such that $\|\widetilde{p}_k\|$ is proportional to $\|g_k\|$. The normalized search direction is $p_k = \frac{\eta_k}{\|g_k\|}\widetilde{p}_k$, where $\eta_k$ is chosen such that $\|p_k\|_B = 1$. Also, since $0 \leq \ell \leq k-2$, $p_\ell^T (A-\rho(x_{k-1})B)p_{k-1}=\mathcal{O}(\sin\theta_\ell)$, and
{\small
\begin{eqnarray}
&&\quad p_\ell^T (A-\rho(x_k)B)p_k = \frac{\eta_k}{\|g_k\|}p_\ell^T(A-\rho(x_k)B)\widetilde{p}_k \\ \nonumber
&=& \frac{\eta_k}{\|g_k\|}p_\ell^T(A-\rho(x_k)B)\left(g_k - \frac{p_{k-1}^T(A-\rho(x_{k-1})B)g_k}{p_{k-1}^T(A-\rho(x_{k-1})B)p_{k-1}}p_{k-1}\right) \\ \nonumber
&=& \frac{\eta_k}{\|g_k\|}\Big\{p_\ell^T(A-\rho(x_k)B)g_k -  \frac{p_{k-1}^T(A-\rho(x_{k-1})B)g_k}{p_{k-1}^T(A-\rho(x_{k-1})B)p_{k-1}} \times \\ \nonumber
&& \qquad\qquad\qquad\qquad\qquad \big[p_\ell^T(A-\rho(x_{k-1})B)p_{k-1}+ \left(\rho(x_{k-1})-\rho(x_k)\right)p_\ell^T B p_{k-1}\big]\Big\} \\ \nonumber
&=& \frac{\eta_k}{\|g_k\|}\Big\{\|g_k\|\mathcal{O}(\sin\theta_\ell) + \mathcal{O}(\|g_k\|)[\mathcal{O}(\sin\theta_\ell)+\mathcal{O}(\sin^2\theta_{k-1})]\Big\} = \mathcal{O}(\sin\theta_\ell).\vspace{-0.2in}
\end{eqnarray}
}

Finally, recall from Lemma \ref{apprxconjkm1k} that $p_{k-1}$ and $p_k$ are approximately conjugate. It follows that $\{p_0,\ldots,p_k\}$ are pairwise approximately conjugate. \qed
\end{proof}

With all the above work, we can finally present the major theorem of this section on the global quasi-optimality of all iterates generated by Algorithm~\ref{alg:pl+1}.\smallskip

\begin{thm}\label{globalopt}
Let $\{x_k\}$ be the iterates of Algorithm~\ref{alg:pl+1}, $r_k = (A -\rho(x_k)B)x_k$ the residual, and $\{p_k\}$ the $B$-normalized search directions. Suppose that Assumptions \ref{assumptionlincvg}, \ref{anglev1Wk}, and \ref{samepoly} hold. In Theorem \ref{conjoptext}, let $y^*$ and $z^*$ be the global minimizer in $U_k$ and $U_{k+1}$, respectively, and $C_{k+1}$ be the constant for which $\rho(x_{k+1})-\rho(z^*)\leq \rho(x_k)-\rho(y^*)+C_{k+1}\sin\theta_{k+1}\sin^2\theta_0$ if $\{p_0,\ldots,p_k\}$ are approximately conjugate. Then for any given $k \geq 3$, $\rho(x_k) - \rho(y^*) \leq \sum_{j=3}^k C_j\sin\theta_j \sin^2\theta_0,$ and therefore, 
{\small
\begin{eqnarray}\label{quasioptratio}
\qquad\: \lim_{\theta_0 \rightarrow 0}\frac{\rho(x_k)-\rho(y^*)}{\rho(x_k)-\lambda_1} &\leq& \lim_{\theta_0 \rightarrow 0}\frac{\sum_{j=3}^k C_j \sin\theta_j\sin^2\theta_0}{\underline{\xi}^k \left(\rho(f_0)\!-\!\lambda_1\right)\sin^2\theta_0}  \leq \lim_{\theta_0 \rightarrow 0} \frac{\left(\underline{\xi}^{-1}\right)^k\sum_{j=3}^k C_j\sin\theta_j}{\lambda_2\!-\!\lambda_1}=0.
\end{eqnarray}
}
\end{thm}
\begin{proof}
The proof is based on mathematical induction. The base case has been established in Lemmas~\ref{apprxconjkm1k} and \ref{apprxconj02} together with \eqref{basecase}. The inductive step is completed in Theorem~\ref{thminductive}. Then, by Theorem~\ref{conjoptext}, the conclusion holds. \qed
\end{proof}

{\bf Remark.} Interestingly, \eqref{quasioptratio} suggests that for a larger $k$, the \emph{relative} difference between $\rho(x_k)$ (locally optimized $\rho(\cdot)$ over $\mathrm{span}\{x_{k-1},g_{k-1},p_{k-2}\}$) and $\rho(y^*)$ (globally optimized $\rho(\cdot)$ over $U_k= \mathrm{span}\{x_0\}+\mathrm{span}\{g_0,\ldots,g_{k-1}\}$) tends to grow larger. The \emph{upper bound} on this relative difference scales like a linear function of $k$ (assuming that $|C_j\sin\theta_j| \leq C$ for all $j \geq 3$) multiplied by an exponentially increasing function of $k$. Nevertheless, for a \emph{given} outer iteration $k \geq 3$, Theorem~\ref{globalopt} shows that Algorithm~\ref{alg:pl+1} iterate $x_k \in U_k$ is almost as good as the corresponding global minimizer $y^*$ if $\theta_0$ is sufficiently small. As a result, Algorithm~\ref{alg:pl+1} would actually converge considerably faster than Theorem~\ref{thmcvgrate} suggests. This theorem also provides insight into the performance of LOPCG, which is an instance of Algorithm~\ref{alg:pl+1}, and the block extension LOBPCG. 

{\bf Remark.} A different viewpoint is followed in \cite{stathopoulos1998restarting} but with qualitatively similar results. 
Consider $k$ steps of CG solving the JD correction equation starting from the Ritz vector at the $i$-th iteration of Lanczos. Then, the distance between the Ritz vector of Lanczos at iteration $i+k$  and the CG solution after $k$ steps is bounded by $O(|\rho_i - \rho_{i+k}|)$. The locally optimal restarting approximates the $A$-norm minimization of CG over the entire space, so if Lanczos converges slowly or if $k$ is small, +K restarting will yield Ritz vectors close to the unrestarted Lanczos. If $\rho_i$ and $\rho_{i+k}$ are far, then convergence is already fast and the use of +K is not needed.

%%%%%%%%%%%%%%%%%%%%%%%%%%%%%%%%%%%%%%%%%%%%%%%%%%%%%%%%%%%%%%%%%%%
\section{Experiments}
\label{sec:numerical experiments}
First we show that TRPL+K does indeed achieve quasi optimality and then we investigate the effect of the maximum basis size $q$ and the number of previous retained vectors $l$ on the performance. Then we present an extensive set of experiments comparing TRPL+K against other methods on a variety of problems.

The matrix set is chosen to overlap with other papers in the literature. All matrices can be reproduced from these papers or downloaded from the SuiteSparse Matrix collection (formerly, at the University of Florida). Table \ref{tb:Matrix-info} lists some basic properties of these matrices. Matrices finan512 and Plate33K\_A0 have larger gap ratios so they are relatively easy problems. Matrices Trefethen\_20k and cfd1 are moderately hard problems, while matrices 1138bus and or56f are the hardest problems. The effectiveness of a method is typically manifested on harder problems.

We compare TRPL+K against the following eigenmethods: 
unrestarted GD (as a representative of unrestarted Lanczos), 
TRLan \cite{wu2000thick}, 
GD+K \cite{stathopoulos2007nearlyI,stathopoulos1998restarting}, 
LOBPCG \cite{knyazev2001toward}, 
and EigIFP \cite{golub2002inverse}. 
To allow for easier comparisons, we use only MATLAB 
implementations. We adopt the publicly available implementations for LOBPCG and EigIFP while we provide implementations for TRPL+K, TRLan, and GD+K.
All methods employ the same stopping criterion satisfying,
\begin{equation}\label{eq:stop_criterion}
    \|Au_i - \theta_i u_i\| \leq \|A\|_F \delta_{user},
\end{equation}
where $\|A\|_F$ is the Frobenius norm of $A$ and $\delta_{user}$ is a  user specified stopping tolerance. All experiments use \textit{$\delta$=1E-14}. 
All methods start with the same initial guess, $\text{rand}(n,k)$, with fixed seed number(12). We set the maximum number of restarts to 5000 for all methods. Since we focus on methods with limited memory, we set \textit{maxBasisSize=18}, \textit{minRestartSize=8} for TRPL+K, TRLan, and GD+K. Since EigIFP restarts with a single vector, we set \textit{maxBasisSize=18} for its inner iteration. For LOBPCG, the method always uses \textit{maxBasisSize=3}$p$. For other parameters we follow the defaults suggested in each code. We seek $p=$ 1, 5, and 10 algebraically smallest eigenpairs for both stardard and generalized eigenvalue problems. We employ soft locking since the desired number of eigenpairs are small. 
All computations are carried out on a Apple Macbook Pro with Intel Core i7 processors at 2.2GHz for a total of 4 cores and 16 GB of memory running the Mac Unix operating system. We use MATLAB 2016a with machine precision $\epsilon = 2.2E-16$. 
We compare both the number of matrix-vector products (reported as ``MV'' in the tables) and runtime in seconds (reported as ``Sec'').
%We report ``*'' if the method cannot converge to the requested stopping tolerance within 5000 restarts.

\begin{table}[ht!]
\centering
\caption{Properties of the test matrices for standard eigenvalue problems. Trefethen\_20k shorts as Tre20k, and Plate33K\_A0 shorts as PlaA0.}
\label{tb:Matrix-info}
\small
\begin{center}
    \begin{tabular}{ c c c c c }  \hline
    Matrix 		    & order   & nnz(A) 	 & $\kappa(A)$  & Source \\ \hline
    1138bus		    & 1138 	  & 4054     &  1.2E7   	& MM \\ 
%    nd3kf 		    & 9000    & 3279690  & 6.1E7  		& Yang  \\ 
    or56f 		    & 9000    & 2890150  & 6.6E7		& Yang \\
    Tre20k          & 20000   & 554466   & 2.9E5		& UF \\
    PlaA0           & 39366   & 914116   & 6.8E9	    & FEAP \\
    cfd1 		    & 70656   & 1825580  &  1.3E6 		& UF \\
    finan512 	    & 74752   & 596992   &  9.8E1		& UF \\ \hline
%    Andrews 		& 60000   & 760200   &  2.7E17  		& UF  	  \\ \hline
%    Cone\_A 		& 22032   & 1433068  &  2.7E6  		& FEAP  	  \\ \hline
    \end{tabular}
\end{center}
\end{table} 

\subsection{Quasi-optimality with +K}
Quasi optimality as defined in (\ref{glbqopt}) implies that as the initial guess becomes better, the relative difference between the vector iterates of TRPL+K and the unrestarted method tend to zero. Figure \ref{fig:quasi-opt} demonstrates that TRPL+K achieves this quasi optimality on two sample matrices. A standard eigenproblem is solved without preconditioning. TRPL+K uses only one inner iteration (i.e., equivalent to LOPCG) because it is easier to compare step by step to the unrestarted method. The Remark after Theorem 11 suggests that the ratio in \eqref{glbqopt} increases with the number of iterations, which we observe in the plots. Therefore, we plot only the first 100 iterations. 
What is important, however, is that for any given iteration, the ratio decreases as we make the initial guess better. Therefore, as TRPL+K converges asymptotically, it increasingly matches  the unrestarted method.

\begin{figure}[ht!]
  \centering
    \includegraphics[width=0.47\textwidth]{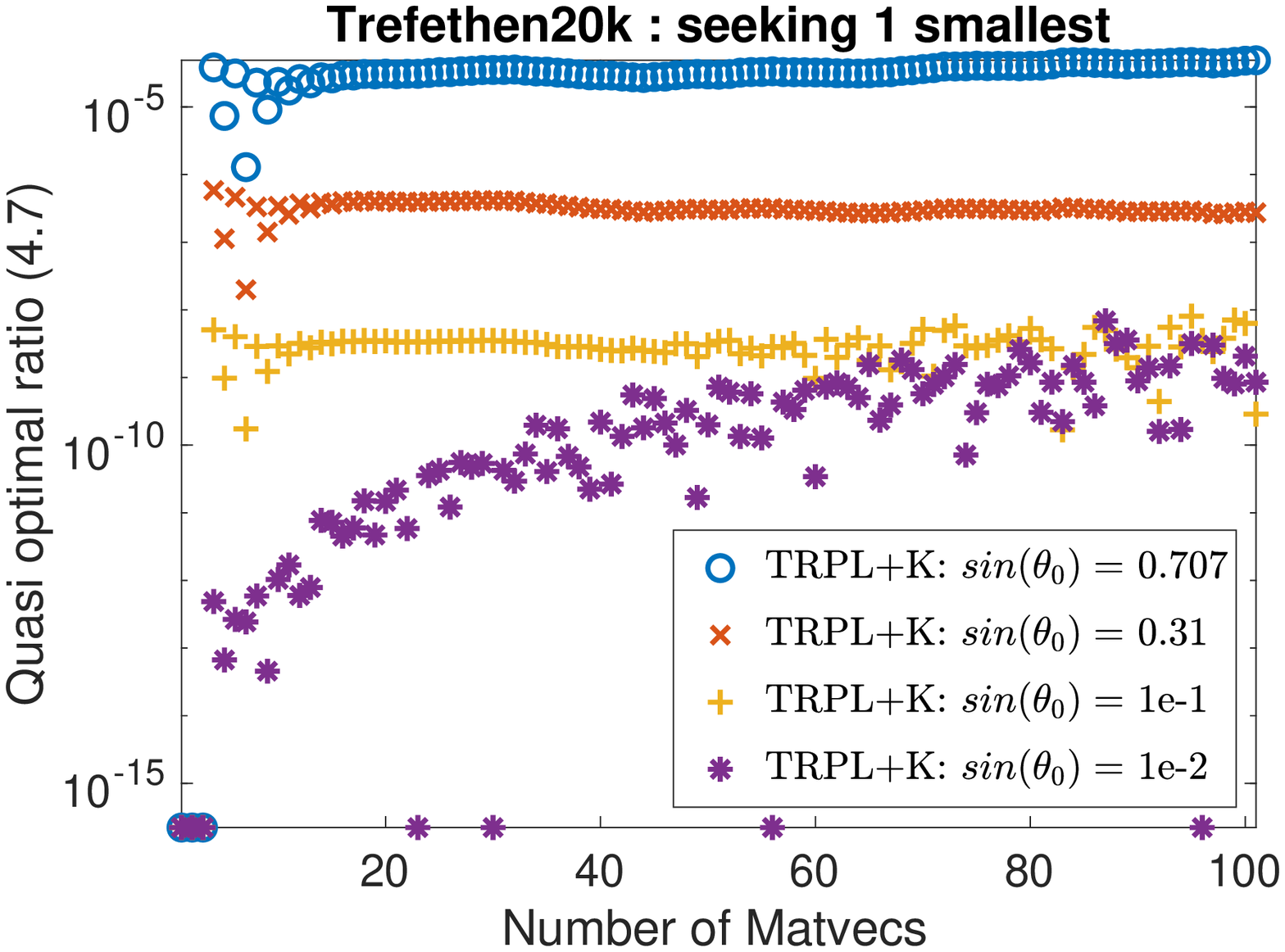}
    \includegraphics[width=0.47\textwidth]{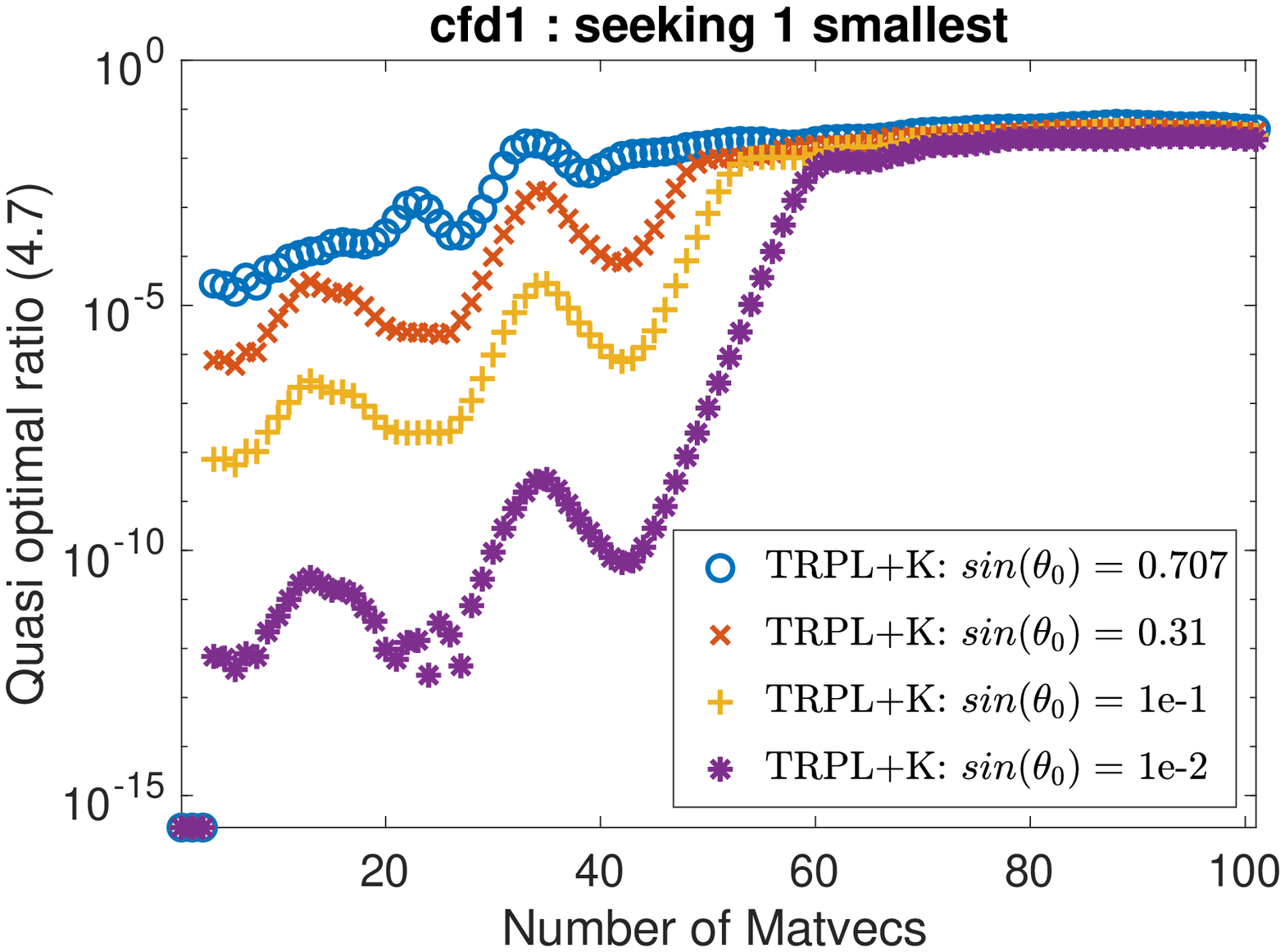}
    \caption{Confirming the quasi-optimality of TRPL+1 with no inner iterations by showing that the ratio (\ref{glbqopt}) converges to zero as the angle of the initial guess to the required eigenvector decreases. We plot this for the first 100 iterations.}
    \label{fig:quasi-opt}
\end{figure}

\subsection{The effects of maxBasisSize $q$ and maxPrevSize $l$}
First we investigate the effects of varying $q$ from 8 to 512 on the performance of TRPL+K (without a preconditioner) and TRLan compared against unrestarted Lanczos, with $l=1$ and $p=1$. We choose three hard problems where the differences are pronounced. Figure \ref{fig:seeking1_varyingmaxBsize} shows that as the maximum basis size increases, both TRPL+K and TRLan become similar to the unrestarted Lanczos. However, while TRLan requires the increased basis to significantly improve convergence, TRPL+K achieves very good or even close to optimal convergence with very small basis size. The slight increase of matvecs in TRPL+K with very large basis sizes may be attributed to the more targeted expansion of the subspace using the residual instead of the Lanczos vectors.

We then vary the number of previous vectors $l$ from 0 to 5 to study its impact on the convergence of TRPL+K (without a preconditioner), seeking $p=5$ smallest eigenvalues. When $l=0$ and without preconditioning, TRPL+K reduces to TRLan. Figure \ref{fig:seeking5_varyingmaxPsize} shows the similar trend for all three cases; a significant reduction in matvecs with $l=1$, while $l>1$ results in a slight deterioration in convergence. This is qualitatively similar to earlier observations for GD+K, where $l>1$ is beneficial only with a block method. This implies that we obtain all the benefits of the +K technique with minimal overhead.

\begin{figure}[ht!]
  \centering
  \begin{subfigure}[b]{0.32\textwidth}
      \includegraphics[width=\textwidth]{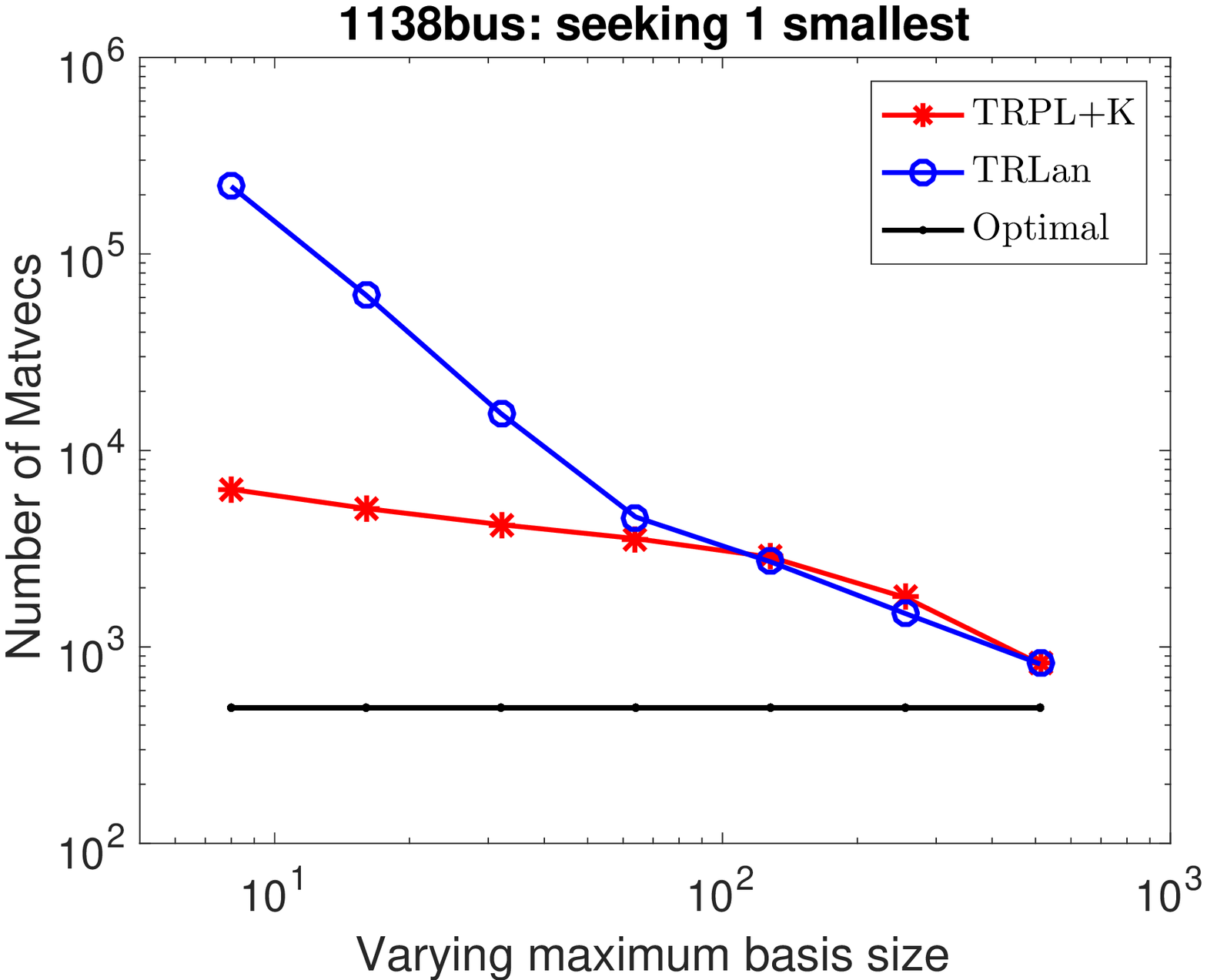}
      \caption{1138bus}
      \label{fig:1138bus_varyingmaxBsize}
    \end{subfigure}
    \begin{subfigure}[b]{0.32\textwidth}
      \includegraphics[width=\textwidth]{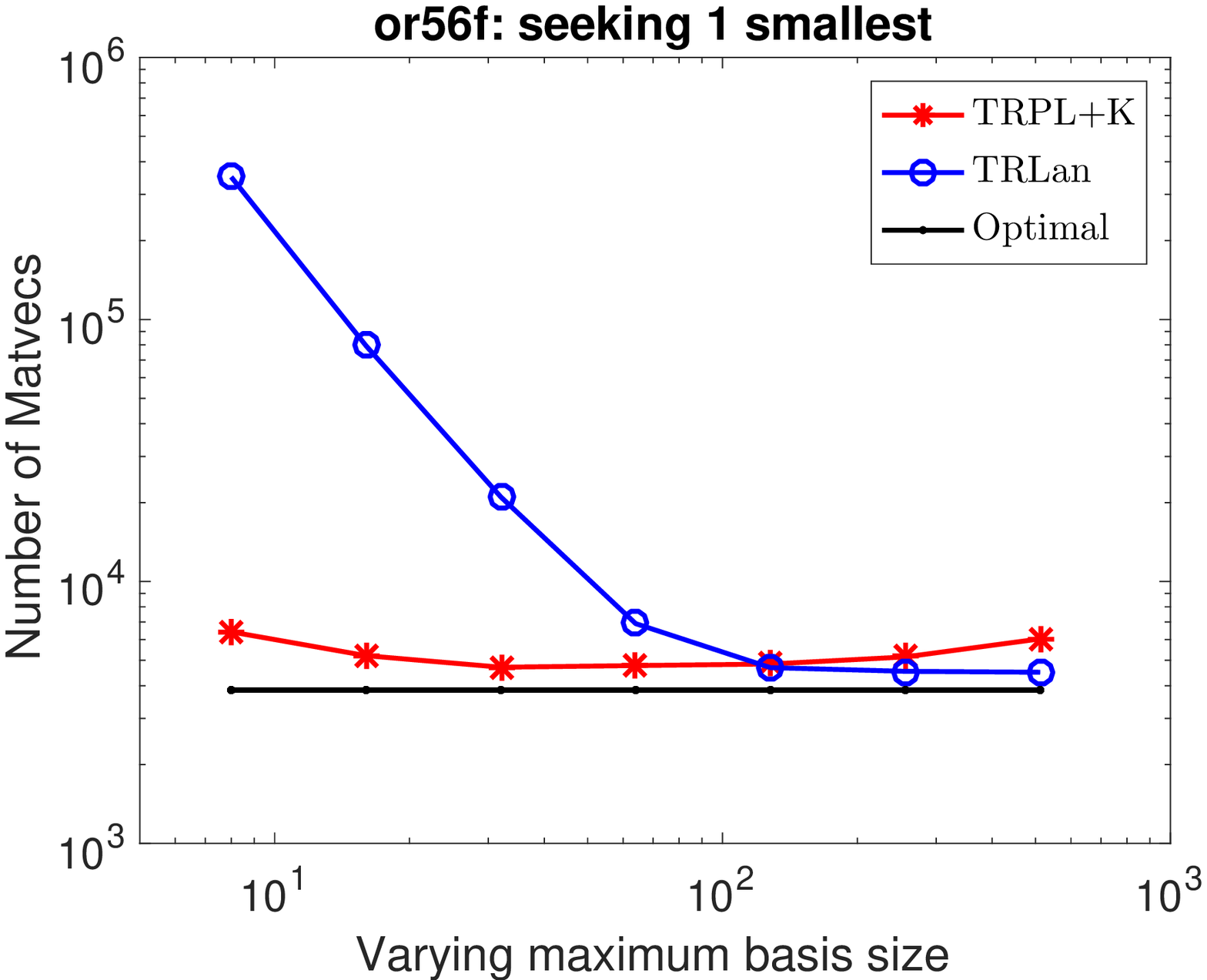}
      \caption{or56f}
      \label{fig:or56f_varyingmaxBsize}
     \end{subfigure}
     \begin{subfigure}[b]{0.32\textwidth}
      \includegraphics[width=\textwidth]{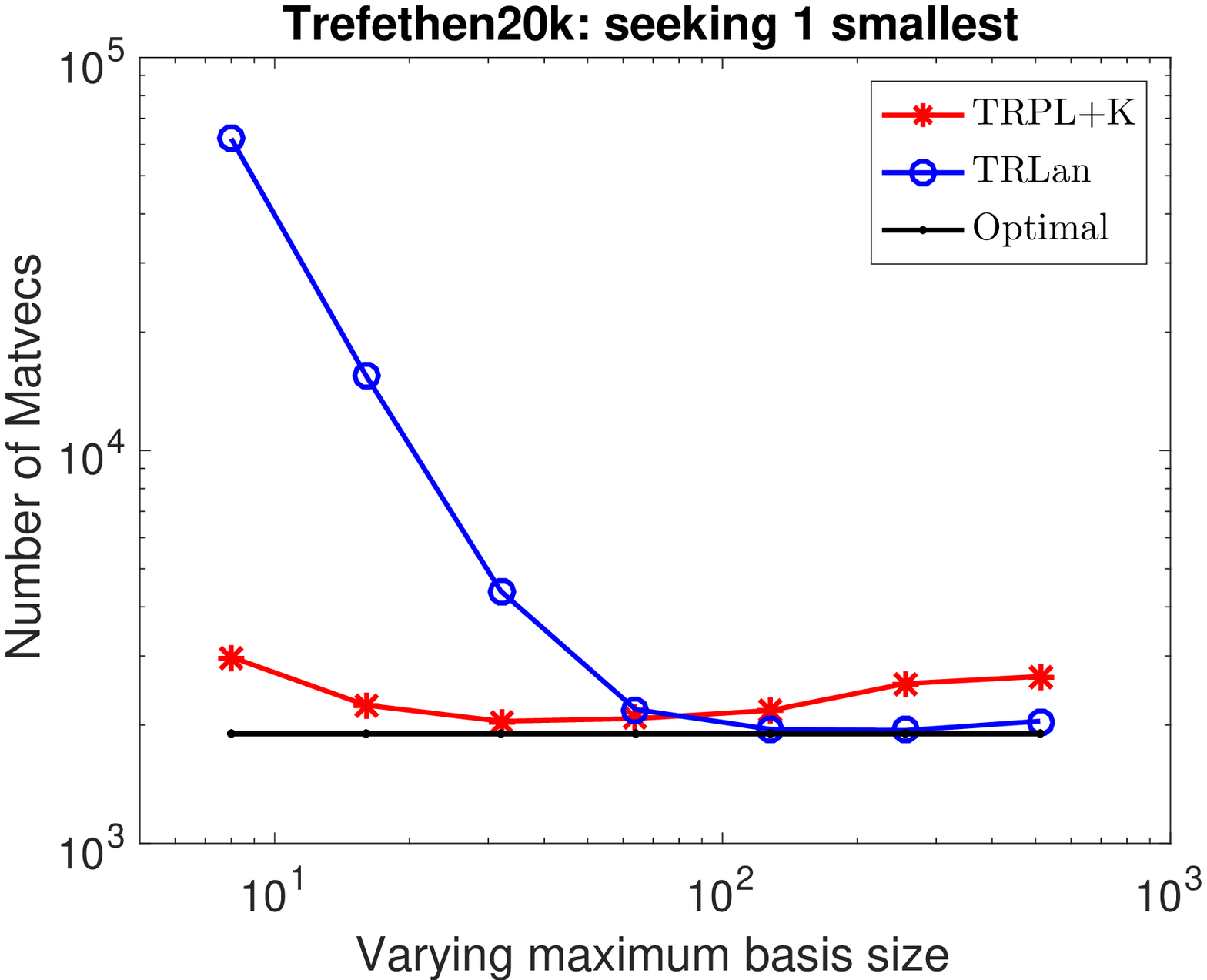}
      \caption{Trefethen\_20k}
      \label{fig:Trefethen20k_varyingmaxBsize}
     \end{subfigure}
     \caption{The effect of varying the maximum basis size $q$ on the convergence of TRPL+K and TRLan. Seeking $p=1$ smallest eigenvalue without preconditioning.}
     \label{fig:seeking1_varyingmaxBsize}
\end{figure}

\begin{figure}[ht!]
  \centering
  \begin{subfigure}[b]{0.32\textwidth}
      \includegraphics[width=\textwidth]{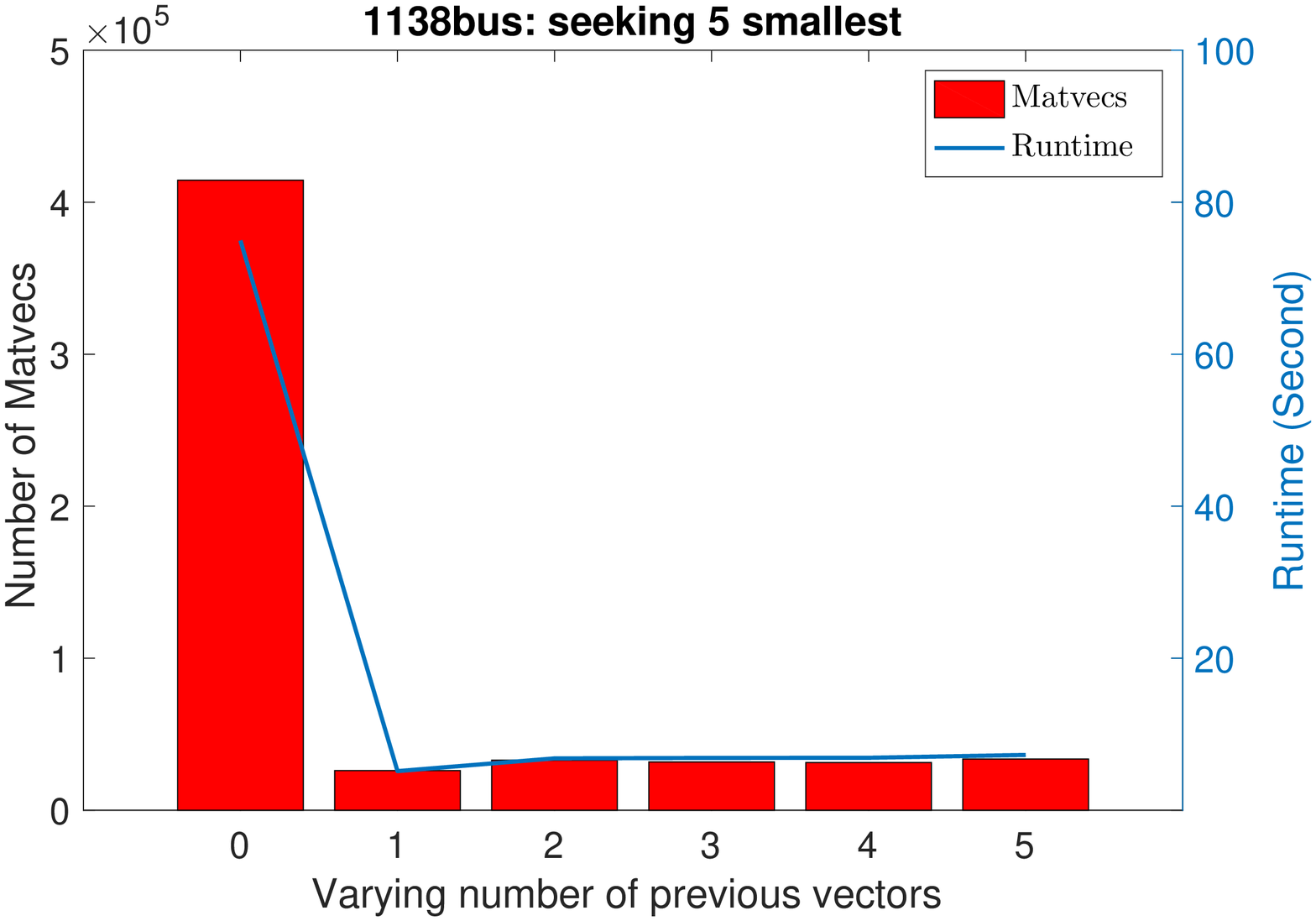}
      \caption{1138bus}
      \label{fig:1138bus_varyingmaxPsize}
    \end{subfigure}
    \begin{subfigure}[b]{0.32\textwidth}
      \includegraphics[width=\textwidth]{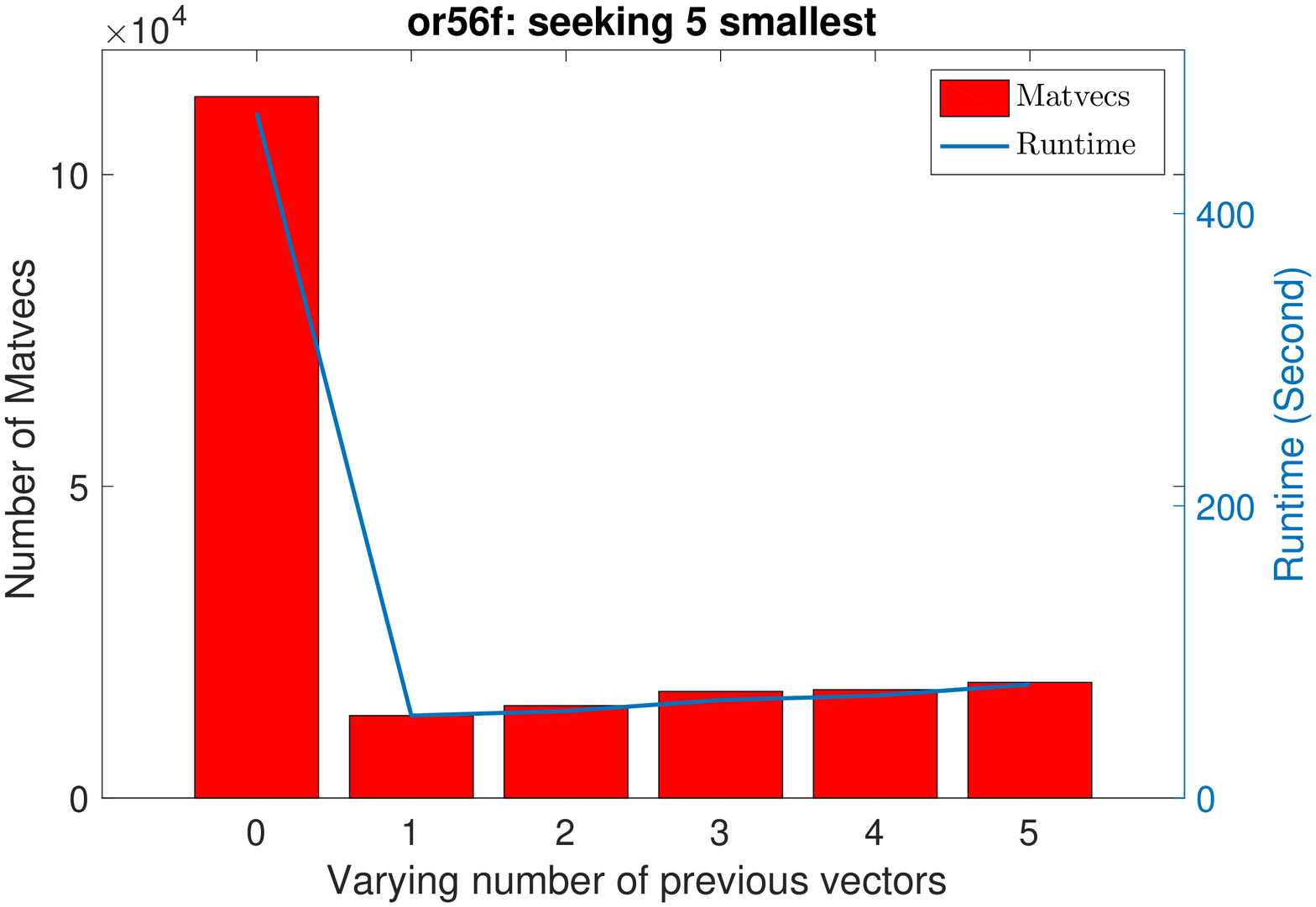}
      \caption{or56f}
      \label{fig:or56f_varyingmaxPsize}
     \end{subfigure}
     \begin{subfigure}[b]{0.32\textwidth}
      \includegraphics[width=\textwidth]{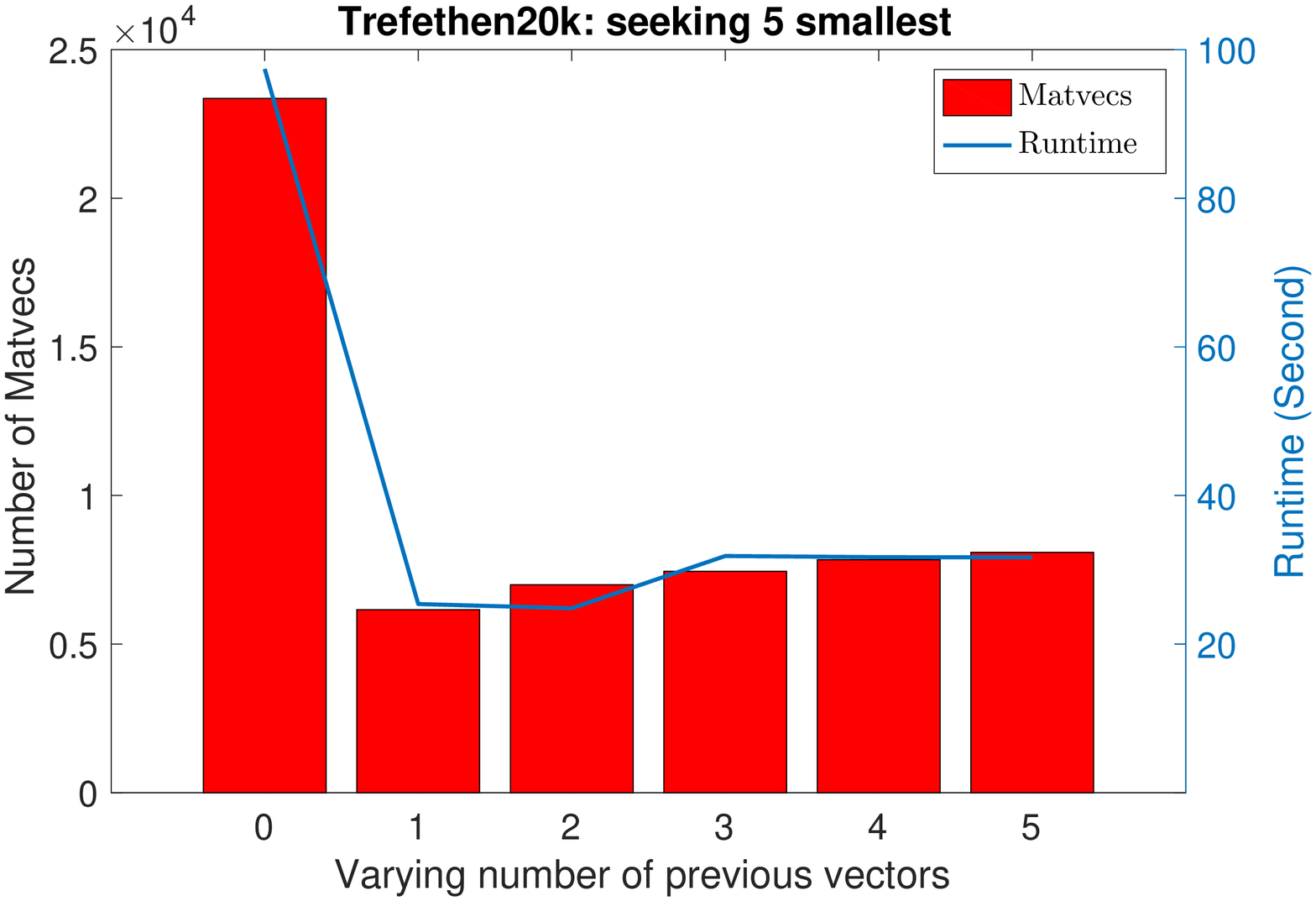}
      \caption{Trefethen\_20k}
      \label{fig:Trefethen20k_varyingmaxPsize}
     \end{subfigure}
    \caption{The effect of varying the number of previous vectors $l$ on the convergence of TRPL+K when seeking $p= 5$ smallest eigenvalues without preconditioning.}     
    \label{fig:seeking5_varyingmaxPsize}
\end{figure}

\subsection{Without preconditioning}
We compare against unrestarted Lanczos and other methods without a preconditioner
to show the near-optimal performance of TRPL+K under limited memory. Figure \ref{fig:seeking1_no_pre} shows the results on the easy case finan512 and on the hard case cfd1, seeking $p=1$ eigenvalue. 
For both cases TRPL+K and GD+K can achieve almost identical performance as unrestarted Lanczos, which is consistent with Theorem \ref{globalopt}.
GD+K requires slightly fewer iterations because its new directions come from the most recent RR over the entire space, although the cost per iteration is higher.
Compared to LOBPCG and EigIFP, TRPL+K is significantly faster because it uses thick restarting and a Krylov subspace (compared to LOBPCG) and because it uses thick and locally optimal restarting (compared to EigIFP).
The difference between TRPL+K and TRLan is relatively small for the easy case, but becomes significant with hard problems.
Figure \ref{fig:seeking5_no_pre} shows similar results when computing 5 eigenvalues. To make it easier to see, we only show graphs for TRPL+K, GD+K, and TRLan. As before, TRPL+K and GD+K are competitive and both of them substantially outperform the performance of TRLan.

\begin{figure}[ht!]
  \centering
  \begin{subfigure}[b]{0.49\textwidth}
      \includegraphics[width=\textwidth]{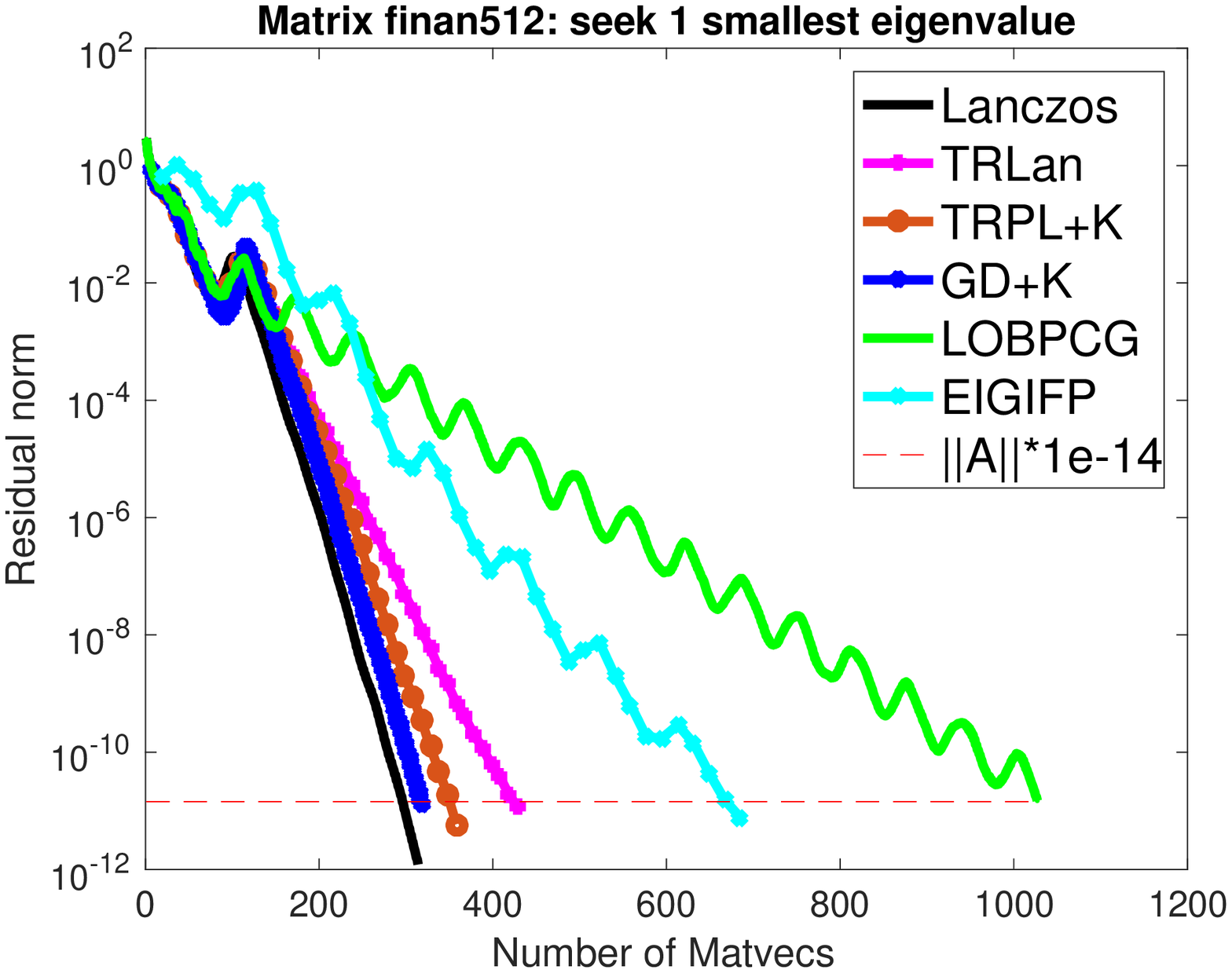}
      \caption{Matrix finan512}
      \label{fig:finan512_seeking1_no_pre}
    \end{subfigure}
    \begin{subfigure}[b]{0.49\textwidth}
      \includegraphics[width=\textwidth]{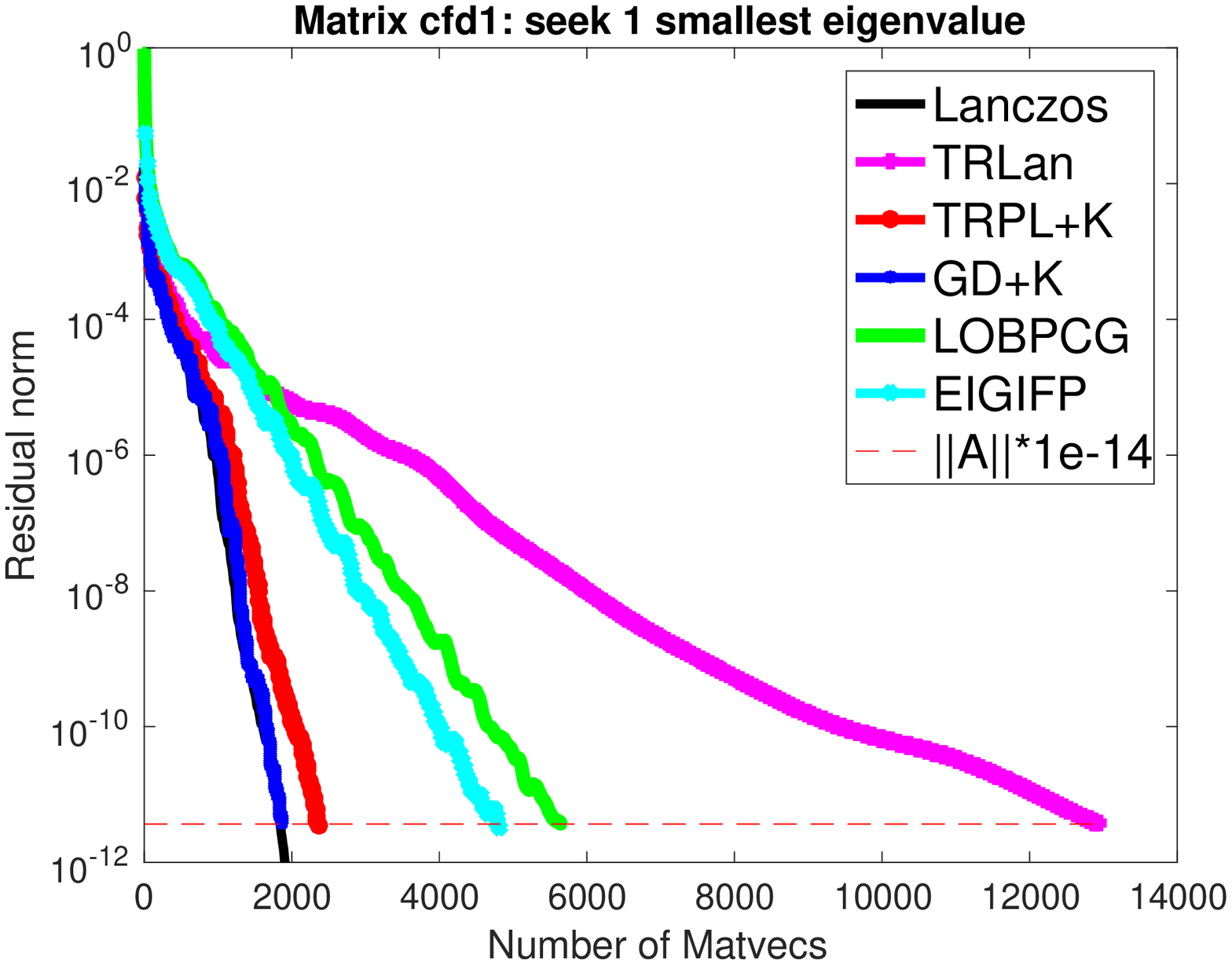}
      \caption{Matrix cfd1}
      \label{fig:cfd1_seeking1_no_pre}
     \end{subfigure}
    \caption{Compare all methods to the optimal convergence of unrestarted Lanczos when seek 1 smallest eigenvalue without preconditioning}     
    \label{fig:seeking1_no_pre}
\end{figure}

\begin{figure}[t]
  \centering
  \begin{subfigure}[b]{0.49\textwidth}
      \includegraphics[width=\textwidth]{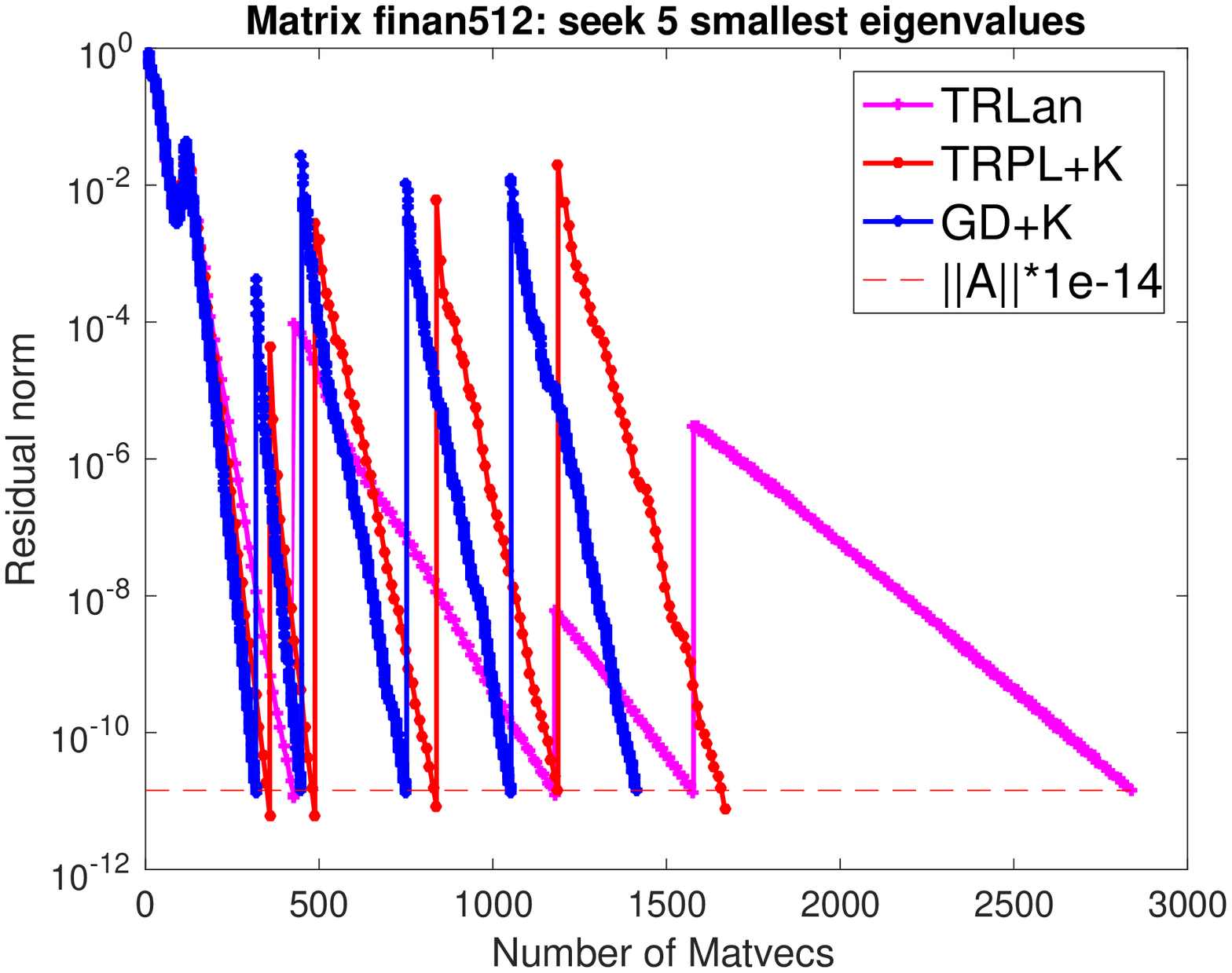}
      \caption{Matrix finan512}
      \label{fig:finan512_seeking5_no_pre}
    \end{subfigure}
    \begin{subfigure}[b]{0.49\textwidth}
      \includegraphics[width=\textwidth]{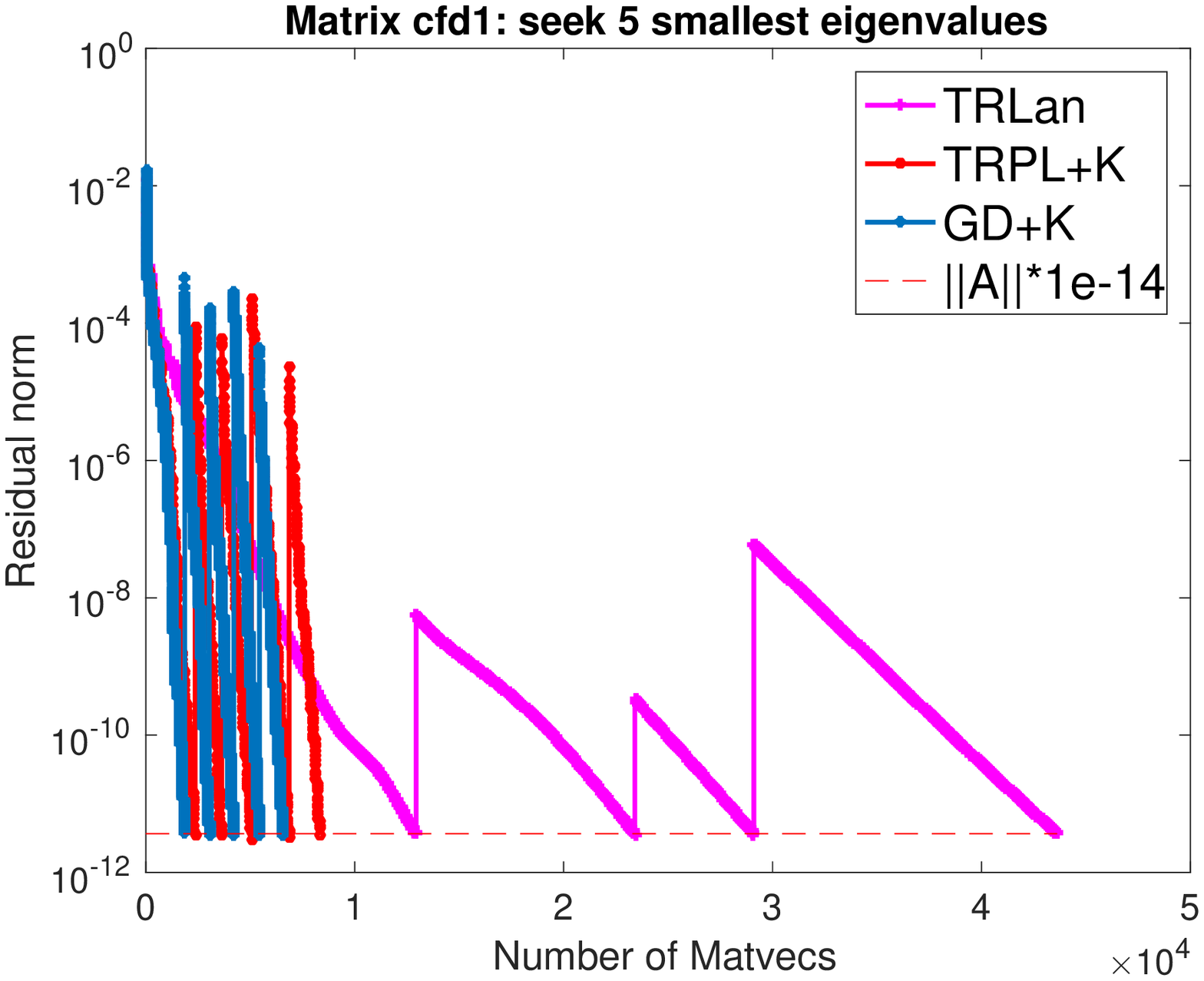}
      \caption{Matrix cfd1}
      \label{fig:cfd1_seeking5_no_pre}
     \end{subfigure}
    \caption{Compare TRPL+K and GD+K to the optimal convergence of unrestarted Lanczos when seek 5 smallest eigenvalue without preconditioning}     
    \label{fig:seeking5_no_pre}
\end{figure}

Tables \ref{tb:comp_allmethods_seek1_wo_precond} and \ref{tb:comp_allmethods_seek5_wo_precond} summarize results from all methods
seeking 1 and 5 eigenvalues respectively.
TRPL+K converges faster than all other methods in terms of matrix-vector products except for GD+K. In fact, the harder the problem the more significant the gains. It is also interesting to observe the differences between the methods which  are due mainly to the use of different algorithmic components. Among them, the locally optimal restarting has the biggest effect, followed by a larger Krylov space, while a combination of all three components in TRPL+K (and GD+K) are clearly the most beneficial while being computationally economical.
Also, TRPL+K requires about 20\% more matrix-vector products than GD+K but it is 35\% faster in runtime, making it the method of choice when the matrix-vector operator is inexpensive. This is because TRPL+K only needs to perform RR procedure once in one outer-iteration while the GD+K has to apply quite number ($m$ times) of RR in every outer-iteration. 

%\Red{AS: Is there a reason why we have not compared against JDQMR?}
%\Blue{LW: JD type methods are not closely related to TRPL+K so it is excluded from %the experiments.}
%\Red{AS: I do not think that is a valid reason but at this point I go along}

\begin{table}[htbp]
\centering
\caption{Comparing all methods for 1 smallest eigenvalue without preconditioning.}
\label{tb:comp_allmethods_seek1_wo_precond}
\footnotesize
\newcommand{\Bd}[1]{\textbf{#1}}
\newcommand{\Em}[1]{\emph{#1}}
	\begin{tabular}{ l rr rr rr rr rr }
	\hline
	\multicolumn{1}{r}{Method:}
	 & \multicolumn{2}{c}{{\tt TRPL+K}}
	 & \multicolumn{2}{c}{{\tt TRLan}}   
	 & \multicolumn{2}{c}{{\tt GD+K}} 
	 & \multicolumn{2}{c}{{\tt LOBPCG}} 
	 & \multicolumn{2}{c}{{\tt EIGIFP}} \\  \hline	 
	 Matrix & MV & Sec & MV & Sec & MV & Sec & MV & Sec & MV & Sec \\ \hline
	1138BUS & \Em{4778} & \Bd{0.9} & 48718 & 8.4 & \Bd{3662} & 1.8 & 15498 & 6.4 & 11107 & \Em{1.2} \\
	or56f & \Em{4908} & \Bd{20.2} & 64058 & 257.5 & \Bd{4024} & \Em{21.5} & 20365 & 82.1 & 19117 & 72.3 \\ 
	Tre20k & \Em{2208} & \Bd{7.1} & 12718 & 36.4 & \Bd{2058} & 10.3 & 3482 & 8.6 & 3799 & \Em{8.0} \\ 
	PlaA0 & \Em{638} & \Bd{3.5} & 1168 & 5.5 & \Bd{556} & 6.4 & 1583 & 5.4 & 1207 & \Em{3.9} \\ 
	cfd1 & \Em{2358} & \Bd{21.7} & 13008 & 106.7 & \Bd{1970} & 44.9 & 5647 & 33.7 & 4807 & \Em{29.5} \\ 
	finan512 & \Em{358} & 6.0 & 428 & 7.3 & \Bd{322} & 6.7 & 1028 & \Em{4.8} & 685 & \Bd{3.3}   \\  \hline	
    \end{tabular}
\end{table} 

\begin{table}[htbp]
\centering
\caption{Comparing all methods for 5 smallest eigenvalue without preconditioning.}
\label{tb:comp_allmethods_seek5_wo_precond}
\scriptsize
\newcommand{\Bd}[1]{\textbf{#1}}
\newcommand{\Em}[1]{\emph{#1}}
	\begin{tabular}{ l rr rr rr rr rr }
	\hline
	\multicolumn{1}{r}{Method:}
	 & \multicolumn{2}{c}{{\tt TRPL+K}}
	 & \multicolumn{2}{c}{{\tt TRLan}}   
	 & \multicolumn{2}{c}{{\tt GD+K}} 
	 & \multicolumn{2}{c}{{\tt LOBPCG}} 
	 & \multicolumn{2}{c}{{\tt EIGIFP}} \\  \hline	 
	 Matrix & MV & Sec & MV & Sec & MV & Sec & MV & Sec & MV & Sec \\ \hline
	 1138BUS & \Em{26888} & \Bd{5.1} & 414518 & 74.9 & \Bd{21318} & 11.2 & 374915 & 56.1 & 68729 & \Em{7.9} \\ 
	 or56f & \Em{13218} & \Bd{55.4} & 115098 & 507.7 & \Bd{11169} & \Em{61.2} & 40940 & 65.5 & 68801 & 260.6 \\ 
	 Tre20k & \Em{6158} & \Bd{18.1} & 23358 & 65.8 & \Bd{5520} & 35.4 & 14335 & \Em{21.8} & 25277 & 55.0 \\ 
	 PlaA0 & \Em{1858} & \Bd{10.0} & 2728 & \Em{12.7} & \Bd{1812} & 23.3 & 12545 & 22.1 & 4235 & 14.3 \\ 
	 cfd1 & \Em{8558} & \Bd{85.8} & 43648 & 386.6 & \Bd{7747} & 212.5 & 70170 & 458.6 & 20993 & \Em{132.3} \\
	 finan512 & \Em{1668} & \Em{27.2} & 2838 & 44.3 & \Bd{1501} & 43.5 & 8830 & 35.2 & 3551 & \Bd{18.2} \\ \hline
	 \end{tabular}
\end{table}

\subsection{With preconditioning}
%    \Red{AS: How do you guarantee that LU is symmetric? Is it important to EigIFP? To TRPL+K? It is not to GD+K.} \Blue{LF: Once preconditioning is used, T is almost a full matrix. With Arnoldi (used in our code and eigifp), ilu is fine to use.}
The above results emphasize the need for preconditioning, especially for problems with highly clustered eigenvalues. In this case TRLan cannot be used. In our experiments, we use MATLAB's ILU factorization on $A$ with parameters {\tt `type\,=\,nofill'}. We then compare TRPL+k against other methods with the constructed preconditioner for finding 1 and 5 smallest eigenvalues. The results are shown in Tables \ref{tb:comp_allmethods_seek1_precond} and \ref{tb:comp_allmethods_seek5_precond} respectively. 
TRPL+K is again the fastest or close to the fastest method. We note that with a good preconditioner the number of iterations decreases and thus the differences between methods are smaller. We also note that the added cost of the preconditioner per iteration is similar to having a more expensive operator which favors GD+K also in runtime.

\begin{table}[htbp]
\centering
\caption{Comparing all methods for 1 smallest eigenvalue with preconditioning.}
\label{tb:comp_allmethods_seek1_precond}
\small
\newcommand{\Bd}[1]{\textbf{#1}}
\newcommand{\Em}[1]{\emph{#1}}
	\begin{tabular}{ l rr rr rr rr}
	\hline
	\multicolumn{1}{r}{Method:}
	 & \multicolumn{2}{c}{{\tt TRPL+K}}
	 & \multicolumn{2}{c}{{\tt GD+K}} 
	 & \multicolumn{2}{c}{{\tt LOBPCG}} 
	 & \multicolumn{2}{c}{{\tt EIGIFP}} \\  \hline	 
	 Matrix & MV & Sec & MV & Sec & MV & Sec & MV & Sec \\ \hline
	 1138BUS & \Em{328} & \Bd{0.2} & \Bd{174} & \Bd{0.2} & 474 & \Em{0.3} & 397 & \Bd{0.2} \\
	 or56f & \Em{158} & \Em{1.5} & \Bd{81} & \Bd{0.8} & 335 & 2.7 & 325 & 2.7 \\
	 Tre20k & 38 & \Em{0.2} & \Em{17} & \Bd{0.1} & \Bd{10} & \Bd{0.1} & 55 & 0.3 \\
	 cfd1 & \Em{838} & 20.0 & \Bd{562} & \Bd{14.0} & 2519 & 26.5 & 1279 &  \Em{17.2} \\
	 finan512 & \Em{98} & 1.9 & \Bd{69} & 1.5 & 144 & \Bd{1.0} & 127 & \Em{1.4} \\ \hline
	 \end{tabular}
\end{table} 

\begin{table}[htbp]
\centering
\caption{Comparing all methods for 5 smallest eigenvalues with preconditioning.}
\label{tb:comp_allmethods_seek5_precond}
\small
\newcommand{\Bd}[1]{\textbf{#1}}
\newcommand{\Em}[1]{\emph{#1}}
    \begin{tabular}{ l rr rr rr rr}
	\hline
	\multicolumn{1}{r}{Method:}
	 & \multicolumn{2}{c}{{\tt TRPL+K}}
	 & \multicolumn{2}{c}{{\tt GD+K}} 
	 & \multicolumn{2}{c}{{\tt LOBPCG}} 
	 & \multicolumn{2}{c}{{\tt EIGIFP}} \\  \hline	 
	 Matrix & MV & Sec & MV & Sec & MV & Sec & MV & Sec \\ \hline
	 1138BUS & \Em{1128} & \Bd{0.3} & \Bd{779} & \Em{0.5} & 2120 & \Em{0.5} & 2795 & 0.7 \\
	 or56f & \Em{368} & 3.3 & \Bd{223} & \Em{2.4} & 375 & \Bd{2.3} & 1175 & 10.1 \\
	 Tre20k & 118 & 0.7 & \Bd{45} & \Em{0.4} & \Em{115} & \Bd{0.3} & 203 & 1.0 \\
	 cfd1 & \Em{2598} & \Em{67.1} & \Bd{1868} & \Bd{56.7} & 12900 & 89.7 & 5135 & 72.8 \\
	 finan512 & \Em{378} & 8.6 & \Bd{289} & 9.0 & 1225 & \Bd{5.2} & 779 & \Em{8.5} \\ \hline
	 \end{tabular}
\end{table} 

\begin{comment}\Red{I noticed that as the number of desired eigenpairs increases from 1 to 5, the \emph{relative} advantage of GD+k over TRPL+k decreases. Should we try 10 or even 20 eigenvalues as well for the standard eigenproblems, as we did for generalized ones? Could this finally makes TRPL+K fastest?}\Blue{AS: we could but I am not sure we have the space. What I think will happen for large p is that the two methods will reach a constant ratio which is smaller than the large ratios we see for p = 1 (which is helped by the fact that GD+1 is near optimal for 1 but then it repeats the same convergence for each subsequent one).  TRPL+1 may not be as "optimal" but it converges to many eigenvalues at good rate. I do not think we'll see a performance inversion for large p}.
\Magenta{LW: with preconditioning, the situation is quite trick. First of all, GD+K applies more frequent preconditioning which is expected to be more beneficial than incrementally increasing the subspace of TRPL+K. But it really depends on the effectiveness of a preconditioner. Since the effectiveness of the preconditioner becomes less effective when seeking more eigenvalues, the cost of applying preconditioner starts becoming more expensive. It explains why the advantages of GD+K over TRPL+K diminishes when seeking more. In general, when the cross line that will happen really depends but we can image eventually we should observe similar results with the cases without preconditioning. }
\end{comment}

\subsection{Generalized eigenvalue problem}

\begin{table}[ht!]
\centering
\caption{Properties of the test matrices for generalized eigenvalue problems.}
\label{tb:GeneralizedMatrix-info}
\small
\begin{center}
    \begin{tabular}{ c c c c c }  \hline
    Matrix & order & nnz(A) & $\kappa(A)$ & Source \\ \hline
    bcsstk23 & 3134 & 45178 & 6.9e+12 & MM \\ 
    bcsstm23 & 3134 & 3134 & 9.4e+08 & MM \\
    bcsstk24 & 3562 & 159910 & 6.3e+11 & MM \\ 
    bcsstm24 & 3562 & 3562 & 1.8e+13 & MM \\
    bcsstk25 & 15439 & 252241 & 1.2e+13 & MM \\ 
    bcsstm25 & 15439 & 15439 & 6.0e+09 & MM \\ \hline
    \end{tabular}
\end{center}
\end{table} 

We perform some sample experiments on generalized eigenvalue problems, comparing the proposed method against LOBPCG and EigIFP. As shown in Table \ref{tb:GeneralizedMatrix-info}, the condition numbers of these problems are quite large, making preconditioning necessary to accelerate the convergence. In this experiment, we use Incomplete LDL factorization \cite{SYM-ILDL} on $A$ with {\tt droptol = 1e-6, 1e-8} for bcsstkm23 and bcsstkm24 respectively, and MATLAB's LDL factorization on $A$ with parameter {\tt THRESH = 0.5} for bcsstkm25. We then use the preconditioner to find 1, 5, and 10 smallest eigenvalues. As shown in Table \ref{tb:comp_allmethods_seek1510_gen_precond}, TRPL+K significantly outperforms other methods in terms of the number of matrix-vector operations for these cases. Note that LOGPCG is quite efficient in terms of runtime, thanks to its efficiently implemented block operations in MATLAB, especially for small matrices that fit in cache. A block TRPL+K is also possible, which is left for future work. 
\begin{comment}

\Red{Interestingly, LOBPCG takes most time for bcsstkm25 when only 1 eigenvalue is sought, and least time if $10$ eigenvalues are wanted. \Blue{AS: For the timings, it is solely the effect of block operations being more effective. However, for BCS24 LOBPCG behavaes crazily for p = 5. Lingfei please check..} \Magenta{LW: lobpcg has convergence issue for seeking 5 on this matrix, where a long plateau is observed and I am not sure what caused it.} Meanwhile, our TRPL+k takes essentially the same cost for bcsstkm25 no matter if $1,5$ or $10$ eigenvalues are desired. \Blue{AS: I do not think this is normal behavior for TRPL+1. Lingfei, can you double check that these results are correct}. EIGIFP behaves most `normally' for this problem. Do we have any intuitive explanation of this observation?}\Blue{I think the spectrum could reveal some info. Unfortunately, Table 5.6 gives condition numbers which is not as relevant... }
\Magenta{LW: yes, i have made some mistakes for bccstm25 and have corrected them. thanks for pointing out.}
\end{comment}

\begin{table}[htbp]
\centering
\caption{Comparing all methods for computing 1, 5, 10 smallest eigenvalues with preconditioning (numbers in parenthesis). Each pencil bcsstkmXY has matrix pairs (bcsstkXY, bcsstmXY), where XY $\in \{ 23,24, 25\}.$}
\label{tb:comp_allmethods_seek1510_gen_precond}
\small
\newcommand{\Bd}[1]{\textbf{#1}}
\newcommand{\Em}[1]{\emph{#1}}
    \begin{tabular}{ l rr rr rr}
	\hline
	\multicolumn{1}{r}{Method:}
	 & \multicolumn{2}{c}{{\tt TRPL+K}}
	 & \multicolumn{2}{c}{{\tt LOBPCG}} 
	 & \multicolumn{2}{c}{{\tt EIGIFP}} \\  \hline	 
	 Matrix & MV & Sec & MV & Sec & MV & Sec \\ \hline
	 bcsstkm23(1) & \Bd{60} & \Bd{0.1} & 134 & 0.2 & 127 & 0.2 \\
	 bcsstkm24(1) & \Bd{29} & \Bd{0.1} & 55 & \Bd{0.1} & 73 & 0.2\\
	 bcsstkm25(1) & \Bd{2816} & \Bd{29.1} & 28355 & 178.9 & 9721 & 66.2\\\hline 
	 bcsstkm23(5) & \Bd{150} & 0.4 & 190 & \Bd{0.3} & 419 & 0.7\\
	 bcsstkm24(5) & \Bd{120} & \Bd{0.6} & 1020 & 3.6 & 257 & 0.8  \\
	 bcsstkm25(5) & \Bd{3830} & \Bd{39.6} & 39895 & 89.1 & 14423 & 97.6\\\hline
	 bcsstkm23(10) & \Bd{210} & 0.4 & 240 & \Bd{0.2} & 820 & 1.0 \\
	 bcsstkm24(10) & \Bd{210} & 0.9 & 270 & \Bd{0.6} & 496 & 1.4 \\
	 bcsstkm25(10) & \Bd{5120} & \Bd{52.9} & 45820 & 63.7 & 20494 & 141.7 \\\hline
	 \end{tabular}
\end{table}

%%%%%%%%%%%%%%%%%%%%%%%%%%%%%%%%%%%%%%%%%%%%%%%%%%%%%%%%%%%%%%%%%%%
\section{Conclusions and future work}
\label{sec:conclusion}
We presented a new near-optimal eigenmethod, thick-restart preconditioned Lanczos +K method (TRPL+K), which is based on three key algorithmic components: thick restarting, locally optimal restarting, and the ability to build a preconditioned Krylov space. 
We provided a proof of an asymptotic global quasi-optimality of the proposed method and provided insights on the near-optimal performance of a group of methods that employ locally optimal restarting. Our extensive experiments demonstrate that TRPL+K either outperforms or matches other state-of-the-art methods in terms of both number of matrix-vector operations and computational time. 

An interesting future direction is to extend this approach to the Lanczos Bidiagonalization method for singular value problems.

%%%%%%%%%%%%%%%%%%%%%%%%%%%%%%%%%%%%%%%%%%%%%%%%%%%%%%%%%%%%%%%%%%%
\section*{Acknowledgment}
The authors are indebted to the referees for their meticulous reading and constructive comments. This work is supported by NSF under grant No. ACI SI2-SSE 1440700, DMS-1719461, and by DOE under a grant No. DE-FC02-12ER41890.

%\clearpage
\bibliographystyle{siam}
\bibliography{trpl_k}

\begin{thebibliography}{10}

\bibitem{baglama2005augmented}
{\sc James Baglama and Lothar Reichel}, {\em Augmented implicitly restarted
  {Lanczos} bidiagonalization methods}, SIAM J. Sci. Comput., 27 (2005),
  pp.~19--42.

\bibitem{calvetti1994implicitly}
{\sc Daniela Calvetti, L~Reichel, and Danny~Chris Sorensen}, {\em An implicitly
  restarted {L}anczos method for large symmetric eigenvalue problems},
  Electronic Transactions on Numerical Analysis, 2 (1994), p.~21.

\bibitem{chapman1997deflated}
{\sc Andrew Chapman and Yousef Saad}, {\em Deflated and augmented {Krylov}
  subspace techniques}, Numerical linear algebra with applications, 4 (1997),
  pp.~43--66.

\bibitem{crouzeix1994davidson}
{\sc Michel Crouzeix, Bernard Philippe, and Miloud Sadkane}, {\em The
  {Davidson} method}, SIAM J. Sci. Comput., 15 (1994), pp.~62--76.

\bibitem{davidson1975iterative}
{\sc Ernest~R Davidson}, {\em The iterative calculation of a few of the lowest
  eigenvalues and corresponding eigenvectors of large real-symmetric matrices},
  Journal of Computational Physics, 17 (1975), pp.~87--94.

\bibitem{Sturler_GCR}
{\sc E.~de~Sturler}, {\em Nested {K}rylov methods based on {GCR}}, Journal of
  Computational and Applied Mathematics, 67 (1996), pp.~15--41.

\bibitem{Sturler_trunc}
\leavevmode\vrule height 2pt depth -1.6pt width 23pt, {\em Truncation
  strategies for optimal {K}rylov subspace methods}, SIAM J. Matrix Anal.
  Appl., 36 (1999), pp.~864--889.

\bibitem{EdelmanAriasSmith}
{\sc A.~Edelman, T.~A. Arias, and S.~T. Smith}, {\em The geometry of algorithms
  with orthogonality constraints}, SIAM Journal on Matrix Analysis and
  Applications, 20 (1998), pp.~303--353.

\bibitem{Golub1996MC}
{\sc Gene~H. Golub and Charles~F. Van~Loan}, {\em Matrix Computations (3rd
  Ed.)}, Johns Hopkins University Press, Baltimore, MD, USA, 1996.

\bibitem{golub2002inverse}
{\sc Gene~H Golub and Qiang Ye}, {\em An inverse free preconditioned {Krylov}
  subspace method for symmetric generalized eigenvalue problems}, SIAM J. Sci.
  Comput., 24 (2002), pp.~312--334.

\bibitem{SYM-ILDL}
{\sc Chen Greif, Shiwen He, and Paul Liu}, {\em {SYM-ILDL:} incomplete
  $ldl^\top$ factorization of symmetric indefinite and skew-symmetric
  matrices}, CoRR, abs/1505.07589 (2015).

\bibitem{Han2011Data}
{\sc Jiawei Han, Micheline Kamber, and Jian Pei}, {\em Data Mining: Concepts
  and Techniques}, Morgan Kaufmann Publishers Inc., San Francisco, CA, USA,
  3rd~ed., 2011.

\bibitem{hernandez2006lanczos}
{\sc V~Hernandez, JE~Roman, A~Tomas, and V~Vidal}, {\em Lanczos methods in
  {SLEPc}}, Universidad Polit{\'e}cnica de Valencia, Valencia, Spain, SLEPc
  Technical Report STR-5,  (2006), p.~136.

\bibitem{hsieh2014nuclear}
{\sc Cho-Jui Hsieh and Peder Olsen}, {\em Nuclear norm minimization via active
  subspace selection}, in Proceedings of the 31st International Conference on
  Machine Learning (ICML-14), 2014, pp.~575--583.

\bibitem{kalantzis2016spectral}
{\sc Vassilis Kalantzis, Ruipeng Li, and Yousef Saad}, {\em Spectral {S}chur
  complement techniques for symmetric eigenvalue problems}, Electronic
  Transactions on Numerical Analysis, 45 (2016), pp.~305--329.

\bibitem{knyazev2001toward}
{\sc Andrew~V Knyazev}, {\em Toward the optimal preconditioned eigensolver:
  Locally optimal block preconditioned conjugate gradient method}, SIAM J. Sci.
  Comput., 23 (2001), pp.~517--541.

\bibitem{knyazev2007block}
{\sc Andrew~V Knyazev, Merico~E Argentati, Ilya Lashuk, and Evgueni~E
  Ovtchinnikov}, {\em Block locally optimal preconditioned eigenvalue xolvers
  (blopex) in hypre and petsc}, SIAM J. Sci. Comput., 29 (2007),
  pp.~2224--2239.

\bibitem{lanczos1950iteration}
{\sc Cornelius Lanczos}, {\em An iteration method for the solution of the
  eigenvalue problem of linear differential and integral operators}, 45 (1950),
  pp.~255--282.

\bibitem{li2016thick}
{\sc Ruipeng Li, Yuanzhe Xi, Eugene Vecharynski, Chao Yang, and Yousef Saad},
  {\em A thick-restart {L}anczos algorithm with polynomial filtering for
  {H}ermitian eigenvalue problems}, SIAM J. Sci. Comput., 38 (2016),
  pp.~A2512--A2534.

\bibitem{liu2013limited}
{\sc Xin Liu, Zaiwen Wen, and Yin Zhang}, {\em Limited memory block {Krylov}
  subspace optimization for computing dominant singular value decompositions},
  SIAM J. Sci. Comput., 35 (2013), pp.~A1641--A1668.

\bibitem{morgan1996restarting}
{\sc Ronald Morgan}, {\em On restarting the {Arnoldi} method for large
  nonsymmetric eigenvalue problems}, Mathematics of Computation of the American
  Mathematical Society, 65 (1996), pp.~1213--1230.

\bibitem{morgan1986generalizations}
{\sc Ronald~B Morgan and David~S Scott}, {\em Generalizations of {D}avidson’s
  method for computing eigenvalues of sparse symmetric matrices}, SIAM Journal
  on Scientific and Statistical Computing, 7 (1986), pp.~817--825.

\bibitem{morgan1993preconditioning}
\leavevmode\vrule height 2pt depth -1.6pt width 23pt, {\em Preconditioning the
  {L}anczos algorithm for sparse symmetric eigenvalue problems}, SIAM J. Sci.
  Comput., 14 (1993), pp.~585--593.

\bibitem{murray1992improved}
{\sc Christopher~W Murray, Stephen~C Racine, and Ernest~R Davidson}, {\em
  Improved algorithms for the lowest few eigenvalues and associated
  eigenvectors of large matrices}, Journal of Computational Physics, 103
  (1992), pp.~382--389.

\bibitem{olsen1990passing}
{\sc Jeppe Olsen, Poul J{\o}rgensen, and Jack Simons}, {\em Passing the
  one-billion limit in full configuration-interaction ({FCI}) calculations},
  Chemical Physics Letters, 169 (1990), pp.~463--472.

\bibitem{paige1972computational}
{\sc Christopher~C Paige}, {\em Computational variants of the {L}anczos method
  for the eigenproblem}, IMA Journal of Applied Mathematics, 10 (1972),
  pp.~373--381.

\bibitem{paige1976error}
\leavevmode\vrule height 2pt depth -1.6pt width 23pt, {\em Error analysis of
  the {L}anczos algorithm for tridiagonalizing a symmetric matrix}, IMA Journal
  of Applied Mathematics, 18 (1976), pp.~341--349.

\bibitem{paige1980accuracy}
{\sc Chris~C Paige}, {\em Accuracy and effectiveness of the {L}anczos algorithm
  for the symmetric eigenproblem}, Linear algebra and its applications, 34
  (1980), pp.~235--258.

\bibitem{parlett1980symmetric}
{\sc Beresford~N. Parlett}, {\em The Symmetric Eigenvalue Problem},
  Prentice-Hall, 1980.

\bibitem{saad1992numerical}
{\sc Youcef Saad}, {\em Numerical methods for large eigenvalue problems},
  Manchester University Press, 1992.

\bibitem{sleijpen2000jacobi}
{\sc Gerard~LG Sleijpen and Henk~A Van~der Vorst}, {\em A {J}acobi--{D}avidson
  iteration method for linear eigenvalue problems}, SIAM review, 42 (2000),
  pp.~267--293.

\bibitem{sorensen1992implicit}
{\sc Danny~C Sorensen}, {\em Implicit application of polynomial filters in
  ak-step {Arnoldi} method}, Siam journal on matrix analysis and applications,
  13 (1992), pp.~357--385.

\bibitem{stathopoulos2007nearlyI}
{\sc Andreas Stathopoulos}, {\em Nearly optimal preconditioned methods for
  {H}ermitian eigenproblems under limited memory. part {I}: Seeking one
  eigenvalue}, SIAM J. Sci. Comput., 29 (2007), pp.~481--514.

\bibitem{stathopoulos1994davidson}
{\sc Andreas Stathopoulos and Charlotte~F Fischer}, {\em A {Davidson} program
  for finding a few selected extreme eigenpairs of a large, sparse, real,
  symmetric matrix}, Computer Physics Communications, 79 (1994), pp.~268--290.

\bibitem{stathopoulos1998restarting}
{\sc Andreas Stathopoulos and Yousef Saad}, {\em Restarting techniques for the
  ({J}acobi-) {Davidson} symmetric eigenvalue methods}, Electron. Trans. Numer.
  Anal, 7 (1998), pp.~163--181.

\bibitem{stathopoulos1998dynamic}
{\sc Andreas Stathopoulos, Yousef Saad, and Kesheng Wu}, {\em Dynamic thick
  restarting of the {Davidson}, and the implicitly restarted {Arnoldi}
  methods}, SIAM J. Sci. Comput., 19 (1998), pp.~227--245.

\bibitem{Szyld.Xue.2016}
{\sc D.~B. Szyld and F.~Xue}, {\em Preconditioned eigensolvers for large-scale
  nonlinear {H}ermitian eigenproblems with variational characterizations. {I}.
  extreme eigenvalues}, Math. Comp., 85 (2016), pp.~2887--2918.

\bibitem{van2002computational}
{\sc Henk~A van~der Vorst}, {\em Computational Methods For Large Eigenvalue
  Problems}, vol.~8, Elsevier, 2002.

\bibitem{vecharynski2015projected}
{\sc Eugene Vecharynski, Chao Yang, and John~E Pask}, {\em A projected
  preconditioned conjugate gradient algorithm for computing many extreme
  eigenpairs of a hermitian matrix}, Journal of Computational Physics, 290
  (2015), pp.~73--89.

\bibitem{EV_CY_FX_16}
{\sc Eugene Vecharynski, Chao Yang, and Fei Xue}, {\em Generalized
  preconditioned locally harmonic residual method for non-{H}ermitian
  eigenproblems}, SIAM J. Sci. Comput., 38 (2016), pp.~A500--A527.

\bibitem{vomel2008state}
{\sc Christof V{\"o}mel, Stanimire~Z Tomov, Osni~A Marques, Andrew Canning,
  Lin-Wang Wang, and Jack~J Dongarra}, {\em State-of-the-art eigensolvers for
  electronic structure calculations of large scale nano-systems}, Journal of
  Computational Physics, 227 (2008), pp.~7113--7124.

\bibitem{wu2000thick}
{\sc Kesheng Wu and Horst Simon}, {\em Thick-restart {L}anczos method for large
  symmetric eigenvalue problems}, SIAM Journal on Matrix Analysis and
  Applications, 22 (2000), pp.~602--616.

\bibitem{wu2016estimating}
{\sc Lingfei Wu, Jesse Laeuchli, Vassilis Kalantzis, Andreas Stathopoulos, and
  Efstratios Gallopoulos}, {\em Estimating the trace of the matrix inverse by
  interpolating from the diagonal of an approximate inverse}, Journal of
  Computational Physics, 326 (2016), pp.~828--844.

\bibitem{wu2016primme_svds}
{\sc Lingfei Wu, Eloy Romero, and Andreas Stathopoulos}, {\em {PRIMME\_SVDS}: A
  high-performance preconditioned svd solver for accurate large-scale
  computations}, SIAM J. Sci. Comput., 39 (2017), pp.~S248--S271.

\bibitem{wu2015preconditioned}
{\sc Lingfei Wu and Andreas Stathopoulos}, {\em A preconditioned hybrid svd
  method for accurately computing singular triplets of large matrices}, SIAM J.
  Sci. Comput., 37 (2015), pp.~S365--S388.

\bibitem{wu2016towards}
{\sc Lingfei Wu, Kesheng~John Wu, Alex Sim, Michael Churchill, Jong~Y Choi,
  Andreas Stathopoulos, Choong-Seock Chang, and Scott Klasky}, {\em Towards
  real-time detection and tracking of spatio-temporal features: Blob-filaments
  in fusion plasma}, IEEE Transactions on Big Data, 2 (2016), pp.~262--275.

\bibitem{yang2005solving}
{\sc Chao Yang}, {\em Solving large-scale eigenvalue problems in {SciDAC}
  applications}, in Journal of Physics: Conference Series, vol.~16, IOP
  Publishing, 2005, p.~425.

\end{thebibliography}

%%%%%%%%%%%%%%%%%%%%%%%%%%%%%%%%%%%%%%%%%%%%%%%%%%%%%%%%%%%%%%%%%%%
%\clearpage
\section{Appendix A: }

\medskip

{\bf Proof of Lemma \ref{lemmaoptalpha}.}
\begin{proof}By definition, the optimal step size $\alpha^*$ is
\begin{eqnarray}
\alpha^* = \arg\min_{\alpha > 0} \rho(x+\alpha p)= \arg\min_{\alpha > 0} \frac{(x+\alpha p)^T A (x+\alpha p)}{(x+\alpha p)^T B (x+ \alpha p)},
\end{eqnarray}
which can be found by letting $\frac{d}{d \alpha}\frac{1}{2}\rho(x+\alpha p)=0$. Using the quotient rule for differentiation, with some algebraic work, we have
\begin{eqnarray}
&&\frac{d}{d\alpha}\frac{\rho(x+\alpha p)}{2} = \frac{a(x,p)\alpha^2+b(x,p)\alpha + c(x,p)}{[(x+\alpha p)^TB(x+\alpha p)]^2}, \\ \nonumber
\end{eqnarray}
where $a(x,p)$, $b(x,p)$ and $c(x,p)$ are given in \eqref{coefabc}. For the sake of simplicity, we refer to these coefficients as $a, b$ and $c$, respectively, when there is no danger of confusion. Depending on the sign of $a$, there are three cases for the solutions of $\frac{d}{d \alpha} \rho(x+\alpha p) = 0$:
\begin{enumerate}
\item $a = 0$. Obviously, there is a unique solution $\alpha^* = -\dfrac{c}{b}>0$. 
\item $a > 0$. Since $c < 0$, there is a unique positive root $\alpha^* = \frac{1}{2a}\left(-b+\sqrt{b^2-4ac}\right)$. By assumption, $b \geq \|x\|^2_B \|p\|^2_B \delta > 0$, and $|c| \leq \|x\|^2_B \|p\| \|Ax - \rho(x)Bx\| = \mathcal{O}(\sin\theta)$ for sufficiently small $\theta$. By the Taylor expansion of $\sqrt{1-t}$ for $|t| < 1$, $$\alpha^* = \frac{1}{2a}\left(-b+b\sqrt{1-\frac{4ac}{b^2}}\right)=-\frac{c}{b}\left(1+\frac{ac}{b^2}+\frac{2a^2c^2}{b^4} +\cdots \right),$$which is slightly smaller than $\alpha^*$ in case 1 because $a>0$ and $c<0$.
\item $a < 0$. In this case, there must be two distinct solutions $0 < \alpha_1^* < \alpha_2^*$ such that $\frac{d}{d \alpha} \rho(x+\alpha p) < 0$ for $0 \leq \alpha < \alpha_1^*$ or $\alpha > \alpha_2^*$, and $\frac{d}{d \alpha} \rho(x+\alpha p) > 0$ for $\alpha_1^* < \alpha < \alpha_2^*$. In fact, if there is no solution or only one repeated solution, then $\frac{d}{d \alpha}\rho(x+\alpha p) \leq 0$ for all $\alpha \geq 0$, and hence $\rho(p) = \lim_{\alpha \rightarrow +\infty} \rho(x+\alpha p) \leq \rho(x)$, contradicting our assumption that $\rho(p) - \rho(x) \geq \delta > 0$. Given these intervals of monotonicity, $\rho(x+\alpha_2^* p) > \lim_{\alpha \rightarrow +\infty} \rho(x+\alpha p) = \rho(p) \geq \rho(x) + \delta$. Hence the minimizer of $\rho(x+\alpha p)$ is achieved at $\alpha_1^*$ with same expression as in case 2. The optimal step size is slightly greater than $\alpha^*=-c/b$ in case 1 as $a,c<0$. We still refer to $\alpha_1^*$ as $\alpha^*$ for notation consistency.
\end{enumerate}

In summary, $\rho(x+\alpha p)$ decreases monotonically on $(0,\alpha^*)$, then increases on $(\alpha^*,+\infty)$ (case 1 and 2) or on $(\alpha^*,\alpha_2^*)$ (case 3). The optimal $\alpha^*$ has a closed form. \qed
\end{proof}

\medskip
 {\bf Proof of Lemma \ref{rhodecrease}.}
\begin{proof}
We note that the denominator of $\frac{d}{d \alpha} \rho(x+\alpha p)$, namely, $(x+\alpha p)^TB(x+\alpha p) = \|p\|^2_B\alpha^2+2p^TBx \alpha + \|x\|^2_B > 0$ for all $\alpha$, and it is a quadratic with positive quadratic term coefficient. Hence, $\max_{0\leq \alpha \leq \alpha^*}(x+\alpha p)^TB(x+\alpha p) = \max_{\alpha =\{0, \alpha^*\}} (x+\alpha p)^TB(x+\alpha p) = \max\{\|x\|^2_B,\|x+\alpha^* p\|^2_B\}$, where $\alpha^*=-\frac{c}{b}$ or $-\frac{c}{b}(1+\frac{ac}{b^2}+\ldots)$; see Lemma \ref{lemmaoptalpha}. For sufficiently small $\theta = \angle(x,v_1)_B$, since $|c| \leq \mathcal{O}(\sin\theta)$, there is a small constant $\eta \in (0,1)$ independent of $\theta$, such that
\begin{eqnarray}\nonumber
&&\rho(x+\alpha^*p) -\rho(x) = \int_{0}^{\alpha^*}\frac{d}{d \alpha}\rho(x+\alpha p) d \alpha = \int_{0}^{\alpha^*}\frac{a\alpha^2+b\alpha+c}{(x+\alpha p)^TB(x+\alpha p)}d\alpha \\ \nonumber
&\leq& \frac{1}{\max\{\|x\|^2_B,\|x+\alpha^* p\|^2_B\}}\int_{0}^{\alpha^*}(a\alpha^2+b\alpha+c) d\alpha = \frac{\frac{1}{3}a(\alpha^*)^3+\frac{1}{2}b(\alpha^*)^2+c\alpha^*}{\max\{\|x\|^2_B,\|x+\alpha^* p\|^2_B\}} \\ \nonumber
&=& \frac{-\frac{c^2}{2b}-\frac{ac^3}{3b^3}+\mathcal{O}(c^4)}{\max\{\|x\|^2_B,\|x+\alpha^* p\|^2_B\}} \leq -\frac{c^2(1-\eta)}{2b\max\{\|x\|^2_B,\|x+\alpha^* p\|^2_B\}},
\end{eqnarray}
where $b \geq \|x\|_B^2 \|p\|_B^2 \delta > 0$ is bounded away from zero; see \eqref{coefabc}.

To establish the lower bound, note that our assumption that $\rho(p)-\rho(x) \geq \delta > 0$ for a constant $\delta$ independent of $\theta = \angle(x,v_1)_B$ means that $p$ is bounded away from $x$ in direction. Therefore, $\|x+\alpha^*p\|_B$ is bounded away from zero. Similarly, for sufficiently small $\theta$, we also have
\begin{eqnarray}
&&\rho(x+\alpha^*p) -\rho(x) \geq  \frac{1}{\min\{\|x\|^2_B,\|x+\alpha^* p\|^2_B\}}\int_{0}^{\alpha^*}(a\alpha^2+b\alpha+c) d\alpha \\ \nonumber
&=& \frac{-\frac{c^2}{2b}-\frac{ac^3}{3b^3}+\mathcal{O}(c^4)}{\min\{\|x\|^2_B,\|x+\alpha^* p\|^2_B\}} \geq -\frac{c^2(1+\eta)}{\min\{\|x\|^2_B,\|x+\alpha^*p\|^2_B\}},
\end{eqnarray}
where $c = \|x\|_B^2 p^T(Ax-\rho(x)Bx) = \mathcal{O}(\sin\theta)\mathcal{O}\big(\cos \angle(p,\nabla \rho(x))\big)$ from \eqref{coefabc}. \qed
\end{proof}\medskip

 {\bf Proof of Lemma \ref{xjxk_lemma1}}
\begin{proof}
First, note that one can always choose $\beta_{jk}, \gamma_{jk}>0$ in \eqref{xjxkeqn}. If a decomposition has negative $\beta_{jk}$ or $\gamma_{jk}$, simply replace $x_j$ with $-x_j$ or $g_{jk}$ with $-g_{jk}$.

Given the decomposition \eqref{xjxkeqn}, the normalization condition $\|x_k\|_B=1$ leads to 
\begin{eqnarray}\nonumber
\hspace{-0.3in}&&\|x_k\|^2_B = (\beta_{jk}x_j+\gamma_{jk}g_{jk})^TB(\beta_{jk}x_j+\gamma_{jk}g_{jk}) \\ \label{unitxk}
\hspace{-0.3in}&=& \|x_j\|^2_B \beta_{jk}^2+2(g_{jk}^TBx_j)\gamma_{jk}\beta_{jk}+\|g_{jk}\|^2_B \gamma^2_{jk} = \beta_{jk}^2+2\mu_{jk} \gamma_{jk}\beta_{jk}+\gamma^2_{jk}\\ \nonumber
\hspace{-0.3in} &=&  (\gamma_{jk}+\mu_{jk}\beta_{jk})^2+(1-\mu_{jk}^2)\beta_{jk}^2 = 1,
\end{eqnarray}
where $\mu_{jk} = g_{jk}^TBx_j = (x_j,g_{jk})_B \in [-1,1]$ because $\|x_j\|_B=\|g_{jk}\|_B=1$. 

Let $r_j = Ax_j - \rho(x_j)Bx_j$ be the eigenresidual associated with $x_j$, and define $d_{jk} \equiv \rho(g_{jk})-\rho(x_j)$. The Rayleigh quotient of $x_k$ is therefore
\begin{eqnarray}\label{rhoxkopt}
&&\rho(x_k) = \frac{x_k^TAx_k}{x_k^TBx_k}=x_k^TAx_k=(\beta_{jk}x_j+\gamma_{jk}g_{jk})^TA(\beta_{jk}x_j+\gamma_{jk}g_{jk}) \\ \nonumber
&=& (x_j^TAx_j)\beta^2_{jk}+2(g_{jk}^TAx_j)\gamma_{jk}\beta_{jk}+(g_{jk}^TAg_{jk})\gamma^2_{jk} \\ \nonumber
&=& \rho(x_j)\beta^2_{jk}+2[g_{jk}^T(r_j+\rho(x_j)Bx_j)]\gamma_{jk}\beta_{jk}+\rho(g_{jk})\gamma^2_{jk} \quad (\mbox{as }Ax_j = r_j+\rho(x_j)Bx_j)\\ \nonumber
&=& \rho(x_j)\beta^2_{jk}+2\rho(x_j)(g_{jk}^T Bx_j)\gamma_{jk}\beta_{jk} + 2(g_{jk}^Tr_j)\gamma_{jk}\beta_{jk} + (\rho(x_j)+d_{jk})\gamma^2_{jk} \\ \nonumber
&=& \rho(x_j)\left(\beta^2_{jk}+2\mu_{jk}\gamma_{jk}\beta_{jk}+\gamma^2_{jk}\right)+2(g_{jk}^Tr_j)\gamma_{jk}\beta_{jk}+d_{jk}\gamma^2_{jk} \\ \nonumber
&=& \rho(x_j)+2(g_{jk}^Tr_j)\gamma_{jk}\beta_{jk}+d_{jk}\gamma^2_{jk}, \qquad\qquad\quad\, \mbox{ (see \eqref{unitxk})}
\end{eqnarray}
or, equivalently, $$2(g_{jk}^Tr_j)\gamma_{jk}\beta_{jk}+d_{jk}\gamma^2_{jk} = \rho(x_k)-\rho(x_j)=\left(\rho(x_k)-\lambda_1\right)-\left(\rho(x_j)-\lambda_1\right).$$ Using Assumption \ref{assumptionlincvg}, it follows that
$$- (1-\underline{\xi}^{k-j})(\rho(x_j)-\lambda_1) \leq 2(g_{jk}^Tr_j)\gamma_{jk}\beta_{jk}+d_{jk}\gamma^2_{jk} \leq - (1-\bar{\xi}^{k-j})(\rho(x_j)-\lambda_1).$$ From \eqref{rqerror}, we divide the above inequality by $\sin^2\theta_j$ and have 
{\footnotesize
\begin{eqnarray}\label{doublebd}
\hspace{-0.3in}- (1-\underline{\xi}^{k-j})(\rho(f_j)-\lambda_1) \leq 2\frac{g_{jk}^Tr_j}{\sin\theta_j}\left(\frac{\gamma_{jk}}{\sin\theta_j}\right)\beta_{jk}+d_{jk}\left(\frac{\gamma_{jk}}{\sin\theta_j}\right)^2  \leq - (1-\bar{\xi}^{k-j})(\rho(f_j)-\lambda_1).
\end{eqnarray}
}Note that the left-hand side and the right-hand side of the above inequality are both on the order of $\mathcal{O}(1)$, bounded away from zero and independent of $\theta_j$. Also note that $\frac{|g_{jk}^Tr_j|}{\sin\theta_j}\leq \frac{\|g_{jk}\|\|r_j\|}{\sin\theta_j} \leq (1+\delta)\|g_{jk}\|\|Af_j-\rho(x_j)Bf_j\| = \mathcal{O}(1)$ from \eqref{eigresbd}. In addition, since $g_{jk} \in \mathrm{span}\{p_{j-1},g_j,\ldots,g_{k-1}\} \subset W_k$, under Assumption \ref{anglev1Wk}, there is a $d>0$ such that $d_{jk} = \rho(g_{jk})-\rho(x_j) \geq d > 0$ for all $1\leq j < k$. Consequently, letting $\theta_j \rightarrow 0$ in \eqref{doublebd}, we have $\gamma_{jk} = \mathcal{O}(\sin\theta_j)$\footnote{We emphasize here that $\gamma_{jk} \ne o(\sin\theta_j)$ and hence $\gamma_{jk}$ can be safely replaced with $\mathcal{O}(\sin\theta_j)$ anywhere appropriate.}.

From \eqref{unitxk}, $\beta_{jk}$ can be written in terms of $\gamma_{jk}$ as follows
{\small
\begin{eqnarray}\label{ajk}
\beta_{jk} = -\mu_{jk}\gamma_{jk}+\sqrt{1-\gamma^2_{jk}+\mu^2_{jk}\gamma^2_{jk}}=1-\mu_{jk}\gamma_{jk}-\frac{1}{2}(1-\mu^2_{jk})\gamma^2_{jk}+\mathcal{O}(\gamma^4_{jk}),
\end{eqnarray}
}which completes the proof. \qed
\end{proof}\medskip

{\bf Proof of Theorem \ref{conjoptext}}
\begin{proof}
Since $y = \beta_y x_0 + \gamma_y g_y$, $g_y \in W_k$, $\|y\|_B=\|g_y\|_B=1$ and $\rho(y) \leq \rho(x_0)$, from Lemma \ref{xjxk_lemma1}, we have $\gamma_y = \mathcal{O}(\sin\theta_0)$ and $|1-\beta_y|=\mathcal{O}(\sin\theta_0)$. Similarly, the iterate $x_k = \beta_{0k}x_0+\gamma_{0k}g_{0k}$ satisfies $\gamma_{0k} = \mathcal{O}(\sin\theta_0)$ and $|1-\beta_{0k}|=\mathcal{O}(\sin\theta_0)$. Note that $\|r_0\| = \mathcal{O}(\sin\theta_0)$, $0 \leq \rho(x_0)-\rho(y) \leq \rho(x_0)-\lambda_1 = \mathcal{O}(\sin^2\theta_0)$ and, similarly, $0 \leq \rho(x_0)-\rho(x_k) \leq \mathcal{O}(\sin^2\theta_0)$. It follows that 
\begin{eqnarray}\label{resdiff}
&&\qquad (A-\rho(y)B)y-(A-\rho(x_k)B)x_k \\ \nonumber
&=& (A-\rho(x_0)B)(\beta_y x_0+\gamma_y g_y)-(A-\rho(x_0)B)(\beta_{0k}x_0 + \gamma_{0k}g_{0k}) \\ \nonumber
&& \qquad + \left(\rho(x_0)-\rho(y)\right)By - \left(\rho(x_0)-\rho(x_k)\right)Bx_k \\ \nonumber
&=& (\beta_y-\beta_{0k})(A-\rho(x_0)B)x_0 + (A-\rho(x_0)B)(\gamma_y g_y - \gamma_{0k}g_{0k})+\mathcal{O}(\sin^2\theta_0) \\ \nonumber
&=& \left[(\beta_y-1)+(1-\beta_{0k})\right]r_0+(A-\rho(x_k)B)(\gamma_yg_y-\gamma_{0k}g_{0k}) \\ \nonumber
&& \qquad +\left(\rho(x_k)-\rho(x_0)\right)B(\gamma_y g_y - \gamma_{0k}g_{0k})+\mathcal{O}(\sin^2\theta_0) \\ \nonumber
&=& r_0\mathcal{O}(\sin\theta_0) + \gamma_y(A-\rho(x_k)B)g_y - \gamma_{0k}(A-\rho(x_k)B)g_{0k}+\mathcal{O}(\sin^3\theta_0)+\mathcal{O}(\sin^2\theta_0) \\ \nonumber
&=& \gamma_y (A-\rho(x_k)B)g_y -\gamma_{0k}(A-\rho(x_k)B)g_{0k}+\mathcal{O}(\sin^2\theta_0).
\end{eqnarray}

By assumption, $g_y,g_{0k}\in W_k = \mathrm{span}\{p_0,\ldots,p_{k-1}\}$ are approximately conjugate to $p_k$, such that $p_k^T(A-\rho(x_k)B)g_y = \mathcal{O}(\sin\theta_0)$ and $p_k^T(A-\rho(x_k)B)g_{0k}=\mathcal{O}(\sin\theta_0)$, since $\sin\theta_{\ell}=\mathcal{O}(\sin\theta_0)$ for $1 \leq \ell \leq k-1$, due to \eqref{assumptionlincvg_eqv}. It follows from \eqref{resdiff} that
\begin{eqnarray}\nonumber
&&p_k^T\left[(A-\rho(y)B)y-(A-\rho(x_k)B)x_k\right] \\ \nonumber
&=&\gamma_yp_k^T(A-\rho(x_k)B)g_y - \gamma_{0k}p_k^T(A-\rho(x_k)B)g_{0k}+\mathcal{O}(\sin^2\theta_0) \\ \nonumber
&=& \gamma_y \mathcal{O}(\sin\theta_0)  -\gamma_{0k}\mathcal{O}(\sin\theta_0)+\mathcal{O}(\sin^2\theta_0) = \mathcal{O}(\sin\theta^2_0),
\end{eqnarray}
which is a crucial observation for the rest of the proof, or equivalently, 
\begin{eqnarray}\label{pktres}
p_k^T(A-\rho(y)B)y = p_k^T(A-\rho(x_k)B)x_k + \mathcal{O}(\sin^2\theta_0).
\end{eqnarray}In addition, note that $x_k - y = (\beta_{0k}-\beta_y)x_0 + (\gamma_{0k}g_{0k}-\gamma_y g_y)=x_0\mathcal{O}(\sin\theta_0)+g_{0k}\mathcal{O}(\sin\theta_0)-g_y\mathcal{O}(\sin\theta_0)=\mathcal{O}(\sin\theta_0)$, and hence $By=Bx_k+\mathcal{O}(\sin\theta_0)$. 

Let $R_3(\rho(x),\alpha p)= \frac{1}{2}\rho(x+\alpha p)-\left[\frac{1}{2}\rho(x)+\alpha p^T\nabla \frac{1}{2}\rho(x)+\frac{1}{2}\alpha^2p^T\nabla^2\frac{1}{2}\rho(x) p\right]$ be the remainder of 2nd order Taylor expansion of $\frac{1}{2}\rho(x+\alpha p)$ at $x$ with $\|x\|_B=1$. Then,
{\footnotesize
\begin{eqnarray}\label{taylorexprhoz}
\hspace{-1in}&&\qquad\quad \frac{1}{2}\rho(z) = \frac{1}{2}\rho(y+\alpha p_k)=\frac{1}{2}\rho(y)+\alpha p_k^T\frac{1}{2}\nabla \rho(y)+\frac{1}{2}\alpha^2p_k^T\frac{1}{2}\nabla^2\rho(y)p_k + R_3(\rho(y),\alpha p_k) \\ \nonumber
&=&\frac{1}{2}\rho(y)+\alpha p_k^T(A-\rho(y)B)y +\frac{1}{2}\alpha^2 p_k^T\big[(A-\rho(y)B)-2(Ay-\rho(y)By)(By)^T \\ \nonumber
&& \qquad\qquad \qquad\qquad\qquad\qquad\qquad\qquad
-2By(Ay-\rho(y)By)^T\big]p_k +R_3(\rho(y),\alpha p_k) \\ \nonumber
&=& \frac{1}{2}\rho(y)+\alpha \left[p_k^T(A-\rho(x_k)B)x_k+\mathcal{O}(\sin^2\theta_0)\right] +  \frac{1}{2}\alpha^2\Big\{p_k^T\left[(A-\rho(x_k)B)+(\rho(x_k)-\rho(y))B\right]p_k \\ \nonumber
&& \qquad\qquad -2\left[p_k^T(A-\rho(x_k)B)x_k+\mathcal{O}(\sin^2\theta_0)\right][x_k^TB+\mathcal{O}(\sin\theta_0)]p_k \\ \nonumber
&& \qquad\qquad -2p_k^T[Bx_k+\mathcal{O}(\sin\theta_0)]\left[x_k^T(A-\rho(x_k)B)p_k+\mathcal{O}(\sin^2\theta_0)\right]\!\Big\} + R_3(\rho(y),\alpha p_k)  \:\:\mbox{ (see \eqref{pktres})}\\ \nonumber
& = &  \frac{1}{2}\rho(y) + \alpha p_k^T\frac{1}{2}\nabla \rho(x_k)+\alpha \mathcal{O}(\sin^2\theta_0) + \frac{1}{2}\alpha^2 p_k^T\Big\{(A-\rho(x_k)B)-2r_k(Bx_k)^T-2(Bx_k)r_k^T\Big\}p_k  \\ \nonumber
&& \qquad\qquad + \alpha^2\left[\mathcal{O}(\sin^2\theta_0)+\|r_k\|\mathcal{O}(\sin\theta_0)\right] + R_3(\rho(y),\alpha p_k) \\ \nonumber
&=& \frac{1}{2}\rho(y)+\alpha p_k^T\frac{1}{2}\nabla \rho(x_k)+\frac{1}{2}\alpha^2 p_k^T\frac{1}{2}\nabla^2\rho(x_k)p_k +\alpha \mathcal{O}(\sin^2\theta_0)+ \alpha^2\mathcal{O}(\sin^2\theta_0)+R_3(\rho(y),\alpha p_k)\\ \nonumber
&=& \frac{1}{2}\rho(x_k)+\alpha p_k^T\frac{1}{2}\nabla \rho(x_k)+\frac{1}{2}\alpha^2 p_k^T\frac{1}{2}\nabla^2 \rho(x_k) p_k +R_3(\rho(x_k),\alpha p_k) \\ \nonumber
&& \qquad \qquad +\frac{1}{2}\left(\rho(y)-\rho(x_k)\right)-R_3(\rho(x_k),\alpha p_k)+\mathcal{O}(\alpha \sin^2\theta_0)+\mathcal{O}(\alpha^2 \sin^2\theta_0)+R_3(\rho(y),\alpha p_k) \\ \nonumber
&=& \frac{1}{2}\rho(x_k+\alpha p_k)+\frac{1}{2}\left(\rho(y)-\rho(x_k)\right)+R_3(\rho(y),\alpha p_k)-R_3(\rho(x_k),\alpha p_k) + \mathcal{O}(\alpha \sin^2\theta_0)+\mathcal{O}(\alpha^2\sin^2\theta_0).
\end{eqnarray}
}

Let the global minimizer in $U_{k+1}$ be $z^*=y(z^*)+\alpha(z^*)p_k$, with $y(z^*) \in U_k$ and $\|y(z^*)\|_B=1$. We note that $y(z^*)$ is generally \emph{not} the global minimizer in $U_k$. Consider the decomposition $y(z^*)=v_1 \cos\theta_{y(z^*)}+f_{y(z^*)}\sin\theta_{y(z^*)}$ with $f_{y(z^*)} \perp Bv_1$ and $\|f_{y(z^*)}\|_B=1$, such that $\rho(y(z^*))-\lambda_1 = \mathcal{O}(\sin^2\theta_{y(z^*)})$. Here, $\alpha(z^*)$ is the optimal step size moving from $y(z^*)$ in the direction of $p_k$, due to the global optimality of $z^*$ in $U_{k+1}$. It follows from Lemma \ref{rhodecrease} that $$\rho(z^*)-\rho(y(z^*)) \leq -\mathcal{O}(\sin^2\theta_{y(z^*)})\mathcal{O}\left(\cos^2 \angle (p_k,\nabla \rho(y(z^*)))\right),$$ where the coefficient of the $\sin^2\theta_{y(z^*)}$ term depends on the quantities $b$ and $c$ defined in \eqref{coefabc} involving the search direction $p_k$. On the other hand, as $z^*$ is the global minimizer in $U_{k+1}$, we have 
\begin{eqnarray}\label{rhozstarerr1}
\rho(z^*)-\lambda_1 \leq \rho(x_{k+1})-\lambda_1 =\mathcal{O}(\sin^2\theta_{k+1}).
\end{eqnarray}Meanwhile, by Lemma \ref{rhodecrease}, we also have
\begin{eqnarray}\label{rhozstarerr2}
&&\quad \rho(z^*)-\lambda_1 = \left[\rho(z^*)-\rho(y(z^*))\right]+\left[\rho(y(z^*))-\lambda_1\right]  \\ \nonumber
&=& -\mathcal{O}(\sin^2\theta_{y(z^*)})\mathcal{O}\left(\cos^2 \angle (p_k,\nabla \rho(y(z^*)))\right)+\mathcal{O}(\sin^2\theta_{y(z^*)}) = \mathcal{O}(\sin^2\theta_{y(z^*)}).
\end{eqnarray}
It follows that
\begin{eqnarray}\label{twothetas}
\sin\theta_{y(z^*)} \leq \mathcal{O}(\sin\theta_{k+1}).
\end{eqnarray} 

Given $z^*=y(z^*)+\alpha(z^*)p_k$, it follows from the triangle inequality that 
\begin{eqnarray}\label{triangleineq}
%1-|\alpha(z^*)|=\|y(z^*)\|_B - |\alpha(z^*)|\|p_k\|_B \leq 
\|z^*\|_B \leq \|y(z^*)\|_B+|\alpha(z^*)|\|p_k\|_B = 1+|\alpha(z^*)|.
\end{eqnarray}
Meanwhile, the $B$-normalized $z^*$ is $\frac{z^*}{\|z^*\|_B} = \frac{1}{\|z^*\|_B}y(z^*)+\frac{\alpha(z^*)}{\|z^*\|_B}p_k$, and we can follow the proof of Lemma \ref{xjxk_lemma1} to show that $\frac{|\alpha(z^*)|}{\|z^*\|_B}=\mathcal{O}(\sin\theta_{y(z^*)}) \leq \mathcal{O}(\sin\theta_{k+1})$ due to \eqref{twothetas}. From \eqref{triangleineq}, $\frac{|\alpha(z^*)|}{1+|\alpha(z^*)|}\leq \frac{|\alpha(z^*)|}{\|z^*\|_B} \leq \mathcal{O}(\sin\theta_{k+1})$, i.e., $|\alpha(z^*)|\leq \mathcal{O}(\sin\theta_{k+1})$. 

Let $\alpha_k^*$ be the minimizer of $\rho(x_k+\alpha p_k)$ and $y^*$ be the global minimizer in $U_k$ (note that $\alpha_k^* \ne \alpha(z^*)$ and $y^* \ne y(z^*)$), which contains all vectors of the form $y=\beta_y x_0 + \gamma_y g_y$ with $\|y\|_B=\|g_y\|_B=1$, $g_y \in W_k$ and $\rho(y) \leq \rho(x_0)$. Let $z = z^*=y(z^*)+\alpha(z^*)p_k$, $y = y(z^*)$ and $\alpha = \alpha(z^*)$ in \eqref{taylorexprhoz}, and note that $\rho(x_k+\alpha(z^*)p_k) \geq \rho(x_k+\alpha^*p_k)$ and $\rho(y(z^*)) \geq \rho(y^*)$. Therefore, we have
{\small
\begin{eqnarray}
&&\frac{1}{2}\rho(z^*)=\frac{1}{2}\rho(x_k+\alpha(z^*)p_k)+\frac{1}{2}\left(\rho(y(z^*))-\rho(x_k)\right)+R_3\left(\rho(y(z^*)),\alpha(z^*)p_k\right) \\ \nonumber
&& \qquad -R_3\left(\rho(x_k),\alpha(z^*)p_k\right)+\mathcal{O}(\alpha(z^*)\sin^2\theta_0)+\mathcal{O}(\alpha(z^*)^2\sin^2\theta_0) \\ \nonumber
&\geq & \frac{1}{2}\rho(x_k+\alpha_k^* p_k)+\frac{1}{2}\left(\rho(y^*)-\rho(x_k)\right)+R_3\left(\rho(y(z^*)),\alpha(z^*)p_k\right) \\ \nonumber
&& \qquad -R_3\left(\rho(x_k),\alpha(z^*)p_k\right)+\mathcal{O}(\alpha(z^*)\sin^2\theta_0)+\mathcal{O}(\alpha(z^*)^2\sin^2\theta_0) \\ \nonumber
& \geq & \frac{1}{2}\rho(x_{k+1})+\frac{1}{2}\left(\rho(y^*)-\rho(x_k)\right)+\mathcal{O}(\sin\theta_{k+1})\mathcal{O}(\sin^2\theta_0). \quad (|\alpha(z^*)|\leq \mathcal{O}(\sin\theta_{k+1}))
\end{eqnarray}
}In the last step above, $\rho(x_k+\alpha^*p_k) \geq \rho(x_{k+1})$ because $x_{k+1}$ is the minimizer over $\mathrm{span}\{x_k,g_k,p_{k-1}\}$, whereas $p_k \in \mathrm{span}\{g_k,p_{k-1}\}$. Equivalently, $$\rho(x_{k+1})-\rho(z^*) \leq \rho(x_k)-\rho(y^*)+\mathcal{O}(\sin\theta_{k+1})\mathcal{O}(\sin^2\theta_0).\qed$$
\end{proof}\medskip

{\bf Proof of Lemma \ref{apprxconjkm1k}}
\begin{proof}Recall that $\widetilde{p}_k = g_k - \frac{p_{k-1}^T(A-\rho(x_{k-1})B)g_k}{p_{k-1}^T(A-\rho(x_{k-1})B)p_{k-1}}p_{k-1}$. Therefore
{\small
\begin{eqnarray}\nonumber
&&p_{k-1}^T(A-\rho(x_{k-1})B)\widetilde{p}_k \\ \nonumber
&=& \textstyle p_{k-1}^T(A-\rho(x_{k-1})B)g_k - \frac{p_{k-1}^T(A-\rho(x_{k-1})B)g_k}{p_{k-1}^T(A-\rho(x_{k-1})B)p_{k-1}}p_{k-1}^T(A-\rho(x_{k-1})B)p_{k-1}=0,
\end{eqnarray}
}and hence $p_{k-1}^T(A-\rho(x_{k-1})B)p_k = 0$. Note that $\rho(x_{k-1})-\rho(x_k) = \left(\rho(x_{k-1})-\lambda_1\right)-\left(\rho(x_k)-\lambda_1\right) \leq (1-\bar{\xi})\left(\rho(x_{k-1})-\lambda_1\right) = \mathcal{O}(\sin^2\theta_{k-1})$, and it follows that
\begin{eqnarray}\nonumber
&&p_{k-1}^T(A-\rho(x_k)B)p_k = p_{k-1}^T(A-\rho(x_{k-1})B)p_k + \left(\rho(x_{k-1})-\rho(x_k)\right)p_{k-1}^TBp_k \\ \nonumber
&=& \left(\rho(x_{k-1})-\rho(x_k)\right)(p_{k-1},p_k)_B= \mathcal{O}(\sin^2\theta_{k-1}).\qed
\end{eqnarray}
\end{proof}\medskip

{\bf Proof of Lemma \ref{apprxconj02}}
\begin{proof}
At step 2 of Algorithm~\ref{alg:pl+1}, $x_2$ is extracted from $\mathrm{span}\{x_1,g_1,p_0\}$, where $p_0 = g_0=Md_0$ up to a scaling factor. By the local optimality, 
\begin{eqnarray}\label{r2r1r0}
r_2 \perp g_0\: \mbox{ and }\: r_2 \perp g_1.
\end{eqnarray}
At step~3, we form $\mathrm{span}\{x_2,g_2,p_1\}$ to extract $x_3$ and $p_2$. At step 7 of Algorithm~\ref{alg:pl+1}, $\widetilde{p}_2 = g_2 - \frac{p_1^T(A-\rho(x_{1})B)g_2}{p_1^T(A-\rho(x_1)B)p_1}p_1$, which shows that $\|\widetilde{p}_2\|$ is proportional to $\|g_2\|$ and independent of the scaling of $p_1$. Hence, the normalized search direction $p_2$ can be written as $p_2 = \frac{\eta_2}{\|g_2\|}\widetilde{p}_2$, where $\eta_2$ is chosen such that $\|p_2\|_B = 1$. From Lemma \ref{xjxk_lemma1}, we have $x_1 = \beta_{01}x_0+\gamma_{01}p_0$, i.e., $p_0 = \frac{1}{\gamma_{01}}(-\beta_{01}x_0+x_1)$, where $\gamma_{01} = \mathcal{O}(\sin\theta_0)$. Therefore
{\small
\begin{eqnarray}\nonumber
&&p_0^T(A-\rho(x_0)B)g_2=\frac{1}{\gamma_{01}}\left[-\beta_{01}x_0^T(A-\rho(x_0)B)+x_1^T(A-\rho(x_1)B)+\left(\rho(x_1)-\rho(x_0)\right)x_1^TB\right]g_2 \\ \nonumber
&&= \frac{\rho(x_1)-\rho(x_0)}{\gamma_{01}}x_1^TBg_2+\frac{1}{\gamma_{01}}\left(-\beta_{01}r_0^T+r_1^T\right)g_2 =  \frac{(\rho(x_1)-\lambda_1)-(\rho(x_0)-\lambda_1)}{\gamma_{01}}x_1^TBg_2+ \\ \nonumber
&& \quad \frac{1}{\gamma_{01}}\left(-\beta_{01}r_0^TMp^{(2)}_{m-1}\left((A-\rho_2 B)M\right)r_2+r_1^TMp^{(2)}_{m-1}\left((A-\rho_2 B)M\right)r_2\right)  \\ \nonumber
&&= \frac{\mathcal{O}(\sin^2\theta_0)}{\mathcal{O}(\sin\theta_0)}x_1^TBg_2 + \frac{1}{\gamma_{01}}\Big\{-\beta_{01}r_0^T\left(p^{(0)}_{m-1}\left(M(A-\rho_0 B)\right)M+\mathcal{O}(\sin\theta_0)\right)r_2 +\\ \nonumber
&& \qquad \qquad \qquad \qquad \qquad \quad r_1^T\left(p^{(1)}_{m-1}\left(M(A-\rho_1 B)\right)M+\mathcal{O}(\sin\theta_1)\right)r_2\Big\}.
\end{eqnarray}
}Since $M$, $A$ and $B$ are all real symmetric, so is $p^{(k)}_{m-1}\left(M(A-\rho_k B)\right)M$. Therefore, $r_k^Tp^{(k)}_{m-1}\left(M(A-\rho_k B)\right)M = g_k^T$, and hence
{\small
\begin{eqnarray}\nonumber
&&\quad p_0^T(A-\rho(x_0)B)g_2  \\ \nonumber
&=& x_1^TBg_2\mathcal{O}(\sin\theta_0) + \frac{1}{\gamma_{01}}\Big\{-\beta_{01}g_0^T r_2+r_0^T r_2\mathcal{O}(\sin\theta_0)+g_1^T r_2 +r_1^T r_2 \mathcal{O}(\sin\theta_1)\Big\} \\ \nonumber
&=& \|g_2\|\mathcal{O}(\sin\theta_0)+\frac{1}{\mathcal{O}(\sin\theta_0)}\big\{r_0^T r_2\mathcal{O}(\sin\theta_0)+r_1^T r_2\mathcal{O}(\sin\theta_1)\big\} \\ \nonumber
&=& \|g_2\|\mathcal{O}(\sin\theta_0)+\|r_2\|\mathcal{O}(\sin\theta_0), \qquad \mbox{(note that $\|r_1\|=\mathcal{O}(\sin\theta_1)=\mathcal{O}(\sin\theta_0)$)}
\end{eqnarray}
}where we use $r_2 \perp g_0$ and $r_2 \perp g_1$ from \eqref{r2r1r0}. This shows that $p_0^T(A-\rho(x_0)B)g_2 \leq \mathcal{O}(\sin\theta_0)$. Also, since $p_0^T(A-\rho(x_0)B)g_2$ is proportional to $\|g_2\|$, we  simply have $p_0^T(A-\rho(x_0)B)g_2 = \|g_2\|\mathcal{O}(\sin\theta_0)$. Therefore, by Lemma \ref{apprxconjkm1k},
{\small
\begin{eqnarray}\nonumber
&& p_0^T(A-\rho(x_0)B){p}_2 = p_0^T(A-\rho(x_0)B)\frac{\eta_2}{\|g_2\|}\widetilde{p}_2\\ \nonumber
&=& p_0^T(A-\rho(x_0)B)\frac{\eta_2}{\|g_2\|}\left\{g_2 - \frac{p_1^T(A-\rho(x_1)B)g_2}{p_1^T(A-\rho(x_1)B)p_1}p_1 \right\} \\ \nonumber
&=& \frac{\eta_2}{\|g_2\|} \left\{p_0^T(A-\rho(x_0)B)g_2 - \frac{p_1^T(A-\rho(x_1)B)g_2}{p_1^T(A-\rho(x_1)B)p_1}\left[p_0^T(A-\rho(x_0)B)p_1\right]\right\} \\ \nonumber
&=& \frac{\eta_2}{\|g_2\|}\left(\|g_2\|\mathcal{O}(\sin\theta_0)+\|g_2\|\mathcal{O}(\sin^2\theta_0)\right)=\mathcal{O}(\sin\theta_0).
\end{eqnarray}
}
Similarly, using a slightly different shift in the above relation, we have
\begin{eqnarray}\label{p2p0}
&&p_2^T(A-\rho(x_2)B)p_0 = p_0^T(A-\rho(x_0)B)p_2+\left(\rho(x_0)-\rho(x_2)\right)p_0^TBp_2 \\ \nonumber
&=& \mathcal{O}(\sin\theta_0)+\left(\left(\rho(x_0)-\lambda_1\right)-\left(\rho(x_2)-\lambda_1\right)\right)( p_0, p_2 )_B \\ \nonumber
&=& \mathcal{O}(\sin\theta_0) + \left(\mathcal{O}(\sin^2\theta_0)-\mathcal{O}(\sin^2\theta_2)\right)( p_0, p_2 )_B=\mathcal{O}(\sin\theta_0).\qed
\end{eqnarray}
\end{proof}\medskip

{\bf Proof of Lemma \ref{qopteqvorth}}
\begin{proof}
The proof is done by contradiction. Assume that $\lim_{\theta_0 \rightarrow 0} \cos\angle(r_k, W_k)$ exists but is not zero. Then there is a vector $u \in U_k$ such that $\lim_{\theta_0 \rightarrow 0}\cos \angle(r_k,u) = \delta < 0$, i.e., $u$ is a descent direction at $x_k$. We have shown in Lemma \ref{rhodecrease} that $$\rho(x_k+\alpha^* u) - \rho(x_k) \leq - \mathcal{O}(\sin^2\theta_k)\mathcal{O}\left(\cos^2\angle(u,r_k)\right) = -\mathcal{O}(\sin^2\theta_k).$$ Since $y^*$ is the global minimizer in $U_k$ and $x_k+\alpha^* u \in U_k$, $\rho(y^*) \leq \rho(x_k+\alpha^* u) \leq \rho(x_k)$. As a result, we have $$\lim_{\theta_0 \rightarrow 0} \frac{\rho(x_k)-\rho(y^*)}{\rho(x_k)-\lambda_1} \geq \lim_{\theta_0 \rightarrow 0} \frac{\rho(x_k)-\rho(x_k+\alpha^* u)}{\rho(x_k)-\lambda_1} \geq \lim_{\theta_0 \rightarrow 0}\frac{\mathcal{O}(\sin^2\theta_k)}{\mathcal{O}(\sin^2\theta_k)} = \mathcal{O}(1),$$ meaning that $x_k$ does not satisfy the global quasi-optimality condition. \qed
\end{proof}

\end{document}